\documentclass{amsart}
\usepackage{palatino} 
\usepackage{amssymb,amsmath,url,graphicx}
\usepackage{amscd} 
\usepackage[all]{xy}
\usepackage{fullpage}
\usepackage{enumitem}
\usepackage{graphicx}
\usepackage[hidelinks]{hyperref}
\usepackage{tikz}

\newif\ifdebug                                                      %
\debugtrue

\newif\ifinfootnote
\infootnotefalse

\let\footnoteasusual\footnote
\renewcommand{\footnote}[1]
{\infootnotetrue\footnoteasusual{#1}\infootnotefalse}

\newcommand{\printname}[1]
{\ifmmode{ \smash{ \raisebox{5pt}{\text{\tiny{#1}}} } }
 \else   {\ifinfootnote \smash{\raisebox{0pt}{\tiny{#1}}}
             \else { \marginpar{
                     \smash{ \makebox[0pt]{\raisebox{-12pt}{\tiny{#1}}} }
                               } } \fi} \fi}

\swapnumbers
\numberwithin{equation}{section}
\newtheorem {Theorem}[equation]         {Theorem}

\newtheorem {Lemma}[equation]           {Lemma}

\newtheorem {Claim*}                    {Claim}

\newtheorem {Corollary} [equation]      {Corollary}
\newtheorem {Proposition}  [equation]   {Proposition}

\theoremstyle{definition}
\newtheorem{Definition}[equation]{Definition}

\theoremstyle{remark}
\newtheorem{Remark}[equation]{Remark}
\newtheorem*{Remark*}{Remark}
\newtheorem{Example}[equation]{Example}

\setlist{topsep=0pt,itemsep=6pt}

\def \Z {{\mathbb Z}}
\def \R {{\mathbb R}}
\def \C {{\mathbb C}}
\def \Q {{\mathbb Q}}

\def\fln{\mbox{Fl}(n)}
\def\fl{\mbox{Fl}}
\newcommand{\Flags}{Flag}
\def\Tn{T}
\def\Y{Pet_n} 
\def\HT{H_{\Tn}^{*}} 
\def\HS{H_{\S1}^{*}} 
\def\S1{{\mathsf{S}}}

\def\f{\theta}

\def\xi{p}
\def\bxi{\check{p}}

\newcommand{\diag}{\mathrm{diag}}

\newcommand{\gl}{\mathfrak{gl}}
\DeclareMathOperator{\Hess}{Hess}

\def\cf{\check f}
\def\Ih{I_h}
\DeclareMathOperator{\pt}{\text{pt}}
\newcommand{\TChFlag}{\tau^T}

\newcommand{\ChFlag}{\tau}
\newcommand{\SChNil}{\bar\tau^\mathsf{S}}

\newcommand{\into}{{\hookrightarrow}}

\def\varphih{\varphi_h}
\DeclareMathOperator{\Poin}{Poin}
\DeclareMathOperator{\Hilb}{Hilb}

\title[Title]
{ The cohomology rings of regular nilpotent Hessenberg varieties }

\author{Megumi Harada}
\address{Department of Mathematics and Statistics, McMaster University, 
1280 Main Street West, 
Hamilton, Ontario L8S 4K1, Canada}
\email{haradam@mcmaster.ca}

\author{Tatsuya Horiguchi} 
\address{National Institute of Technology, Ube College, 2-14-1, Tokiwadai, Ube, Yamaguchi, Japan 755-8555} 
\email{tatsuya.horiguchi@gmail.com} 

\date{\today}

\begin{document}

\maketitle

\begin{abstract} 
This manuscript is a contributed chapter in the forthcoming CRC Press volume, titled the \emph{Handbook of Combinatorial Algebraic Geometry: Subvarieties of the Flag Variety}.  The book, as a whole, is aimed at a diverse audience of researchers and graduate students seeking an expository introduction to the area. In our chapter, we give an overview of some of the past research on the cohomology rings of regular nilpotent Hessenberg varieties, with no claim to being exhaustive. For the purposes of this manuscript, we focus mainly on the case of Lie type A, with some brief remarks on the general Lie types. We end the chapter with a selection of topics currently active in this area.  
\end{abstract}

List of some notation in this manuscript:
\begin{itemize}
\item $[n]=\{1,2,\ldots,n \}$
\item $\mathsf{N}$: the regular nilpotent element in Jordan canonical form
\item $\Hess(\mathsf{N},h)$: the regular nilpotent Hessenberg variety
\item $\Y$: the Peterson variety in type $A_{n-1}$
\item $\S1$: the one-dimensional torus acting on $\Hess(\mathsf{N},h)$ 
\item $\Hilb(R,q)$: the Hilbert series of a graded commutative algebra $R$ in the variable $q$
\item $\Poin(X,q)$: the Poincar\'e polynomial of a topological space $X$ in the variable $q$
\end{itemize}

\maketitle

\setcounter{tocdepth}{1}

\tableofcontents

\section{Introduction  }

\label{section: intro}

This manuscript gives an account of the study of the cohomology rings\footnote{In this manuscript, all cohomology rings are with coefficients in $\Q$ unless otherwise specified.}  of Hessenberg varieties. Several comments are in order before we begin the exposition in earnest. 

Firstly, in analogy to the situation of the flag variety $\fln \cong GL_n(\C)/B$ and the torus action on $\fln$ coming from the maximal torus $T$ of $GL_n(\C)$, one of the crucial geometric features of Hessenberg varieties is that they admit torus actions.  (Warning: the dimension of the torus that acts can vary, depending on the choice of parameters that define the Hessenberg variety.) These actions allow us to bring to bear all the techniques of \emph{(torus)-equivariant} topology and \emph{(torus)-equivariant} cohomology to the study of these varieties.  Moreover, since these torus actions arise as restrictions of the maximal torus $T$-action on $GL_n(\C)/B$, the extensive knowledge we have about the $T$-equivariant geometry, and associated combinatorics, of the flag variety can (and does) directly inform our understanding of the equivariant geometry of Hessenberg varieties. 

Secondly, a number of the arguments currently in the literature -- those which yield an explicit presentation of the (equivariant and ordinary) cohomology rings of Hessenberg varieties -- have a similar underlying structure, and a key component of these arguments is the commutative-algebraic theory of Hilbert series and regular sequences. Therefore, in this manuscript we have attempted to give (1) a broad overview of our topic, (2) a brief but sufficiently complete pedagogical account of the theories of equivariant cohomology and of regular sequences and Hilbert series, which would allow a reader to appreciate (3) an exposition of several concrete results concerning both the equivariant and ordinary cohomology rings of special classes of Hessenberg varieties.   

Some of the results in item (3) above concern presentations, via generators and relations, of these cohomology rings -- thus exhibiting them as quotients of polynomial rings.  These results should be understood, roughly speaking, as Hessenberg analogues the well-known Borel presentation of the cohomology rings of flag varieties. In another direction, some of the results in (3) are in the spirit of Schubert geometry and Schubert calculus. The idea for these results is that we wish to find explicit module (or $\Q$-vector space) bases of the cohomology rings of Hessenberg varieties, and to find explicit formulas for their multiplicative structure constants which have rich geometric and/or combinatorial interpretations.

 We now describe the content of the manuscript in some more detail. Section~\ref{sec: background} is a quick primer on the two essential theories for this paper, namely, Hilbert series and regular sequences (Section~\ref{subsec: background on Hilbert series and regular sequences}) and equivariant cohomology (Section~\ref{subsec: background equivariant cohomology}). While we have attempted to be complete enough that a reader can get the essential prerequisites for appreciating what follows, we did not aim to replace the full, formal treatments of these topics in subject-specific graduate-level textbooks. We have provided references for the reader to consult for more in-depth information.  Next, in Section~\ref{section: flag}, we illustrate some of the basic techniques and terminology of Section~\ref{subsec: background equivariant cohomology} by giving a detailed explanation of the $T$-equivariant cohomology of the (full) flag variety $\fln \cong GL_n(\C)/B$. Due to the fact, mentioned above, that much of the analysis of the equivariant topology of Hessenberg varieties comes from, and is motivated by, our knowledge of the special case $GL_n(\C)/B$, the exposition in Section~\ref{section: flag} serves more than just a pedagogical purpose.   In fact, this section sets up the necessary terminology and notation for much of our later mathematical arguments.   
 
 Equipped with the background material of Sections~\ref{sec: background} and~\ref{section: flag}, we shift our focus in Section~\ref{section: A variation on the theme: the equivariant cohomology rings of Hessenberg varieties} to explain how the computation of $H^*_T(GL_n(\C)/B)$, using the techniques of Section~\ref{subsec: background equivariant cohomology}, can be modified 
  to also account for the more general cases of Hessenberg varieties.   We attempt to highlight both the similarities and the distinctions between the flag variety case and the Hessenberg case, so as to build the reader's intuition on these topics.

  It is in Section~\ref{section: peterson in type A} that we get started in earnest.  In this section, we describe the cohomology rings of Peterson varieties in Lie type A (i.e. for the group $G = GL_n(\C)$) in different ways.  The Peterson variety is the ``smallest'' case of Hessenberg variety in a sense that can be made precise (cf. Section~\ref{section: peterson in type A}), and we present this case in some detail because it is simpler than the general case, and yet the arguments already richly illustrate the flavor of the arguments (both the Hilbert series techniques and the equivariant cohomology techniques).  More specifically, in Section~\ref{subsection:basis_Peterson_typeA}, we use techniques inspired from both the general theory of equivariant cohomology (specifically, the localization theorem in equivariant cohomology) and the theory of equivariant Schubert classes in $H^*_T(\fln)$ to derive a module basis for the equivariant cohomology rings of Peterson varieties, and we derive explicit formulas for some of the multiplicative structure constants in the spirit of classical Schubert calculus.  Then in Section~\ref{subsec: presentation of Peterson cohomology type A}   we switch gears and give an explicit presentation, by generators and relations, of the equivariant and ordinary cohomology rings of the Peterson variety, thus presenting it efficiently as a quotient of a polynomial ring by a concretely described ideal. It is in this section that we first see the effectiveness of the theory of Hilbert series and regular sequences.

  At this point we remind the reader that the theory of Hessenberg varieties is not limited to $G=GL_n(\C)$, i.e., to Lie type A. Hessenberg varieties can be readily defined and studied in all Lie types.  We have chosen to focus on $G=GL_n(\C)$ for expositional purposes, but to illustrate the general principle, in Section~\ref{section: peterson general Lie type} we describe the  generalizations -- to arbitrary Lie type -- of the results in Section~\ref{section: peterson in type A}.

  Finally, we present in Section~\ref{section: reg nilp Hess} the result that computes, for all regular nilpotent Hessenberg varieties in Lie type A, an explicit presentation via generators and relations of their equivariant and ordinary cohomology rings $H^*_{\S1}(\Hess(\mathsf{N},h))$ and $H^*(\Hess(\mathsf{N},h))$ \cite{AHHM2019}.  In particular, this is a generalization of the theorem explained in Section~\ref{subsec: presentation of Peterson cohomology type A}.  As we already intimated, the outline of the argument that proves the generalization is, at least in broad strokes, very similar to that given in the special case of the Peterson variety.  Therefore, in our exposition in Section~\ref{section: reg nilp Hess}, we have chosen to focus mostly on explaining what is \emph{different} about the general case.  In particular, what turns out to be one of the most conceptually important contributions of \cite{AHHM2019} is the derivation of the appropriate set of relations that are satisfied by the generators. It turns out that these relations are not so straightforward a generalization of the quadratic relations that are seen in the Peterson case.  Consequently, in Section~\ref{section: reg nilp Hess} we spend the bulk of the section giving a leisurely and hopefully fun explanation of how we were inspired to guess the appropriate relations for our chosen generators of $H^*_{\S1}(\Hess(\mathsf{N},h))$.  
  
  Our last Section~\ref{section.further.developments} is a biased and non-exhaustive list of some of the most recent developments in this theory. We provide many references to point the reader to the current literature.

\section{Background: the essential techniques}\label{sec: background}

In this section, we provide a crash course in the two main techniques that we will use throughout this exposition: the theory of Hilbert series and regular sequences,  and the theory of equivariant cohomology.  While we have kept the exposition tight, we have endeavored to give sufficient examples to illustrate the varied uses of these ideas, to prepare the reader for what follows in later sections.

\subsection{Background on Hilbert series and regular sequences} 
\label{subsec: background on Hilbert series and regular sequences}

In this section, we give a brief account of the commutative algebra tools we use throughout the manuscript: namely, Hilbert series and regular sequences. Our point of view on the subject will be biased by our intended use of the techniques to our study of Hessenberg varieties, and in particular, we do not claim to be an authoritative text on the subject.

We begin with the definition of \textbf{Hilbert series}. Let $k$ be a field and $R=\bigoplus_{i=0}^\infty
R_i$ an $\mathbb{N}$-graded\footnote{We take the convention that $\mathbb{N} = \{0,1,2,3,\cdots\}$, i.e., $\mathbb{N}$ includes $0$.} commutative $k$-algebra where each $R_i$ is a finite-dimensional $k$-vector space.
We assume $R_0 \cong k$ and that $R$ is finitely generated (the cohomology rings appearing in this paper all satisfy these conditions).
Then we define the \textbf{Hilbert series of $R$} to be 
\[
\Hilb(R,q):=\sum_{\ell=0}^\infty (\dim_{k}R_{\ell}) q^\ell \in \mathbb{N}[[q]]
\]
where $q$ is a formal parameter. 
We may define the Hilbert series for any graded $k$-vector space in a similar fashion. 
(Since we deal with rational cohomology rings in this manuscript, for simplicity, the reader may wish to take $\Q$ as the field $k$, though 
this is not necessary.) 
We equip the polynomial ring $k[x_1, \ldots, x_n]$ with the $\mathbb{N}$-grading defined by 
\begin{align*}
\deg x_i=2 \textup{ for all } i \in [n]. 
\end{align*} 
This slightly non-standard grading on the variables is motivated by the fact that, in this paper, the graded rings we consider are 
cohomology rings of algebraic varieties whose degrees are concentrated in even degrees, and the generators are degree-$2$ elements.

We begin with the following fundamental computation which forms the basis of many subsequent arguments. 

\begin{Example}\label{example.Hilbert.polynomial.ring}
Let $R$ be the polynomial ring $k[x_1, \ldots, x_n]$ with $\deg x_i=2$ for all $i \in [n]$.
Then the set of monomials  $$\{x_1^{\ell_1} x_2^{\ell_2} \cdots x_n^{\ell_n} \mid \ell_1 \geq 0, \ell_2 \geq 0, \ldots, \ell_n \geq 0, \ell_1+\ell_2+\cdots+\ell_n = \ell \}$$ 
form a basis for the degree-$2\ell$ subspace $R_{2\ell}$ of $R$, 
and $R$ is $0$ in all odd degrees. Hence the Hilbert series of $k[x_1, \ldots, x_n]$ is equal to 
\[
\Hilb(k[x_1, \ldots, x_n],q) = (1+q^2+q^4+\cdots)^n = \frac{1}{(1-q^2)^n}.
\]
\end{Example}

Next, we introduce the second of the two main notions we use from commutative algebra, that of \textbf{regular sequences}. 
For more details we refer the reader to \cite{stan96}.

We begin with a definition of non-zero-divisors. We warn the reader that the convention here -- in particular, whether or not a non-zero-divisor is itself non-zero \emph{by definition} -- depends on the textbook (e.g. \cite{stan96} takes one convention whereas \cite{DummitFoote} takes the other).  Below, we take the convention of \cite{stan96}, which requires a non-zero-divisor to be non-zero. 

\begin{Definition} 
An element $x \in R$ is a \textbf{non-zero-divisor} on $R$ if $x$ is non-zero, and moreover, whenever $y \in R$ and $xy=0$, then $y=0$. In other words, $x \in R$ is a non-zero-divisor in $R$ if $x$ is non-zero, and the map $R \xrightarrow{ \times x} R$ given by multiplication by $x$ is injective.
\end{Definition}

\begin{Definition}\label{def:regular}
Let $R$ be a non-zero $\mathbb{N}$-graded commutative algebra over $k$ and let $R_+$
denote the sum of the positively graded components of $R$.  We say that a 
sequence $\f_1,\dots,\f_r \in R_+$ of homogeneous elements is a \textbf{regular sequence} if 
for every $1 \leq j \leq r$, the equivalence class of 
$\f_j$ is a non-zero-divisor in the quotient ring $R/(\f_1,\dots,\f_{j-1})$.  
\end{Definition}

\begin{Remark} 
Notice that since the $\f_j$ in Definition~\ref{def:regular} are 
in $R_+$ for all $j$, 
the degree-$0$ piece $R_0 \not \cong 0$ survives in all the quotient rings $R/(\f_1, \cdots, \f_{j-1})$ as well as in the last quotient $R/(\f_1, \cdots, \f_r)$; therefore, all the quotients are non-zero rings. 
\end{Remark}

As in the case of Hilbert series above, the case of the polynomial ring $R=k[x_1,\ldots, x_n]$ provides the most basic example. 

\begin{Example}\label{ex: reg sequence poly ring} 
Let $R = k[x_1,\cdots,x_n]$ with $\deg x_i = 2$ for all $1 \leq i \leq n$. Then each $x_i \in R_+$ is homogeneous, and it is easy to check that the
sequence $x_1,\ldots,x_r \ (r \leq n)$ in the polynomial ring $k[x_1, \ldots, x_n]$ is a regular sequence.
More generally, for a choice of distinct indices $i_1, \ldots, i_r$ in $[n]$, the homogeneous sequence $x_{i_1},\ldots,x_{i_r}$
 is also a regular sequence. 
\end{Example}

Generalizing the second claim in Example~\ref{ex: reg sequence poly ring}, we have the following simple lemma. 

\begin{Lemma}
Let $R$ be a graded $k$-algebra.
Suppose that $\f_1,\dots,\f_r \in R_+$ is a regular sequence.  
Then any permutation of $\f_1,\dots,\f_r$ is also a regular sequence. 
\end{Lemma}

\begin{proof}
Since every permutation can be expressed as a product of simple transpositions, 
it is clear that it suffices to check the claim for a simple transposition. 
Fix an $i$, $1 \leq i \leq r-1$. We wish to show that 
 $\f_1,\dots,\f_{i+1},\f_i,\dots,\f_r$ is a regular sequence.  
In fact, it is enough to prove the claim for $i=1$, which can be seen as follows. 
Notice first that, from the definition of regular sequences, it is immediate that $\f_{i},\f_{i+1},\f_{i+2},\dots,\f_r$ is a regular sequence in $\bar{R}:=R/(\theta_1,\ldots,\theta_{i-1})$. Now suppose we know the claim for $i=1$;  then it follows that $\f_{i+1},\f_i,\f_{i+2},\dots,\f_r$ is also a regular sequence in $\bar{R}$. 
This observation, together with the original assumption that $\f_1,\dots,\f_r$ is a regular sequence in $R$, implies that  $\f_1,\dots,\f_{i+1},\f_i,\dots,\f_r$ is a regular sequence in $R$. Thus, it remains to prove the claim for $i=1$. 

For this we need to show that $\f_2,\f_1,\f_3,\dots,\f_r$ is a regular sequence in $R$, under the assumption that $\f_1, \f_2, \cdots, \f_r$ is a regular sequence. It would suffice to show the following: 
\begin{enumerate}
\item The kernel of the multiplication map $R \xrightarrow{\times \theta_2} R$ is trivial. 
\item The kernel of the multiplication map $R/(\theta_2) \xrightarrow{\times \theta_1} R/(\theta_2)$ is trivial. 
\end{enumerate}
We begin with $(1)$. Let $I$ be the kernel of the multiplication map $R \xrightarrow{\times \theta_2} R$. Since $\theta_2$ is homogeneous, we know that $I$ is a homogeneous ideal. Thus it is enough to show that any homogeneous element $x \in I$ is equal to $0$. 
To see this, let $x \in I$ be homogeneous and consider the commutative diagram
\[
  \begin{CD}
     R @>{\times \theta_2}>> R \\
  @VVV    @VVV \\
     R/(\theta_1)   @>{\times \theta_2}>>  R/(\theta_1)
  \end{CD}
\]
where we know that the bottom horizontal map is injective by assumption (and we slightly abuse notation and denote by $\theta_2$ its equivalence class in $R/(\theta_1)$). Since $x$ is in the kernel of the top horizontal arrow, it is also in the kernel of the composition $R \to R/(\theta_1) \stackrel{\times \theta_2}{\rightarrow} R/(\theta_1)$, and by the injectivity of the bottom horizontal map it follows that 
$x=\theta_1 x'$ in $R$ for some homogeneous element $x' \in R$.
Since $x \in I$, we also know that $\theta_1(\theta_2x')=\theta_2(\theta_1x')=\theta_2x=0$ in $R$. 
Since $\theta_1$ is a non-zero-divisor in $R$ by assumption, we conclude that $\theta_2x'=0$ in $R$, which in turn implies $x' \in I$.
Moreover, note that $\deg(x') = \deg(x) - \deg(\theta_1) < \deg(x)$, since $\theta_1$ is homogeneous of strictly positive degree. Continuing the same argument for $x'$ and so on, we see that $x = \theta_1^k y$ for some $k$ where $0 \leq \deg(y) < \deg(\theta_1)$ and $y \in I$. 
Applying the argument again to $y$ we obtain $y = \theta_1 z$ for some homogeneous $z$, but  since $\deg(y) < \deg(\theta_1)$, this is impossible unless $z=0$. Hence $z=0$ and therefore $y=0$ and therefore $x=0$, as was to be shown.

We now prove (2).  Let $J$ be the kernel of the multiplication map $R/(\theta_2) \xrightarrow{\times \theta_1} R/(\theta_2)$. Since 
$\theta_1, \theta_2$ are both homogeneous, $R/(\theta_2)$ is a graded ring and $J$ is a homogeneous ideal. 
Let $y$ be a homogeneous element in $J$. It suffices to show that $y=0$ in $R/(\theta_2)$.
Since $y \in J$, we know there exists $z \in R$ such that  $\theta_1y=\theta_2z$ in $R$.
This means that $\theta_2z=0$ in $R/(\theta_1)$, and by the original assumption that $\theta_1, \theta_2$ form a regular sequence, we know this implies $z = 0$ in $R/(\theta_1)$. Thus there exists $y' \in R$ such that $z=\theta_1 y'$ in $R$
and we conclude $\theta_1 y=\theta_1 \theta_2 y'$ in $R$. 
Again by the original assumption, we know $\theta_1$ is a non-zero-divisor in $R$, so $\theta_1 y = \theta_1 \theta_2 y'$ implies that $y=\theta_2 y'$ in $R$, which in turn implies $y=0$ in $R/(\theta_2)$. This is what was to be shown, and we conclude that $J$ is trivial, as desired. 

Therefore, $\f_2,\f_1,\f_3,\dots,\f_r$ is a regular sequence in $R$.
This completes the proof.

\end{proof}

Next, we introduce some notation that will be useful below. Suppose that $\f_1,\dots,\f_r \in R_+$ is a homogeneous sequence. We define
\begin{equation}\label{eq: def Ri} 
R^{(j)}:=R/(\f_1,\dots,\f_j)
\end{equation} 
for $1\le j\le r$ and $R^{(0)}:=R$.
By definition, $\f_1,\dots,\f_r$ is a regular sequence in $R$ if and only if the multiplication map 
$R^{(j-1)}\stackrel{\times \f_j}\longrightarrow R^{(j-1)}$ is injective for all $1\le j\le r$.
In other words, for each $1 \leq j \leq r$ we have the following exact sequence of $R$-modules: 
\begin{equation} \label{eq:exact_regular_sequence}
0 \to R^{(j-1)}\stackrel{\times \f_j}\longrightarrow R^{(j-1)}\to R^{(j)}\to 0. 
\end{equation}
Using the exact sequences above, we can identify different criteria for being a regular sequence.  The next lemma is an example.

\begin{Lemma}\label{lemma: reg seq alg-ind} 
Let $R$ be a graded $k$-algebra and suppose that $\f_1, \dots, \f_r \in R_+$ are homogeneous elements.
Then $\f_1, \dots, \f_r$ is a regular sequence if and only if the set $\f_1,\dots,\f_r$ is algebraically independent over $k$ and $R$ is a free $k[\f_1,\dots,\f_r]$-module.
\end{Lemma}

\begin{proof}
We will first assume that $\f_1,\cdots, \f_r$ is a regular sequence and then prove that $\{\f_1, \cdots, \f_r\}$ is algebraically independent over $k$ and that $R$ is a free $k[\f_1, \cdots, \f_r]$-module. We begin by showing the algebraic independence, using induction on $r$.

Let us consider the base case when $r=1$. An algebraic relation involving the single element $\f_1$ is of the form $\alpha \theta_1^\ell=0$ for some $\alpha \in k$ and $\ell>0$ a positive integer. This implies $\alpha=0$ because $\theta_1$ is a non-zero-divisor on $R$. Thus $\f_1$ is algebraically independent, as claimed. 
Suppose now that $r > 1$ and that the claim of the lemma holds for $r-1$. 
Note that since the $\f_i$ are assumed homogeneous, we may without loss of generality assume that an algebraic relation among the $\f_1, \cdots, \f_r$ is also a homogeneous relation.  So suppose now that 
\begin{equation} \label{eq:proof_algebraically_independent}
\sum_{\ell_1 \geq 0, \dots, \ell_r \geq 0} \alpha_{\ell_1,\ldots,\ell_r} \theta_1^{\ell_1} \cdots \theta_r^{\ell_r}=0 \ \ \ {\rm in} \ R
\end{equation}
is such a homogeneous relation, where the $\ell_i \in \Z$ are non-negative integers and 
$\alpha_{\ell_1,\ldots,\ell_r} \in k$. We wish to show that all the coefficients $\alpha_{\ell_1,\ldots,\ell_r}$ are $0$. 
If we consider the equation \eqref{eq:proof_algebraically_independent} in the quotient ring $\bar{R}:=R/(\theta_1)$, then all terms containing a $\theta_1$ are equivalent to $0$ in $R/(\theta_1)$ and hence we obtain the relation 
\begin{equation*} 
\sum_{\ell_1 = 0, \ell_2 \geq 0, \dots, \ell_r \geq 0} \alpha_{\ell_1,\ldots,\ell_r} \theta_2^{\ell_2} \cdots \theta_r^{\ell_r}=0 \ \ \ {\rm in} \ \bar{R}.
\end{equation*}
Recall that, by assumption, $\f_2,\dots,\f_r$ is regular sequence in $\bar{R}=R/(\theta_1)$ of length $r-1$. Thus, by the inductive hypothesis applied to the graded ring $\bar{R} := R/(\theta_1)$ and the sequence $\f_2, \cdots, \f_r$, we conclude that $\alpha_{\ell_1,\ldots,\ell_r}=0$ 
when $\ell_1=0$. 
Returning to the original algebraic relation~\eqref{eq:proof_algebraically_independent} in $R$, the fact that many coefficients are equal to $0$ simplifies that equation to yield the following relation 
\begin{equation*} 
\sum_{\ell_1 > 0, \ell_2 \geq 0, \dots, \ell_r \geq 0} \alpha_{\ell_1,\ldots,\ell_r} \theta_1^{\ell_1} \cdots \theta_r^{\ell_r}=0  \ \ \ {\rm in} \ R.
\end{equation*}
Notice that since $\ell_1>0$ for all the terms that appear on the LHS of the above relation, and since $\theta_1$ is by assumption a non-zero-divisor in $R$, we can factor out a $\theta_1$ and by setting $\ell'_1 := \ell_1-1$ we obtain  
\begin{equation*} 
\sum_{\ell'_1 \geq 0, \dots, \ell_r \geq 0} \alpha_{\ell'_1+1,\ldots,\ell_r} \theta_1^{\ell'_1} \theta_2^{\ell_2} \cdots \theta_r^{\ell_r}=0  \ \ \ {\rm in} \ R.
\end{equation*}
This is again a homogeneous algebraic relation in $R$ among the $\f_1, \cdots, \f_r$.  
Applying the same argument as above, we can conclude that $\alpha_{\ell_1,\ldots,\ell_r}=0$ if $\ell'_1=0$, i.e. $\ell_1=1$, 
and again as above we obtain the relation 
$\sum_{\ell''_1 \geq 0, \dots, \ell_r \geq 0} \alpha_{\ell''_1+2,\ldots,\ell_r} \theta_1^{\ell''_1} \cdots \theta_r^{\ell_r}=0 $ in $R$.
Continuing this process, 
we can conclude that all coefficients $\alpha_{\ell_1,\ldots,\ell_r}$ are equal to $0$ and hence that $\{\f_1,\dots,\f_r\}$ is algebraically independent over $k$, as desired.

Next, we need to show that if $\f_1, \dots, \f_r$ is a regular sequence, then $R$ is a free $k[\f_1,\dots,\f_r]$-module.
To accomplish this, we prove that $R^{(j-1)}$ is a free $k[\f_j,\dots,\f_r]$-module for all $1 \leq j \leq r$, by a descending induction argument on $j$.
First, let us consider the base case when $j=r$. 
 From the exact sequence in \eqref{eq:exact_regular_sequence} for $j=r$ it follows that 
for each $\ell$-th graded piece $R^{(r-1)}_\ell$ of $R^{(r-1)}$ we can decompose 
 \begin{equation}\label{eq: decomp R r-1}
 R^{(r-1)}_\ell \cong (\theta_r \cdot R^{(r-1)})_\ell \oplus R^{(r)}_\ell. 
 \end{equation} 
 (Note that  $(\theta_r \cdot R^{(r-1)})_\ell = 0$ if $\ell < \deg(\theta_r)$.) Moreover, the factor $(\theta_r \cdot R^{(r-1)})_\ell$ can be decomposed further, using the homogeneity of $\theta_r$, namely, for $\ell = \deg(\theta_r) \cdot q + s$ for $q,s \in \Z$ and $0 \leq s < \deg(\theta_r)$ we have 
\begin{equation}\label{eq: decomp theta r R r-1}
\begin{split} 
 (\theta_r \cdot R^{(r-1)})_{\ell}  & =\theta_r \cdot R^{(r-1)}_{\ell-\deg(\theta_r)} \cong  \theta_r \cdot \big( (\theta_r \cdot R^{(r-1)})_{\ell-\deg(\theta_r)} \oplus R^{(r)}_{\ell-\deg(\theta_r)} \big) \\
  & =\theta_r^2 \cdot (R^{(r-1)})_{\ell - 2 \deg(\theta_r)} \oplus \theta_r \cdot (R^{(r)})_{\ell - \deg(\theta_r)}  \\
  & \cong \cdots \\ 
  & \cong \theta_r^q \cdot (R^{(r-1)})_{\ell - q \deg(\theta_r)} \oplus \theta_r^{q-1} \cdot (R^{(r)})_{\ell - (q-1) \deg(\theta_r)} \oplus \cdots \oplus \theta_r \cdot (R^{(r)})_{\ell - \deg(\theta_r)} \\
  & \cong \theta_r^q \cdot (R^{(r)})_{\ell - q \deg(\theta_r)} \oplus \theta_r^{q-1} \cdot (R^{(r)})_{\ell - (q-1) \deg(\theta_r)} \oplus \cdots \oplus \theta_r \cdot (R^{(r)})_{\ell - \deg(\theta_r)}. \\ 
 \end{split} 
 \end{equation}
 
Applying~\eqref{eq: decomp R r-1} and~\eqref{eq: decomp theta r R r-1} 
we then see that 
 \begin{equation} 
 (R^{(r-1)})_{\ell} \cong \bigoplus_{t=0}^q \theta_r^t \cdot (R^{(r)})_{\ell - t \deg(\theta_r)}. 
 \end{equation} 

 Since this is true of any $\ell$,  
this implies that a $k$-basis for $R^{(r)}$ forms a $k[\theta_r]$-basis of $R^{(r-1)}$, and hence 
$R^{(r-1)}$ is a free $k[\theta_r]$-module. This concludes the base case. 
Now suppose that $j < r$ and assume by induction that the claim holds for $j+1$.
The argument for this inductive step is very similar to the base case, so we will be brief.
By our original hypothesis, we know we have an exact sequence~\eqref{eq:exact_regular_sequence}. 
Moreover, by the inductive assumption, $R^{(j)}$ is a free $k[\theta_{j+1},\ldots,\theta_r]$-module.
An argument similar to the one for the base case shows that 
a $k[\theta_{j+1},\ldots,\theta_r]$-basis $\{v_{\lambda} \}_\lambda$ of $R^{(j)}$ spans $R^{(j-1)}$ with $k[\theta_j,\ldots,\theta_r]$-coefficients. 
To see the linear independence of $\{v_{\lambda}\}_\lambda$ over $k[\theta_j,\ldots,\theta_r]$, suppose that 
\begin{align} \label{eq:linearly_independence}
\sum_{\lambda} g_\lambda v_\lambda = 0 \ \ \ {\rm in} \ R^{(j-1)} 
\end{align}
for $g_\lambda \in k[\theta_j,\ldots,\theta_r]$. We wish to show $g_\lambda=0$ for all $\lambda$. 
For each $\lambda$, one can write $g_\lambda = \sum_{s \geq 0} \theta_j^s f_s^{(\lambda)}$ for some $f_s^{(\lambda)} \in k[\theta_{j+1},\ldots,\theta_r]$ (only finitely many $f_s^{(\lambda)}$ are non-zero). Considering the image of this relation in $R^{(j)}$ we obtain 
$$
\sum_{\lambda} f_0^{(\lambda)} v_\lambda =0 \ \ \ {\rm in} \ R^{(j)}.
$$
Since $\{v_\lambda \}$ is linearly independent in $R^{(j)}$ over $k[\theta_{j+1},\ldots,\theta_r]$, we have $f_0^{(\lambda)}=0$ for all $\lambda$.
Thus, we can write \eqref{eq:linearly_independence} as $\sum_{\lambda} \big(\sum_{s \geq 1} \theta_j^s f_s^{(\lambda)} \big) v_\lambda = 0$ in $R^{(j-1)}$. 
However, since $\theta_j$ is a non-zero-divisor in $R^{(j-1)}$, we have 
\begin{align*}
\sum_{\lambda} \big(\sum_{s \geq 0} \theta_j^s f_{s+1}^{(\lambda)} \big) v_\lambda = 0 \ \ \ {\rm in} \ R^{(j-1)}. 
\end{align*}
Considering the image of this relation in $R^{(j)}$ and proceeding as above (and as in the argument earlier in this proof for the algebraic independence) it is not hard to conclude that $f_s^{(\lambda)}=0$ for all $\lambda$ and $s$, hence $g_\lambda = 0$ for all $\lambda$, as desired. 
Therefore, a $k[\theta_{j+1},\ldots,\theta_r]$-basis $\{v_{\lambda} \}_\lambda$ of $R^{(j)}$ forms a $k[\theta_j,\ldots,\theta_r]$-basis for $R^{(j-1)}$, and $R^{(j-1)}$ is a free module over $k[\theta_j,\ldots,\theta_r]$, as desired. This completes the induction step, and 
we conclude that $R=R^{(0)}$ is a free $k[\f_1,\dots,\f_r]$-module, as was to be shown.

We now prove the converse direction. To do this, we suppose that $\f_1,\dots,\f_r$ is algebraically independent over $k$ and that $R$ is a free $k[\f_1,\dots,\f_r]$-module.  We then need to prove that $\f_1, \dots, \f_r$ is a regular sequence. 
We note first that it suffices to show that $R^{(j-1)}$ is a free $k[\theta_j,\ldots,\theta_r]$-module, since this would imply that 
$\theta_j$ is a non-zero-divisor in $R^{(j-1)}$. 
In what follows, it is useful also to note that since the $\f_1, \cdots, \f_r$ are algebraically independent over $k$ by assumption, the subring 
$k[\theta_1, \cdots, \theta_r]$ generated by $\theta_1, \cdots, \theta_r$ is isomorphic to a polynomial ring.
Now let $\{v_\lambda\}_\lambda$ be a $k[\theta_1,\ldots,\theta_r]$-basis of $R$.
Then it follows that $\{v_\lambda\}_\lambda$ generates $R^{(j-1)}$ over $k[\theta_j,\ldots,\theta_r]$.
It suffices now to show that $\{v_\lambda\}_\lambda$ is linearly independent in $R^{(j-1)}$ over $k[\theta_j,\ldots,\theta_r]$.
Suppose we have a relation $\sum_{\lambda} g_{\lambda} v_\lambda =0$ in $R^{(j-1)}$ for some $g_{\lambda} \in k[\theta_j,\ldots,\theta_r]$. We wish to show $g_\lambda=0$ for all $\lambda$. 
Since $R^{(j-1)} = R/(\theta_1, \cdots, \theta_{j-1})$ by definition, the relation above implies that we have $\sum_{\lambda} g_{\lambda} v_\lambda = \theta_1 y_1 + \dots + \theta_{j-1} y_{j-1}$ in $R$ for some $y_1,\ldots,y_{j-1} \in R$.
On the other hand, since the $v_\lambda$ generate, we know we can write $y_s=\sum_{\lambda} f_{\lambda}^{(s)} v_\lambda$ for some $f_{\lambda}^{(s)} \in k[\theta_1,\ldots,\theta_r]$, which in turn implies 
$$
\sum_{\lambda} (g_{\lambda} -\theta_1 f_{\lambda}^{(1)} - \dots -\theta_{j-1} f_{\lambda}^{(j-1)}) v_\lambda=0 \ \ \ {\rm in} \ R.
$$
By linear independence of the $\{v_\lambda\}$ over $k[\theta_1,\cdots,\theta_r]$, we conclude 
\begin{equation}\label{eq: g lambda f lambda} 
g_{\lambda} -\theta_1 f_{\lambda}^{(1)} - \dots -\theta_{j-1} f_{\lambda}^{(j-1)}=0 \Leftrightarrow 
g_{\lambda} = \theta_1 f_{\lambda}^{(1)} + \cdots + \theta_{j-1} f_{\lambda}^{(j-1)} 
\end{equation}
 for all $\lambda$.
However, by construction, the $g_\lambda$ are elements of $k[\theta_j,\ldots,\theta_r]$. Since equation~\eqref{eq: g lambda f lambda} holds in $k[\theta_1,\cdots,\theta_r]$, which is a polynomial ring in the variables $\theta_j$'s, if the $g_\lambda$ are equal to a polynomial involving $\theta_1, \cdots, \theta_{j-1}$, then both sides must be equal to $0$. Thus~\eqref{eq: g lambda f lambda} implies $g_\lambda=0$ for all $\lambda$, as desired. 
Hence, $\{v_\lambda\}_\lambda$ is a linearly independent in $R^{(j-1)}$ over $k[\theta_j,\ldots,\theta_r]$ and hence they form a basis. 
Therefore, $R^{(j-1)}$ is a free $k[\theta_j,\ldots,\theta_r]$-module, as was to be shown. 
\end{proof}

\begin{Remark} \label{remark:regular sequence basis}
The proof given for Lemma~\ref{lemma: reg seq alg-ind} shows that if $\f_1, \dots, \f_r$ is a regular sequence in $R$, then a $k$-basis $\{v_{\lambda} \}_\lambda$ of $R/(\f_1, \dots, \f_r)$ forms a $k[\theta_1,\ldots,\theta_r]$-basis for $R$.
\end{Remark}

Here is another well-known characterization of regular sequences, given in terms of Hilbert series. 

\begin{Lemma} \label{lemma: hilbert regular criterion}
A homogeneous sequence $\f_1,\dots,\f_r \in R_+$ 
is a regular sequence if and only if 
\begin{equation} \label{eq:4.6}
\Hilb(R/(\f_1,\dots,\f_r),q)=\Hilb(R,q)\prod_{j=1}^r(1-q^{\deg{\f_j}}). 
\end{equation}
\end{Lemma} 

\begin{proof}
We first prove the ``only if'' direction. Suppose that $\f_1,\dots,\f_r \in R_+$ is a regular sequence. Then the exact sequence \eqref{eq:exact_regular_sequence} yields 
\begin{equation*} 
\Hilb(R^{(j)},q)=\Hilb(R^{(j-1)},q)(1-q^{\deg\f_j}) \quad\text{for any $1\le j\le r$}. 
\end{equation*} 
From this, the equality \eqref{eq:4.6} immediately follows.

We now prove the converse, i.e., we show that if \eqref{eq:4.6} holds, then $\f_1,\dots,\f_r$ is a regular sequence.
For this purpose, let $K_j$ denote the kernel of the multiplication map $R^{(j-1)} \xrightarrow{\times \f_j} R^{(j-1)}$. 
Then we have an exact sequence 
\begin{equation*} 
0 \to R^{(j-1)}/K_j \stackrel{\times \f_j}\longrightarrow R^{(j-1)}\to R^{(j)}\to 0 \quad \textup{for $1\le j\le r$},
\end{equation*}
from which we can conclude that 
\begin{equation*} 
\Hilb(R^{(j)},q)=\Hilb(R^{(j-1)},q)(1-q^{\deg\f_j})+q^{\deg\f_j}\Hilb(K_j,q) \quad\textup{for any $1\le j\le r$}.
\end{equation*} 
A straightforward inductive computation yields the formula
\begin{align*} 
\Hilb(R^{(r)},q)=&\Hilb(R,q)\prod_{j=1}^r(1-q^{\deg{\f_j}}) \\
&+ \sum_{j=1}^r (1-q^{\deg{\f_r}})(1-q^{\deg{\f_{r-1}}})\cdots(1-q^{\deg{\f_{j+1}}})q^{\deg{\f_j}} \Hilb(K_j,q).
\end{align*}
Since we assume that \eqref{eq:4.6} holds, we have 
\begin{align} \label{eq:proof_regular_Hilbert} 
\sum_{j=1}^r (1-q^{\deg{\f_r}})(1-q^{\deg{\f_{r-1}}})\cdots(1-q^{\deg{\f_{j+1}}})q^{\deg{\f_j}} \Hilb(K_j,q)=0.
\end{align} 
To prove that $\f_1,\cdots, \f_r$ is a regular sequence, it would suffice to show that $K_j=0$ for all $j$, 
for which it in turn suffices to show that $\Hilb(K_j,q)=0$ for all $1 \leq j \leq r$. 
Suppose, in order to obtain a contradiction, that there exist 
$s$ such that $\Hilb(K_s,q) \neq 0$.
For any such $s$, let us write $\Hilb(K_s,q) = a_{m_s} q^{m_s} + a_{m_s+1} q^{m_s+1}+\cdots$ where $a_{m_s} > 0$, 
i.e., $a_{m_s} q^{m_s}$ is the lowest term of $\Hilb(K_s,q)$. 
Then, the coefficient of the lowest term for the left hand side of \eqref{eq:proof_regular_Hilbert} is of the form $\sum a_{m_s}$ where the sum runs over $s$ such that $\deg{\f_s}+m_s$ is the minimum of $\{\deg{\f_\ell}+m_\ell \mid \Hilb(K_\ell,q) \neq 0 \}$. 
Comparing to the corresponding coefficient on the right hand side of \eqref{eq:proof_regular_Hilbert}, we see that 
the sum $\sum a_{m_s}$ must be zero. Since the $a_{m_s}$ are all positive, this gives the desired contradiction.
Hence, we have $\Hilb(K_j,q)=0$ for all $1 \leq j \leq r$, which means the the kernel $K_j$ of the multiplication map $R^{(j-1)} \xrightarrow{\times \f_j} R^{(j-1)}$ is trivial for all $1 \leq j \leq r$.
Therefore, $\f_1,\dots,\f_r$ is a regular sequence, as was to be shown. 
\end{proof} 

\begin{Remark} \label{remark:regular_sequence}
Suppose that $\f_1,\cdots, \f_r$ are a sequence of positive-degree homogeneous elements in $\Q[x_1,\cdots,x_n]$. 
Since the Hilbert series of $\Q[x_1,\cdots,x_n]/(\f_1,\cdots,\f_r)$ (where each graded piece is viewed as a $\Q$-vector space) is equal to the Hilbert series of $\C[x_1,\cdots,x_n]/(\f_1,\cdots,\f_r)$ (where each graded piece is viewed as a $\C$-vector space), 
and the right hand side of~\eqref{eq:4.6} depends only on the degrees of the $\f_i$, it follows that 
$\{\f_1, \cdots, \f_r\}$ are a regular sequence in $\Q[x_1,\cdots,x_n]$ if and only if $\f_1, \cdots, \f_r$ form a regular sequence in 
$\C[x_1, \cdots, x_n]$. 
\end{Remark} 

For the remainder of this section, we focus on the case when $R=\Q[x_1,\dots,x_n]$ is the polynomial ring with rational coefficients and $n$ variables. 
In this situation, there is a useful criterion, stated in Lemma~\ref{lemma: solution regular criterion}, for 
determining that a sequence of positive-degree homogeneous polynomials $\f_1,\dots,\f_n \in \Q[x_1,\dots,x_n]$ is a 
regular sequence. 
Here, it is important that the number of the homogeneous polynomials in the sequence is equal to the number of the variables  
$x_1, \cdots, x_n$ generating the ring $\Q[x_1,\dots,x_n]$.
We have the following.

\begin{Lemma}\label{lemma: solution regular criterion}
A sequence of positive-degree
  homogeneous elements $\f_1,\dots,\f_n$ in  $\Q[x_1,\dots,x_n]$ is a regular sequence if and only if the
  solution set in $\C^n$ of the equations $\f_1=0,\dots,\f_n=0$
  consists only of the origin $\{0\}$. 
\end{Lemma}

Before discussing the proof of Lemma~\ref{lemma: solution regular criterion}, we pause to give an example of an application.

\begin{Example} \label{example:flag regular sequences1}
We can use Lemma~\ref{lemma: solution regular criterion} to prove that 
the elementary symmetric polynomials $$\mathfrak{e}_1(x_1,\dots ,x_n),\ldots,\mathfrak{e}_n(x_1,\dots ,x_n)$$ form a regular sequence in $\Q[x_1,\ldots,x_n]$.
Indeed, by the lemma, it suffices to show that any point $x  \in \C^n$ satisfying $\mathfrak{e}_j(x_1,\dots ,x_n)=0$ for all $j \in [n]$ 
must satisfy $x=0$. To see this, first observe that $\mathfrak{e}_n(x_1,\dots ,x_n)=x_1 x_2 \cdots x_n = 0$ implies that some $x_i$ must be equal to $0$, and 
without loss of generality, let us assume that $x_n=0$. 
Since $x_n=0$, the original $n$ equations then imply that $\mathfrak{e}_j(x_1,\dots ,x_{n-1})=0$ for all $j \in [n-1]$. 
Now $\mathfrak{e}_{n-1}(x_1, \cdots, x_{n-1}) = x_1 x_2 \cdots x_{n-1} =0$ implies that some $x_i$ with $i \in [n-1]$ must equal $0$. 
Proceeding in this manner, we conclude that $x_1=\cdots=x_n=0$, as desired.
Applying Lemma~\ref{lemma: solution regular criterion}, we conclude that
 $\mathfrak{e}_1(x_1,\dots ,x_n),\ldots,\mathfrak{e}_n(x_1,\dots ,x_n)$ form a regular sequence in $\Q[x_1,\ldots,x_n]$.

From this we can also compute, using Lemma~\ref{lemma: hilbert regular criterion},
the Hilbert series of the quotient ring $$\Q[x_1,\dots,x_n]/(\mathfrak{e}_1(x_1,\dots ,x_n),\dots,\mathfrak{e}_n(x_1,\dots ,x_n))$$ as
\begin{equation*}
\begin{split} 
&\Hilb(\Q[x_1,\dots,x_n]/(\mathfrak{e}_1(x_1,\dots ,x_n),\dots,\mathfrak{e}_n(x_1,\dots ,x_n)),q) \\
= &\Hilb(\Q[x_1,\cdots,x_n], q) \cdot \prod_{j=1}^n(1- q^{\deg(\mathfrak{e}_j(x_1,\dots ,x_n))}) \\
= &\frac{1}{(1-q^2)^n} \cdot \prod_{j=1}^n (1-q^{2j}) \\ 
= &\prod_{j=1}^n \frac{1-q^{2j}}{1-q^2} \\ 
= &\prod_{j=1}^{n-1} (1+q^2+q^4+\cdots+q^{2j}),
\end{split} 
\end{equation*} 
where we are taking the grading to be $\deg x_i=2$ for all $i \in [n]$ and we have used that the Hilbert series of a polynomial ring in $n$ variables is $\frac{1}{(1-q^2)^n}$ (where each variable has degree $2$). 
By a classical result in the study of the topology of flag varieties (as we will recount briefly in Section~\ref{section: flag}), it is known that the cohomology ring of $\fln$ (with $\Q$ coefficients) is precisely $\Q[x_1,\dots,x_n]/(\mathfrak{e}_1(x_1,\dots ,x_n),\dots,\mathfrak{e}_n(x_1,\dots ,x_n))$. Thus the above Hilbert series is, in fact, exactly the Poincar\'e polynomial of cohomology ring (over $\Q$) of the flag variety $\fln$. 
\end{Example}

We also give a second example which we will see again when we discuss the cohomology rings of Peterson varieties (cf. Section~\ref{section: peterson in type A}).

\begin{Example}
We define polynomials $f_1$ and $f_2$ in $\Q[z_1,z_2]$ as 
\begin{align*}
f_1=z_1\left(z_1-\frac{1}{2}z_2\right), \ f_2=z_2\left(z_2-\frac{1}{2}z_1\right).
\end{align*}
One can easily check from Lemma~\ref{lemma: solution regular criterion} that they form a regular sequence in $\Q[z_1,z_2]$. 
We will encounter these polynomials $f_1$ and $f_2$ again in Section~\ref{section: peterson in type A}, where we will see that
these polynomials are the relations defining the cohomology ring (over $\Q$) of the Peterson variety in $\fl(3)$, the variety of flags in $\C^3$. 
\end{Example}

We now proceed to the proof of Lemma~\ref{lemma: solution regular criterion}. However, before launching into the proof, we take a moment to warn the reader that, for the purposes of this proof, we have chosen a different style of exposition than what the reader has hitherto experienced in this manuscript. Our rationale is that the proof requires some prior knowledge of commutative algebra, in particular the topic of Krull dimension, as well as concepts such as `homogeneous systems of parameters' and the `Cohen-Macaulay property'.  Although these are standard tools in the toolkit of commutative algebra, it would take us too far afield to provide detailed definitions and a thorough treatment of these subjects. As such, we have treated these commutative algebra facts as ``black boxes'' and have provided the reader, to the best of our ability, with specific references where the relevant details may be found. We rely mainly on the text \cite{stan96}. As supplemental references, the reader may also consult the well-known texts of Bruns-Herzog, Eisenbud, and Matsumura \cite{BrunsHerzog, Eisenbud, Matsumura}.

With the above caveat in mind, we embark on the proof.

\begin{proof}[Proof of Lemma~\ref{lemma: solution regular criterion}]
First, we note that by the observation in Remark~\ref{remark:regular_sequence}, in order to prove the claim of the lemma, 
it suffices to show that 
the sequence $\f_1, \cdots, \f_n$ in $\C[x_1,\cdots, x_n]$ is a regular sequence if and only if the solution set in $\C^n$ of the equations $\f_1=0,\dots,\f_n=0$ consists only of the origin $\{0\}$. Thus we will focus on the case when the coefficient field is $\C$.

Second, it is a general fact about Krull dimension that a homogeneous sequence $\f_1, \cdots, \f_n$ in $k[x_1,\cdots,x_n]$ is a regular sequence if and only if the Krull dimension of the quotient ring $k[x_1,\cdots, x_n]/(\f_1,\cdots,\f_n)$ is zero. A very brief sketch of the proof goes as follows: 
if $\f_1,\cdots,\f_n$ is a regular sequence, then the $\f_1,\cdots,\f_n$ is a partial homogeneous system of parameters (see \cite[Definition~5.6 and surrounding discussion]{stan96}). Noting that the cardinality of this sequence is the same as that of the set of variables in the ring, we may conclude that the Krull dimension of the quotient is zero. See \cite[Definition~5.1]{stan96}. 
For the other direction, suppose that $k[x_1,\cdots,x_n]/(\f_1, \cdots, \f_n)$ has Krull dimension zero. Then the $\f_1, \dots,
\f_n$ are a homogeneous system of parameters for
$k[x_1,\dots,x_n]$ \cite[Definition 5.1]{stan96}. Moreover, since
the polynomial ring 
$k[x_1,\dots,x_n]$ is Cohen-Macaulay, by \cite[Theorem
5.9]{stan96} we may conclude that the homogeneous system of parameters
$\f_1, \dots, \f_n$ is a regular sequence. 

From the above facts, we see that it suffices to show that the Krull dimension of $\C[x_1,\dots,x_n]/(\f_1,\dots,\f_n)$ is zero if and only if the solution set of $\{\f_1=\cdots=\f_n=0\}$ consists only of the origin.
Since we are now working over $\C$, we may apply Hilbert's Nullstellensatz (see e.g. \cite[Theorem 1.6]{Eisenbud}) and the quotient ring
\[
\C[x_1,\dots,x_n]/(\f_1,\dots,\f_n)
\]
has Krull dimension $0$ if
and only if the algebraic set in $\C^n$ defined by the equations
$\f_1=0,\dots,\f_n=0$ is zero-dimensional. Since the polynomials
$\f_1,\dots,\f_n$ are assumed to be homogeneous, this corresponding
zero-dimensional algebraic set in $\C^n$ can only be $\{0\}$. 
This proves the lemma. 
\end{proof}

\subsection{Background on equivariant cohomology} 
\label{subsec: background equivariant cohomology}

Now that we have introduced the tools we need from commutative algebra in Section~\ref{subsec: background on Hilbert series and regular sequences}, we turn 
to the other toolkit we need. Specifically, in this section we briefly
recall some basic background on equivariant topology, or more specifically, the theory of equivariant cohomology. 
Standard treatments of the subject include \cite{AlldayPuppe, Bredon, Brion, Hsi}, to which the reader may refer for more leisurely expositions.

There are two remarks worth making at the beginning. First, the theory of equivariant cohomology may be formulated for general Lie groups. We note first that, in this exposition, we focus exclusively on the abelian case, i.e., the case of the torus. We next note that it is customary in the standard treatments (such as those mentioned above) to discuss mainly the case of the \emph{compact} torus, isomorphic to $(S^1)^n)$ for some $n$. However, in recent years, equivariant cohomology has become more well-known and standard among algebraic geometers --- as indicated (for instance) by the appearance of a text on the subject by Anderson and Fulton \cite{AndersonFulton} -- for whom the most natural setting is that of the \emph{complex (algebraic)} torus, isomorphic to $(\C^*)^n$ for some $n$. It turns out that the distinction between the algebraic and compact torus is not very important, because -- as is carefully explained in \cite[Example 1.1 of Chapter 1]{AndersonFulton}  for the case of tori -- the cohomology of a complex group is the same as that of its maximal compact subgroup. Therefore, despite the fact that many of standard topology references (such as those listed above) use the \emph{compact} torus, the reader may be assured that when we speak of the \emph{complex (or algebraic)} torus in this section, the referenced statements do still hold.

Secondly, there is a potential for confusion regarding the notation which appears in Section~\ref{subsec: background on Hilbert series and regular sequences} as compared to the notation in the current section, which we aim to clarify before we begin. In Section~\ref{subsec: background on Hilbert series and regular sequences} we used the variables $x_i$ exclusively to denote the indeterminates in the polynomial rings under discussion. In the current section, we will use two sets of variables $\{t_1, t_2, \cdots, t_n\}$ and $\{x_1, x_2, \cdots, x_m\}$, and occasionally both at once, i.e., the set of variables $\{t_1, t_2,\cdots, t_n, x_1,x_2, \cdots, x_m\}$. As we explain below, the reader should think of the variables $t_i$ as being associated to the torus $T = (\C^*)^n$ that acts, and the $x_1, \cdots, x_m$ as representing classes in the ordinary cohomology of the space $X$. Thus the notation of the variables serves as a reminder of the manner in which the associated cohomology classes arise.

With the above remarks in mind, we let $T$ denote an algebraic torus of rank $n$, so $T \cong (\C^*)^n$. Let $ET \rightarrow BT$ be a universal principal $T$-bundle;  in particular, 
$ET$ is a contractible topological space with free (right) $T$-action and $BT:=ET/T$. 
In fact, we may think of $BT$ as $(\C P^{\infty})^n$. 
Hence
$$
H^*(BT) \cong \bigotimes_{n} H^*(\C P^{\infty}) \cong \Q[t_1, t_2, \cdots, t_n]
$$
 is a polynomial ring in $n$ variables $t_1, t_2, \cdots, t_n$, each of (cohomology) degree 2. 
Next, let $X$ be a topological space, equipped with a continuous (left) action $T \times X \to X$ by the torus $T$. 
For simplicity, henceforth we will always assume that $X$ is locally contractible, compact, and Hausdorff. (These assumptions are satisfied by all the Hessenberg varieties considered in this manuscript.) 
The $T$-equivariant cohomology $H^\ast_T(X)$ is defined to be the ordinary cohomology of the space $ET \times_T X$, i.e., 
$$
H^\ast_T(X):=H^\ast(ET \times_T X)
$$
where $ET \times_T X$ denotes the orbit (quotient) space of $ET \times X$ by the (left) $T$-action defined by $g \cdot (u, x):=(u \cdot g^{-1},g \cdot x)$ for $(u,x) \in ET \times X$ and $g \in T$.
In particular, if $X=\{\pt \}$ is a point, then $ET \times_T \{\pt \} \cong BT \times \{\pt \}$ and hence we have 
$$
H^\ast_T(\pt) \cong H^\ast(BT) \cong \Q[t_1, t_2, \cdots, t_n]. 
$$ 
More generally, if the $T$-action on $X$ is trivial, then $H^\ast_T(X) = H^\ast(BT) \otimes H^*(X)$.

We note that the $T$-action on $X$ preserves each path-connected component of $X$ since $T$ is connected. This implies that if $X = \sqcup_i X_i$ is the decomposition of $X$ into path-connected components then the $T$-equivariant cohomology $H^*_T(X)\cong \bigoplus_i H^*_T(X_i)$.  In what follows, we assume for simplicity that $X$ is path-connected.
As noted above, the $T$-action on the total space $ET$ of the universal $T$-bundle is free. Thus, the natural projection $ET \times_T X \to ET/T = BT$ to the left factor yields a fibration, and the fibers are isomorphic to $X$. This projection is called the \emph{Borel fibration} associated to $X$ and its $T$-action.  Using this Borel fibration and the inclusion $X \into ET \times_T X$ of a fiber into the total space, 
we obtain the following ring homomorphisms:
$$
H^*(BT) \rightarrow H^*_T(X) \rightarrow H^*(X).
$$
In particular, the ring homomorphism $H^\ast(BT) \rightarrow H^\ast_T(X)$ equips the equivariant cohomology ring $H^*_T(X)$ 
with the structure of a $H^*(BT)$-algebra.

It turns out that, in certain situations, the structure of $H^*_T(X)$ with respect to the homomorphisms $H^*(BT) \to H^*_T(X)$ and $H^*_T(X) \to H^*(X)$ is particularly simple. 
One such situation is when the ordinary cohomology groups of $X$ vanish in the odd degrees. 
This is a reasonable case for us to consider, since by \cite{ty} it is known that type-A Hessenberg varieties are 
paved by complex affine spaces, which in particular implies that the odd-degree cohomology is zero. 
In general when $X$ is a $T$-space satisfying $H^{\mathrm{odd}}(X) = 0$, it is known that the Serre spectral sequence of the fibration
$ET \times_T X \to BT$ collapses, and this in turn implies that  
\begin{equation} \label{eq:free module over H(BT)}
H^\ast_T(X) \cong H^\ast(BT) \otimes_{\Q} H^*(X) \ \ \ \text{ as } H^*(BT) \text{-modules.}
\end{equation}
(It is worth emphasizing that the isomorphism ~\eqref{eq:free module over H(BT)} holds only as $H^*(BT)$-modules, and \emph{not} as rings.) 
In fact, even more is true: in this situation, 
the restriction map $H^*_T(X) \rightarrow H^*(X)$ induced from the Borel fibration is surjective, 
and its kernel is the ideal generated by $H^{>0}(BT)$ (see e.g. \cite[Chapter~III, Theorem~4.2]{MiTo}).
Hence, we can conclude that there is an isomorphism
\begin{equation} \label{eq:from equivariant to ordinary}
H^*(X) \cong H^*_T(X)/(H^{>0}(BT))
\end{equation}
of rings. (If $X$ is not path-connected, then as noted above, $H^*_T(X)$ is a direct sum of $H^*_T(X_i)$ for its (finitely many) path-connected components, and the equation~\eqref{eq:from equivariant to ordinary} holds for each individual summand in the direct sum.) 

Next, we turn to another fundamental and useful technique in torus-equivariant topology: namely,  the restriction map to the fixed point set of the torus action. 
Let $X^T$ denote the set of $T$-fixed points of $X$. Then the inclusion $X^T \into X$ induces the restriction map to the fixed points,
namely, 
\begin{equation} \label{eq:localization map}
\iota: H^*_T(X) \rightarrow H^*_T(X^T). 
\end{equation}
The following major result is called the \textbf{localization theorem} (see e.g. \cite[Chapter III, Section 2]{Hsi}, \cite[Section 1]{ChangSkjelbred}). 
Before saying anything further, we emphasize one point which can be quite confusing for a beginner. Namely, there are two notions of localization which appear in the theorem (and consequent discussion) below, and they are \emph{different}. One of them is the so-called ``(equivariant) localization to fixed points'' which has already appeared numerous times above, namely, the restriction homomorphism~\eqref{eq:localization map} which is injective under suitable circumstances. The other notion of localization which appears below is the purely ring-theoretic concept of ``localization with respect to a multiplicative subset (of an ambient ring)'', which we denote below by $\mathcal{S} \subset R$ (here $\mathcal{S}$ is the multiplicative subset, $R$ is the ambient ring). For the novice, it does not help matters that \emph{both} concepts are used in one of the most famous theorems of equivariant topology, namely Theorem~\ref{theorem:localization_theorem}, below. We content ourselves here to simply warn the reader of this clash in nomenclature.

\begin{Theorem}[``The localization theorem''] \label{theorem:localization_theorem}
Let $T$ be a torus and let $X$ be a locally contractible compact Hausdorff $T$-space. Let $\mathcal{S}:=H^*(BT) \setminus \{0 \}$, a multiplicatively closed set. 
Then the localization of the homomorphism in \eqref{eq:localization map} with respect to $\mathcal{S}$ yields the isomorphism
\begin{equation*} 
\mathcal{S}^{-1}\iota: \mathcal{S}^{-1}H^*_T(X) \stackrel{\cong}{\rightarrow} \mathcal{S}^{-1}H^*_T(X^T).
\end{equation*}
\end{Theorem}
In other words, up to localization by $\mathcal{S}$, the $T$-fixed point set captures the $T$-equivariant cohomology of $X$. 
The following is a well-known and often useful consequence of the localization theorem.

\begin{Corollary} \label{corollary:localization_theorem}
If the odd degree (ordinary) cohomology groups of $X$ vanish, then the restriction map in \eqref{eq:localization map} is injective.
\end{Corollary}

\begin{proof} 
First observe that we have the following commutative diagram relating the 
different homomorphisms under consideration: 
\begin{equation*}
\begin{CD}
\mathcal{S}^{-1}H^*_T(X)@>{\mathcal{S}^{-1}\iota}>{\cong}> \mathcal{S}^{-1}H^*_T(X^T)\\
@A{}AA @A{}AA\\
H^*_T(X)@>{\iota}>> H^*_T(X^T)
\end{CD}
\end{equation*}
where the vertical maps are the usual ones taking an algebra to its localization, and the horizontal maps are the ones appearing in~\eqref{eq:localization map} and Theorem~\ref{theorem:localization_theorem}. 
By assumption, the odd-degree ordinary cohomology of $X$ vanishes, so we have~\eqref{eq:free module over H(BT)} and 
in particular, $H^*_T(X)$ is a free $H^*(BT)$-module. 
This implies that the left vertical map in the commutative diagram above is injective. 
This fact, together with the commutative diagram above, then imply that
 the restriction map $\iota :H^*_T(X)\rightarrow H^*_T(X^T)$ must also be injective. 
\end{proof}

Next, we introduce concepts involving $T$-equivariant vector bundles over $T$-spaces.  
Let $E \rightarrow X$ denote a complex $T$-equivariant vector bundle. 
Since the total space $E$ is a $T$-space, we may apply the Borel construction to both $E$ and $X$ and obtain 
the projection 
\begin{equation}\label{eq: assoc vec bdle}
ET \times_T E \rightarrow ET \times_T X.
\end{equation}
It is not hard to see that this is itself a vector bundle.

The Chern classes of the vector bundle~\eqref{eq: assoc vec bdle}, which by definition 
are elements of $H^*(ET \times_T X) = H^*_T(X)$ (so in particular, they are equivariant cohomology classes), 
are called the $T$-\textbf{equivariant Chern classes} of $E$. They should be thought of as the equivariant analogues 
of the ordinary Chern classes of the vector bundle $E \to X$, which lie in the ordinary cohomology of $X$. 
We denote by $c_i^T(E) \in H^{2k}_T(X)$ the $i$-th equivariant Chern class, i.e., the $i$-th ordinary Chern class of~\eqref{eq: assoc vec bdle}. 
The restriction map $H^*_T(X) \rightarrow H^*(X)$ sends the equivariant Chern class of $E$ to the ordinary Chern class of $E$.

Certain $T$-equivariant Chern classes will be useful for the later discussion, for which we take a moment to set some notation. Let $\alpha: T \rightarrow \C^*$ be an algebraic group homomorphism, which naturally 
defines a one-dimensional representation of $T$, denoted $\C_{\alpha}$.  Any $T$-representation can be regarded as a $T$-equivariant bundle over a point. Thus we can consider the equivariant first Chern class $c_1^T(\C_\alpha) \in H^*(BT)$ of $\C_{\alpha} \to \mathrm{pt}$. It is known that the correspondence $\alpha \mapsto c_1^T(\C_\alpha)$ yields an isomorphism of groups from the lattice
$\text{Hom}(T,\C^*)$ to $H^2(BT;\Z)$. 
Moreover, it is also well-known that there is a $\Z$-module embedding $\text{Hom}(T,\C^*) \hookrightarrow \mathfrak{t}^*$ in its dual Lie algebra given by $\alpha \mapsto d \alpha_e$, where $d$ denotes the exterior derivative. We define the lattice $\mathfrak{t}^*_\Z$ as the image of the embedding $\text{Hom}(T,\C^*) \hookrightarrow \mathfrak{t}^*$. 
With the previous discussion in place, throughout this paper we use the association $\alpha \mapsto c_1^T(\C_{\alpha})$ to identify these three lattices:  
\begin{equation} \label{eq:equiv coh of pt degree two}
\text{Hom}(T,\C^*) \cong \mathfrak{t}^*_\Z \cong H^2(BT;\Z), \quad \alpha \mapsto c_1^T(\C_\alpha). 
\end{equation}
This in turn yields the following identification
\begin{align}\label{eq:equiv coh of pt}
 H^*(BT) \cong \text{Sym}_{\Q}(\mathfrak{t}^*_\Z \otimes_{\Z}\Q) 
\end{align}
where the right hand side denotes the symmetric algebra, over $\Q$, of the $n$-dimensional $\Q$-vector space $\mathfrak{t}^*_\Z \otimes_{\Z}\Q$.

Finally, we wrap up this introductory section with the following simple but useful fact. Note that we continue to assume that 
the odd-degree ordinary cohomology of $X$ vanishes.

\begin{Proposition} \label{proposition:generators equivariant cohomology}
Let $X$ be a path-connected $T$-space. 
Assume that the odd-degree (ordinary) cohomology groups of $X$ vanish. 
Let $m$ be a positive integer and let $S=\{x_1,\ldots, x_m \}$ in $H^*(X)$ be a set of $m$ homogeneous-degree elements
 and let $\widetilde{S}=\{\widetilde{x_1},\ldots, \widetilde{x_m} \}$ denote a choice of lift of $S$ to $H^*_T(X)$ consisting of homogeneous-degree elements, i.e., for all $k$, $1 \leq k \leq m$, the image of $\widetilde{x_k}$ is $x_k$ under the surjection $H^*_T(X) \to H^*(X)$ induced from the Borel fibration, and each $\widetilde{x_k} \in H^*_T(X)$ is a homogeneous-degree element. 
\begin{enumerate}
\item Let $t_1, \cdots, t_n$ be a $\Q$-vector space basis of $H^2(BT) \cong \mathfrak{t}^*_{\Z} \otimes_{\Z} \Q$. If $S$ generates the cohomology ring $H^*(X)$ as a $\Q$-algebra, then $\widetilde{S} \cup \{t_1,\ldots,t_n \}$ generates $H^*_T(X)$ as a $\Q$-algebra. 
\item If $S$ is an additive $\Q$-vector space basis of the cohomology $H^*(X)$, then $\widetilde{S}$ is a basis of $H^*_T(X)$ as an $H^*(BT)$-module.
\end{enumerate}
\end{Proposition}

\begin{proof}
We begin with (1). 
Since we assume the odd-degree ordinary cohomology of $X$ vanishes, and because $H^*_T(\mathrm{pt})$ is generated in even degrees, we know that $H^*_T(X)$ also vanishes in odd degree.  To show that $\widetilde{S} \cup \{t_1,\cdots, t_n\}$ generate $H^*_T(X)$ as a $\Q$-algebra, it would suffice to show that any homogeneous element can be written as a polynomial in these elements with coefficients in $\Q$. 
We will argue by induction on the degree of the homogeneous element. 

Suppose $f$ is homogeneous and suppose $\mathrm{deg}(f) = 2p$ for some $p \geq 0$. 
We begin with the base case, which is $p=0$. 
By \eqref{eq:free module over H(BT)} we have $H^0_T(X) \cong H^0(BT) \otimes_{\Q} H^0(X) \cong H^0(X)$ since $H^0(BT) =\Q$.
Thus, we see that the vector space $H^0_T(X)$ is $1$-dimensional and consists of constant multiples of the identity element $1 \in H^0_T(X)$. 
Now a constant multiple of $1$ can certainly be written as a polynomial in the $\tilde{S} \cup \{t_1,\cdots, t_n\}$ by taking all coefficients of all non-constant monomials in the $\tilde{S} \cup \{t_1,\cdots, t_n\}$ to be $0$ and by taking the appropriate constant multiple of $1$ as the constant term in the polynomial.  This completes the base case. 

Now suppose by induction that $\tilde{S} \cup \{t_1,\cdots, t_n\}$ generate the elements in $H^*_T(X)$ of degree $<2p$ for some integer $p$. We need to show that they also generate $H^{2p}_T(X)$.  
We know that, since the odd degree cohomology of $X$ vanishes, the induced map $i^*: H^*_T(X) \to H^*(X)$ from the Borel fibration $X \xrightarrow{i} ET \times _T X \xrightarrow{\pi} BT$ is surjective, so there exists a section 
$s: H^*(X) \to H^*_T(X)$ of the surjection $i^*$, i.e., the map $s$ satisfies $i^* \circ s = id$.
This means that $s(x_1^{\ell_1}x_2^{\ell_2}\cdots x_m^{\ell_m}) - \widetilde{x_1}^{\ell_1}\widetilde{x_2}^{\ell_2}\cdots\widetilde{x_m}^{\ell_m} \in \ker i^* = (t_1,\ldots, t_n) \cdot H^{\ast}_{T}(X)$.
Moreover, the Leray-Hirsch theorem for the Borel fibration yields the isomorphism 
\begin{align*}
H^*(BT) \otimes_\Q H^*(X) \cong H^*_T(X); \ \ \ \sum \alpha \otimes \beta \mapsto \sum \pi^*(\alpha) \cdot s(\beta)
\end{align*}
as $H^*(BT)$-modules. 
Using this isomorphism, we can see that any element of $H^{2p}_T(X)$ must be a linear combination of (the images under the Leray-Hirsch isomorphism of the) elements in $H^{2k}(BT) \otimes_\Q H^{2\ell}(X)$ for $i+j=p$. This means that any element $f \in H^{2p}_{T}(X)$ is of the form 
$$
f = \sum_{K=(k_1,k_2,\ldots,k_m) \atop L=(\ell_1,\ell_2,\ldots,\ell_n)} a_{K,L} \, t_1^{\ell_1}t_2^{\ell_2}\cdots t_n^{\ell_n} \cdot s \left(x_1^{k_1}x_2^{k_2}\cdots x_m^{k_m} \right)
$$ 
for some $a_{K,L} \in \Q$ where $\lvert K \rvert :=\sum k_{s}$, $\lvert L \rvert :=\sum \ell_{s'}$, and $p=\lvert K \rvert + \lvert L \rvert$.
Here we have used that $S = \{x_1, \cdots, x_m\}$ generate $H^*(X)$. 
Since we have $s(x_1^{k_1}x_2^{k_2}\cdots x_m^{k_m}) \in \widetilde{x_1}^{k_1}\widetilde{x_2}^{k_2}\cdots\widetilde{x_m}^{k_m} + (t_1,\ldots, t_n) \cdot H^{\ast}_{T}(X)$, the equality above can written as
\begin{align*}
f = \sum_{K=(k_1,k_2,\ldots,k_m) \atop L=(\ell_1,\ell_2,\ldots,\ell_n)} a_{K,L} \, t_1^{\ell_1}t_2^{\ell_2}\cdots t_n^{\ell_n} \cdot \left( \widetilde{x_1}^{k_1}\widetilde{x_2}^{k_2}\cdots\widetilde{x_m}^{k_m} + t_1 g_1^K +\cdots + t_n g_n^K \right)
\end{align*}
for some $g_1^K, \ldots, g_n^K \in H^{2(\lvert K \rvert -1)}_{T}(X)$. 
Since their degrees are less than $2p$, by the inductive assumption, each $g_1^K, \ldots, g_n^K$ can be written as a homogeneous polynomial in the elements $\widetilde{x_1},\ldots, \widetilde{x_m}$ and $t_1,\dots ,t_n$. 
Therefore, the expression of $f$ above shows that $f$ can also be written as a homogeneous polynomial in the elements $\widetilde{x_1},\ldots, \widetilde{x_m}$ and $t_1,\dots ,t_n$, as desired. This concludes the inductive step, and the claim follows. 

(2) Since the odd degree cohomology groups of $X$ vanish, \eqref{eq:free module over H(BT)} holds.
Comparing the Hilbert series of both sides of \eqref{eq:free module over H(BT)}, we obtain
\begin{align*}
\Hilb(H^*_T(X),q)&=\Hilb(H^*(BT),q) \cdot \Hilb(H^*(X),q) \\
&=\frac{1}{(1-q^2)^n} \cdot \Hilb(H^*(X),q)
\end{align*}
where we use the fact that 
 $H^*(BT)$ is the polynomial ring $\Q[t_1,\ldots,t_n]$ with $\deg t_i=2$ for all $i \in [n]$ to compute its Hilbert series. 
Noting that $H^*_T(X)/(t_1,\ldots,t_n) \cong H^*(X)$, it follows from Lemma~\ref{lemma: hilbert regular criterion} that $t_1,\ldots,t_n$ is a regular sequence in $H^*_T(X)$.
Thus, by the observation made in Remark~\ref{remark:regular sequence basis},  if $S$ is an additive basis of $H^*(X)$, then $\widetilde{S}$ forms an  $H^*(BT)$-basis for $H^*_T(X)$, as desired. 
\end{proof}

\section{ A basic example: the equivariant cohomology of flag varieties } 
\label{section: flag}

In this section, in order to illustrate some of the tools introduced in Section~\ref{sec: background},
 we briefly sketch a computation of the equivariant and ordinary cohomology rings of the classical flag variety. 
The arguments are illuminating and suggestive in part because, as it will turn out, the equivariant cohomology ring of the classical flag variety can be specified in a purely combinatorial fashion. 

In this section, we take $T$ to be the rank-$n$ torus consisting of diagonal $n \times n$ matrices in $GL_n(\C)$, so $T \cong (\C^*)^n$.  
It will be convenient to denote such a diagonal element by listing only its diagonal entries. More specifically, we let $\diag(g_1, \cdots, g_n) \in T$ denote a diagonal matrix with entries $g_1, \cdots, g_n$, where $g_i \in \C^*$ for all $i$. 
For any $i$ with $1 \leq i \leq n$, let $\chi_i$ denote the projection from $T$ to the $i$-th factor of $T$, i.e. 
$$ \chi_i: T\rightarrow \C^*, \quad \diag(g_1,\dots,g_n)\mapsto g_i.$$ 
Then $\chi_i$ is evidently a homomorphism from $T$ to $\C^*$.  We associate to it an equivariant first Chern class $c_1^T((\C_{\chi_i}^*)) \in H^*(BT)$, 
which we denote by $t_i$, i.e., 
\begin{align}\label{deg of ti in coh}
t_i := c_1^T((\C_{\chi_i}^*)) \in H^2(BT).
\end{align} 
(Note that because we take the dual representation $\C_{\chi_i}^*$ instead of $\C_{\chi_i}$, we have that $t_i$ corresponds to the group homomorphism $\chi_i^{-1}: T \to \C^*$ which takes $(g_1, g_2, \cdots, g_n)$ to  $g_i^{-1}$ for all $i$, i.e., there is a negative sign that appears in the exponent.)
Recall also that the elements $t_i$ form an additive basis of $H^2(BT)$, and we have $H^*(BT) \cong \Q[t_1,\dots,t_n]$ (compare with \eqref{eq:equiv coh of pt}).

Note that the torus $T$ of invertible diagonal matrices is a subgroup of the Borel subgroup $B$ of invertible upper-triangular matrices in $GL_n(\C)$. Recall that the flag variety $\fln$ can be identified with the set of cosets $GL_n(\C)/B$. From this it follows that $T$ acts on $\fln \cong GL_n(\C)/B$ by left multiplication of cosets. It is well-known (and not hard to compute directly) that the set $\fln^{\Tn}$ of fixed points in $\fln$ under this $T$-action is the (finite) set of permutation flags $\{w E_{\bullet}\}_{w \in S_n}$ where $E_{\bullet}$ is the standard flag $E_\bullet = \{ E_i \}$ with $E_i := \mathrm{span}\{e_1, e_2, \cdots, e_i\}$ for $1 \leq i \leq n$. Here, $\{e_1,e_2,\cdots,e_n\}$ is the standard basis of $\C^n$. 
More specifically, if we write $w E_{\bullet}=(V_i)_{1 \leq i \leq n}$, then $wE_{\bullet}$ satisfies 
\begin{equation} \label{eq:2.3}
V_i= \langle e_{w(1)}, e_{w(2)}, \cdots ,e_{w(i)}\rangle\, \,  \text{ for $i \in [n]$}. 
\end{equation}
In what follows, we sometimes denote a permutation flag $wE_{\bullet}$ by $w$ for simplicity.

Based on the above, we see that in the context of $X=\fln$, the restriction map~\eqref{eq:localization map} takes the form  
\begin{equation} \label{eq:restriction map flag}
\iota_1: H^*_T(\fln) \hookrightarrow \bigoplus_{w \in S_n} H^*_T(w) \cong \bigoplus_{w \in S_n} \Q[t_1,\dots,t_n]
\end{equation}
where the injectivity follows from Corollary~\ref{corollary:localization_theorem} because the flag variety $\fln$ is paved by affines (the cells of which are the famous \emph{Schubert cells}) and hence the odd degree ordinary cohomology of $\fln$ vanishes. 
Here and in what follows, for an element $f \in H^*_T(\fln)$, we denote the $w$-th component in $\Q[t_1,\dots,t_n]$ of the image of $f$ under the restriction map $\iota_1$ by $f|_w$ or $f(w)$. The following lemma computes the restrictions of the classes $t_i := c_1^T((\C_{\chi_i}^*))$ to the $T$-fixed points, where here we abuse notation and use $t_i$ to also denote the image in $H^2_T(\fln)$ of $t_i \in H^2(BT)$ via the pullback map $H^*(BT) \to H^*_T(\fln)$.  

\begin{Lemma} \label{lemma:flag0.5}
Let $i \in [n]$. Consider the class $t_i = c_1^T((\C_{\chi_i}^*)) \in H^*(BT)$ as a class in $H^*_T(\fln)$. 
Then we have $t_i|_w=t_{i}$ for all $w \in S_n$. 
\end{Lemma}

\begin{proof}
Since the restriction map $\iota_1$ is induced from a $T$-equivariant inclusion map, $\{\pt \} \hookrightarrow \fln$, by standard naturality properties of cohomology, we know the homomorphism $\iota_1$ is a homomorphism as $H^*(BT)$-algebras.
Thus we may compute 
\begin{align*}
\iota_1(t_i)=t_i \cdot \iota_1(1)=t_i \cdot (1)_{w \in S_n} = (t_i)_{w \in S_n} \ \ \ \text{ in } \bigoplus_{w \in S_n} \Q[t_1,\dots,t_n].
\end{align*}
This means the $w$-th component of $\iota_1(t_i)$ is exactly $t_i$ for all $w \in S_n$.
\end{proof}

We now briefly recall some standard constructions on the ambient
space $\fln$ which will lead to 
a well-known ring
presentation for the equivariant cohomology of $\fln$. Let $\mathcal{E}_i$ denote the $i$-th tautological vector
bundle over $\fln$: more specifically, $\mathcal{E}_i$ is the sub-bundle of the
trivial vector bundle $\fln \times \C^n$ over $\fln$
whose fiber over a point $V_{\bullet}=(V_1\subset \cdots\subset V_n) \in
\fln$ is exactly $V_i$.  The quotient line bundles $\mathcal{E}_i/\mathcal{E}_{i-1}$ over $\fln$ for $1 \leq i \leq n$ are called the tautological line bundles over $\fln$. We let 
\begin{align}\label{def of T eq ch in flag}
\TChFlag_i  := c_1^T((\mathcal{E}_i/\mathcal{E}_{i-1})^*) = - c_1^T(\mathcal{E}_i/\mathcal{E}_{i-1}) \in H^2_T(\fln)
\end{align}
denote the $T$-equivariant first Chern class of the dual\footnote{For the purpose of describing an explicit presentation of the (equivariant) cohomology ring of $\fln$, it is possible to not take duals in the two definitions~\eqref{deg of ti in coh} and~\eqref{def of T eq ch in flag} (e.g. \cite{AHHM2019, FukukawaIshidaMasuda}). However, later in this paper we discuss (equivariant) Schubert classes, so in order to achieve an internal compatibility of conventions taken within this paper, here we take the dual. } of the tautological line bundle $(\mathcal{E}_i/\mathcal{E}_{i-1})^*$.

\begin{Lemma} \label{lemma:flag1}
Let $i \in [n]$.
Then $\TChFlag_i|_w=t_{w(i)}$ for any $w \in S_n$.
\end{Lemma}

\begin{proof}
For $w \in S_n$, let $j_w$ be the $T$-equivariant inclusion map from $\{w \}$ to $\fln$.
Considering the restriction  $j_w^*(\mathcal{E}_i/\mathcal{E}_{i-1}) = (\mathcal{E}_i/\mathcal{E}_{i-1})|_w$ of the $i$-th tautological line bundle 
to the point $\{w\}$, we have the commutative diagram 
\begin{equation*}
\begin{CD}
(\mathcal{E}_i/\mathcal{E}_{i-1})|_w@>>> \mathcal{E}_i/\mathcal{E}_{i-1} \\
@V{}VV @V{}VV\\
\{w\}@>{j_w}>> \fln
\end{CD}
\end{equation*}
where the horizontal arrows are inclusion maps of spaces and the vertical arrows are the projection maps of vector bundles.
From the naturality of Chern classes (e.g. \cite[Lemma~14.2]{MS}) we have 
\begin{equation} \label{eq:lemma:flag1-1}
\TChFlag_i|_w=-c_1^T(\mathcal{E}_i/\mathcal{E}_{i-1})|_w=-c_1^T((\mathcal{E}_i/\mathcal{E}_{i-1})|_w).
\end{equation}
But, by the definition of the tautological bundles $\mathcal{E}_i$, the line bundle $(\mathcal{E}_i/\mathcal{E}_{i-1})|_w$ is isomorphic to the one-dimensional $T$-representaion 
$$
\langle e_{w(1)}, e_{w(2)}, \cdots ,e_{w(i)}\rangle/\langle e_{w(1)}, e_{w(2)}, \cdots ,e_{w(i)-1}\rangle.
$$ 
One can easily see that for any $z \in \langle e_{w(1)}, e_{w(2)}, \cdots ,e_{w(i)}\rangle/\langle e_{w(1)}, e_{w(2)}, \cdots ,e_{w(i)-1}\rangle$ and $g \in T$, we have 
$$
g \cdot z = g_{w(i)}z=\chi_{w(i)}(g) z
$$
implying that 
\begin{equation} \label{eq:lemma:flag1-2}
c_1^T((\mathcal{E}_i/\mathcal{E}_{i-1})|_w)=c_1^T(\C_{\chi_{w(i)}})=-c_1^T((\C_{\chi_{w(i)}})^*) = -t_{w(i)}.
\end{equation}
Therefore, we conclude from \eqref{eq:lemma:flag1-1} and \eqref{eq:lemma:flag1-2} that $\TChFlag_i|_w=t_{w(i)}$, as desired.
\end{proof}

Let $i \in [n]$.  We let $\mathfrak{e}_i$ denote the $i$-th elementary symmetric polynomial in the input variables. For example, $\mathfrak{e}_i(t_1,\dots ,t_n)$ is of degree $2i$ since $\deg t_i=2$ for all $i \in [n]$. We have the following.

\begin{Proposition} \label{proposition:flag1}
For any $i\in[n]$, we have 
$$
\mathfrak{e}_i(\TChFlag_1,\dots ,\TChFlag_n)-\mathfrak{e}_i(t_1,\dots ,t_n) = 0
$$ 
in the equivariant cohomology $H^{\ast}_{T}(\fln)$.

\end{Proposition}

\begin{proof}
For all $w \in S_n$, we obtain
\begin{align*}
\mathfrak{e}_i(\TChFlag_1,\dots ,\TChFlag_n)|_w&=\mathfrak{e}_i(\TChFlag_1|_w,\dots ,\TChFlag_n|_w) \\ 
                                                 &=\mathfrak{e}_i(t_{w(1)},\dots ,t_{w(n)}) \ \ \ \text{by Lemma~\ref{lemma:flag1}}\\
                                                 &=\mathfrak{e}_i(t_1,\dots ,t_n) \ \ \ \text{since } \mathfrak{e}_i \text{ is symmetric} \\
                                                 &=\mathfrak{e}_i(t_1,\dots ,t_n)|_w \ \ \ \text{by Lemma~\ref{lemma:flag0.5}}.
\end{align*}
Hence, the image of $\mathfrak{e}_i(\TChFlag_1,\dots ,\TChFlag_n)-\mathfrak{e}_i(t_1,\dots ,t_n)$ under the restriction map $\iota_1$ in \eqref{eq:restriction map flag} vanishes. 
By the injectivity of $\iota_1$ we obtain $\mathfrak{e}_i(\TChFlag_1,\dots ,\TChFlag_n)-\mathfrak{e}_i(t_1,\dots ,t_n) = 0
$ in $H^{\ast}_{T}(\fln)$, as claimed. 
\end{proof}

By Proposition~\ref{proposition:flag1}, the homomorphism
\begin{equation} \label{cohomology flag}
\varphi: \mathbb{Q}[x_1,\dots,x_n,t_1,\dots,t_n]/(\mathfrak{e}_i(x_1,\dots ,x_n)-\mathfrak{e}_i(t_1,\dots ,t_n) \mid i\in[n]) \rightarrow H^{\ast}_{T}(\fln)
\end{equation}
which sends $x_i$ to $\TChFlag_i$ and $t_i$ to $t_i$ for each $i \in [n]$ is well-defined. 
Below, we sketch the proof that $\varphi$ is an isomorphism, which means that $\varphi$ is the desired generators-and-relations description of $H^{\ast}_T(\fln)$ as a ring.  Giving this sketch of proof allows us to illustrate several of the 
techniques introduced in Section~\ref{sec: background}; some details are omitted for brevity. 

First, we prove that $H^{\ast}_{T}(\fln)$ is generated as a ring by the elements $\TChFlag_1,\dots ,\TChFlag_n, t_1, \cdots, t_n$. 
In other words, we aim to prove that $\varphi$ is surjective. 
To see this, we need some facts about the ordinary cohomology $H^*(\fln)$ of $\fln$. 
Let
\begin{align}\label{def of ch in flag}
\ChFlag_i \in H^2(\fln)
\end{align}
be the (ordinary, i.e., non-equivariant) first Chern class of the dual of the tautological line bundle $(\mathcal{E}_i/\mathcal{E}_{i-1})^*$, i.e. $\ChFlag_i=c_1((\mathcal{E}_i/\mathcal{E}_{i-1})^*) = 
-c_1(\mathcal{E}_i/\mathcal{E}_{i-1}) \in H^2(\fln)$. 
Basic facts about projective bundles and Chern classes together with the construction of the flag variety $\fln$ as a sequence of projective bundles imply the following (see e.g. \cite[Section~10.2, Proof of Proposition~3]{Fulton-Young} for details). 

\begin{Proposition} \label{proposition:flag generators non-equivariant}
The cohomology ring $H^{\ast}(\fln)$ is generated by $\ChFlag_1,\dots ,\ChFlag_n$ as a $\Q$-algebra. 
Moreover, we can take a set of monomials $\{\ChFlag_1^{\ell_1}\ChFlag_2^{\ell_2}\cdots\ChFlag_n^{\ell_n} \mid 0 \leq \ell_j \leq n-j \}$ as an additive basis for $H^*(\fln)$.
In particular, the Hilbert series of the ordinary cohomology ring $H^{\ast}(\fln)$ (also called the Poincar\'e polynomial of $\fln$) is given by
\begin{align*}
\Hilb(H^{\ast}(\fln), q) = \Poin(\fln,q)=\prod_{i=1}^{n-1} (1+q^2+q^4+\cdots+q^{2i}).
\end{align*}
\end{Proposition}

Proposition~\ref{proposition:flag generators non-equivariant} together with Proposition~\ref{proposition:generators equivariant cohomology}-(1) then yields the following result.\footnote{It turns out that this fact can also be proved in a purely combinatorial manner by using the so-called GKM graph of $\fln$ (\cite[Lemma~3.2]{FukukawaIshidaMasuda}). }

\begin{Proposition} \label{proposition:flag generators}
The equivariant cohomology ring $H^{\ast}_{T}(\fln)$ is generated as a $\Q$-algebra by the elements $\TChFlag_1,\dots ,\TChFlag_n$ and $t_1,\dots ,t_n$. 
In particular, the map $\varphi$ in \eqref{cohomology flag} is surjective.
\end{Proposition}

Recall that we want to show that the ring homomorphism $\varphi$ is an isomorphism; the above result shows that $\varphi$ is surjective. We 
also already know that $\varphi$ preserves the gradings on both rings. Thus, in order to prove that $\varphi$ is an isomorphism, it suffices to check that each of the graded pieces of the domain and the codomain of $\varphi$ are equal, or equivalently, it suffices to check 
that the Hilbert series of both rings are equal. 
For the computations below it is useful to recall that 
$\deg x_i = \deg t_i = 2$ for all $i$. 
We first compute the Hilbert series of the equivariant cohomology ring of $\fln$. 

\begin{Lemma} \label{lemma:flag Hilbert series1}
We have $\Hilb(H^{\ast}_{T}(\fln),q)=\frac{1}{(1-q^2)^{n}}\prod_{i=1}^{n-1} (1+q^2+q^4+\cdots+q^{2i})$.
\end{Lemma}

\begin{proof}
By \eqref{eq:free module over H(BT)} we have 
\begin{equation*}
\Hilb(H^{\ast}_{T}(\fln),q)=\frac{1}{(1-q^2)^{n}} \Hilb(H^{\ast}(\fln), q) 
\end{equation*}
so the result is immediate from 
Proposition~\ref{proposition:flag generators non-equivariant}.
\end{proof}

As our second step, we need to compute the Hilbert series of $$\mathbb{Q}[x_1,\dots,x_n,t_1,\dots,t_n]/(\mathfrak{e}_i(x_1,\dots ,x_n)-\mathfrak{e}_i(t_1,\dots ,t_n) \mid i\in[n]).$$
For notational simplicity, we write 
\begin{equation*}
\f_i:=\mathfrak{e}_i(x_1,\dots ,x_n)-\mathfrak{e}_i(t_1,\dots ,t_n)
\end{equation*}
for $1\le i\le n$ for the rest of this section.

\begin{Lemma} \label{lemma:flag regular sequences2}
The sequence $\f_1,\ldots,\f_n,t_1,\ldots,t_n$ is 
a regular sequence in the polynomial ring $\mathbb{Q}[x_1,\dots,x_n,t_1,\dots,t_n]$.
\end{Lemma}

\begin{proof}  
We compute 
\[
\begin{split}
&\Hilb(\Q[x_1,\dots,x_n,t_1,\dots,t_n]/(\f_1,\dots,\f_n,t_1,\dots,t_n),q)\\
=&\Hilb(\Q[x_1,\dots,x_n]/(\mathfrak{e}_1(x_1,\dots ,x_n),\dots,\mathfrak{e}_n(x_1,\dots ,x_n)),q)\\
=&\prod_{i=1}^{n-1} (1+q^2+q^4+\cdots+q^{2i}) \ \ \ 
\end{split}
\]
where the last equality was obtained in Example~\ref{example:flag regular sequences1}. 
This implies that \eqref{eq:4.6} is satisfied in our setting because $\deg \f_i=2i$ for $1\le i\le n$, $\deg t_j=2$ for $1\le j\le n$ and 
\begin{equation*} 
\Hilb(\Q[x_1,\dots,x_n,t_1,\dots,t_n],q)=\frac{1}{(1-q^2)^{2n}}. 
\end{equation*} 
The claim now follows from Lemma~\ref{lemma: hilbert regular criterion}. 
\end{proof}

Now we can compute the Hilbert series we need.

\begin{Lemma} \label{lemma:flag Hilbert series2}
We have $\Hilb(\Q[x_1,\dots,x_n,t_1,\dots,t_n]/(\f_1,\dots,\f_n),q)=\frac{1}{(1-q^2)^{n}}\prod_{i=1}^{n-1} (1+q^2+q^4+\cdots+q^{2i})$.
\end{Lemma}

\begin{proof}
From Lemma~\ref{lemma:flag regular sequences2} we know that $\f_1, \cdots, \f_n, t_1, \cdots, t_n$ is a regular sequence, and from the definition of a regular sequence it follows that the subsequence $\f_1,\dots,\f_n$ is again a regular sequence. Hence it follows from Lemma~\ref{lemma: hilbert regular criterion} that 
\[
\begin{split}
\Hilb(\Q[x_1,\dots,x_n,t_1,\dots,t_n]/(\f_1,\dots,\f_n),q)&=\frac{1}{(1-q^2)^{2n}}\prod_{i=1}^{n}(1-q^{2i})\\
&=\frac{1}{(1-q^2)^{n}}\prod_{i=1}^{n-1} (1+q^2+q^4+\cdots+q^{2i})
\end{split}
\]
as desired. 
\end{proof}

 We can now prove the main claim of this section, namely, that $\varphi$ is an isomorphism; 
 this gives the ring presentation of $H^{\ast}_T(\fln)$. 
 
\begin{Theorem} \label{theorem:presentation flag typeA}
The map $\varphi$ in \eqref{cohomology flag} is an isomorphism as $H^*(BT)$-algebras.
\end{Theorem}

\begin{proof}
By Proposition~\ref{proposition:flag generators} the homomorphism 
$$
\varphi: \Q[x_1,\dots,x_n,t_1,\dots,t_n]/(\f_1,\dots,\f_n) \rightarrow H^{\ast}_{T}(\fln); \ x_i \mapsto \TChFlag_i, t_i \mapsto t_i
$$ 
is surjective. 
However, it follows from Lemmas~\ref{lemma:flag Hilbert series1} and \ref{lemma:flag Hilbert series2} that the domain and the codomain have the same Hilbert series.  Since $\varphi$ is surjective, the only way that the two rings can have the same Hilbert series is if $\varphi$ is an isomorphism. This proves the result. 
\end{proof}

\begin{Remark}
The $H^*(BT)$-algebra homomorphism $\varphi$ is also an isomorphism when we consider the cohomology rings with integral coefficients (see e.g. \cite[Theorem~3.1]{FukukawaIshidaMasuda}). 
\end{Remark}

The discussion above concerns the $T$-equivariant cohomology of $\fln$, but 
we can also describe the ordinary cohomology ring $H^*(\fln)$ by taking the quotient by the ideal 
generated by $H^{>0}(BT)$ as in~\eqref{eq:from equivariant to ordinary}; this can be accomplished by setting the variables $t_i$ equal to $0$. We have therefore proved the following.

\begin{Theorem} \label{theorem:flag cofomology}
There exists a $\Q$-algebra isomorphism
\begin{equation}\label{eq:cohofl}
H^{\ast}(\fln)\cong \Q[x_1,\dots,x_n]/(\mathfrak{e}_i(x_1,\dots ,x_n) \mid i\in[n])
\end{equation}
which sends each $x_i$ to the ordinary (i.e. non-equivariant) first Chern class $\ChFlag_i$ defined in \eqref{def of ch in flag}.
\end{Theorem}

\begin{Remark}
The isomorphism~\eqref{eq:cohofl} remains valid also with integral coefficients (see e.g. \cite[Section~10.2, Proposition~3]{Fulton-Young}). 
\end{Remark}

Our second main theorem for this section is an important result of Sara Billey, often called \textbf{Billey's formula} \cite{Billey}. 
Giving a sketch of the proof is beyond the scope of this manuscript, so we restrict ourselves to a statement of the formula and some discussion. We first recall some context. 
Schubert varieties are certain $T$-invariant subvarieties of $\fln$, indexed by the permutations $w \in S_n$, and their corresponding Schubert (cohomology) classes $\{\sigma_w \in H^{\ast}(\fln)\}_{w \in S_n}$ form an additive basis of the ordinary cohomology $H^*(\fln)$.  Since the Schubert varieties are $T$-invariant, it turns out that we can also consider the associated $T$-equivariant Schubert (cohomology) classes, which we denote as $\sigma^T_w \in H^*_T(\fln)$. Note that the $T$-equivariant Schubert classes $\{\sigma^T_w \in H^{\ast}_T(\fln)\}_{w \in S_n}$ form an $H^*(BT)$-basis of the equivariant cohomology $H^*_T(\fln)$ by Proposition~\ref{proposition:generators equivariant cohomology}-(2).
Billey's formula allows us to explicitly compute the restriction of the equivariant Schubert classes to the fixed points $w \in S_n$. 
To state the result, we need some preliminaries.

\begin{Definition}\label{definition.term.in.billey}
Given a permutation $w$,  an index $i \in \{1,2,\ldots,\ell(w)\}$, and a choice of reduced word 
${\bf b}=(b_1, b_2, \ldots, b_{\ell(w)})$ for $w \in S_n$, we define 
\begin{equation}\label{eq:def-beta}
r(i, \mathbf{b}) := s_{b_1} s_{b_2} \cdots s_{b_{i-1}} (\alpha_{b_i}),
\end{equation}
where $\alpha_k \in H^2(BT) = \Q[t_1, \ldots, t_n]$ for $k \in [n-1]$ is defined as  
\begin{align} \label{eq:simple root in type A}
\alpha_k:=t_{k+1} - t_{k}. 
\end{align}
\end{Definition}

Note that, by definition, $r(i, \mathbf{b})$ is an element of $\Q[t_1, \ldots, t_n]$ of the form $t_j - t_k$ for some $j,k$. 
It is known that $r(i,\mathbf{b})$ is in
fact a \textit{positive} root, namely, it has the form $t_j - t_k$
for $j>k$ (see e.g. \cite[Equation 4.1 and surrounding discussion]{Billey}).  These positive
roots $r(i, \mathbf{b})$ are the building blocks of Billey's formula, which we now state.

\begin{Theorem}\label{theorem:billey} (``Billey's formula'', \cite[Theorem 4]{Billey})
Let $w \in S_n$.  Fix a reduced word decomposition $w = s_{b_1}
s_{b_2} \cdots s_{b_{\ell(w)}}$ and let $\mathbf{b}= (b_1, b_2,
\ldots, b_{\ell(w)})$ be the sequence of its indices. 
Let $v \in S_n$. Then the value of the $T$-equivariant Schubert class $\sigma^T_v$ at
the $T$-fixed point $w$ is given by 
\begin{equation}\label{eq:billey-formula}
\sigma^T_v(w) = \sum r(i_1, \mathbf{b})r(i_2, \mathbf{b}) \cdots
r(i_{\ell(v)}, \mathbf{b})
\end{equation}
where the sum is taken over subwords $s_{b_{i_1}} s_{b_{i_2}} \cdots s_{b_{i_{\ell(v)}}}$ of $\mathbf{b}$ that are reduced words for $v$.
\end{Theorem}

We refer to an individual summand of the
expression~\eqref{eq:billey-formula}, corresponding to a single reduced
subword $v=s_{b_{i_1}} s_{b_{i_2}} \cdots s_{b_{i_{\ell(v)}}}$ of $w$,
 as a {\bf
  summand in Billey's formula}. 
  
\begin{Remark}
Although we restricted the above discussion to the case of $GL_n(\C)$, Billey's formula is in fact stated in \cite[Theorem 4]{Billey}
for general Lie types. We will use this more general version in Section~\ref{section: peterson general Lie type}. 
\end{Remark}

\begin{Example}\label{example: billey upper triangular} 
It is not hard to see from Billey's formula that the  equivariant Schubert classes $\sigma^T_v$ for $v \in S_n$ are ``upper-triangular'' in the following precise sense. Recall that the \textbf{Bruhat (partial) order on $S_n$} is defined as follows: $v \leq w$ if and only if for some (and therefore any) reduced word decomposition for $w$, i.e. $w = s_{b_1}
s_{b_2} \cdots s_{b_{\ell(w)}}$ for some sequence of simple transpositions, then there exists a subword $s_{b_{i_1}} s_{b_{i_2}} \cdots s_{b_{i_{\ell(v)}}}$ which is a reduced word for $v$. Given this definition, it is clear from Theorem~\ref{theorem:billey} that if $v \not \leq w$, then $\sigma^T_v(w)=0$.\footnote{This property also follows from the fact that a permutation flag $w$ belongs to $X^v:=\overline{B_-vB/B}$ if and only if $v \leq w$.} This is the ``upper-triangularity'' phenomenon for $T$-equivariant Schubert classes.  In particular, this means -- for example -- that $\sigma^T_{s_k}(e)=0$ for any simple transposition $s_k$, since $s_k$ is not less than the identity element $e \in S_n$. This will be used in an argument below. 
\end{Example} 

We finish the section with a lemma which we will use later, as well as an example computation.

\begin{Lemma} \label{lemma:equivariant Schubert class simple reflection}
Let $i \in [n-1]$.  Then $\sigma^T_{s_i} = \sum_{k=1}^i (\tau^T_k - t_k)$.  
\end{Lemma}

\begin{proof}
Since we have not formally defined the Schubert classes, what follows is a sketch of the proof. 
The class $\sigma^T_{s_k}$ corresponding to the simple transposition $s_k$ is degree $2$. 
From Theorem~\ref{theorem:presentation flag typeA} we may conclude that there is a unique expression 
$\sigma^T_{s_k} = \sum_{i=1}^{n-1} a_i \tau^T_i + \sum_{j=1}^n b_j t_j$ in $H^2_T(\fln)$. By setting $t_i=0$ for all $i \in [n]$, we then 
have $\sigma_{s_k} = \sum_{i=1}^{n-1} a_i \tau_i$ in $H^2(\fln)$. It is known that $\sigma_{s_k} = \sum_{i=1}^k \tau_i$ in $H^2(\fln)$ (see e.g. \cite[p.171]{Fulton-Young}).  It also follows from Proposition~\ref{proposition:flag generators non-equivariant} that the 
 $\tau_1,\ldots,\tau_{n-1}$ are linearly independent. The last two facts together imply that $a_i=1$ if $i \leq k$ and $a_i=0$ if $i>k$.
Thus we have $\sigma^T_{s_k} = \sum_{i=1}^{k} \tau^T_i + \sum_{j=1}^n b_j t_j$ in $H^2_T(\fln)$. Restricting this class to the identity element, we have $0 = \sum_{i=1}^{k} t_i + \sum_{j=1}^n b_j t_j$ in $\Q[t_1,\ldots,t_n]$ by Lemmas~\ref{lemma:flag0.5} and \ref{lemma:flag1}, and the fact that $\sigma^T_{s_k}(e)=0$ for any $k$, as observed in Example~\ref{example: billey upper triangular} above.
This implies $b_j=-1$ if $j \leq k$ and $b_j=0$ if $j>k$, as desired.
\end{proof}

Using Lemma~\ref{lemma:equivariant Schubert class simple reflection}, we can give a sample explicit computation of Billey's formula. 
Here we use both the one-line notation $w= w(1) \, w(2) \, w(3) \cdots w(n)$ of $w$ as well as its decomposition into simple transpositions. 

\begin{Example}
Let $n=3$ and consider the permutation $w = 3 2 1 = s_1 s_2 s_1$ which satisfies $w(1)=3, w(2)=2, w(3)=1$. 
By Lemmas~\ref{lemma:equivariant Schubert class simple reflection}, \ref{lemma:flag0.5}, and \ref{lemma:flag1} we have 
\begin{align*}
\sigma^T_{s_1}(321) = (\tau^T_1 - t_1)|_{321} = t_3 - t_1 = (t_2 - t_1) + (t_3 - t_2).
\end{align*}
On the other hand, let us fix a reduced expression $321=s_1s_2s_1$ as above, namely $\mathbf{b}=(1,2,1)$.
Then, by using Billey's formula (Theorem~\ref{theorem:billey}) we can also compute
\begin{align*}
\sigma^T_{s_1}(321) &= r(1, \mathbf{b}) + r(3, \mathbf{b}) = \alpha_1 +s_1s_2(\alpha_1) \\
&=(t_2-t_1) +s_1s_2(t_2-t_1) = (t_2 - t_1) + (t_3 - t_2)
\end{align*}
which is the same result as above, obtained in a different way. 
Note that even if we choose a different reduced word for the same permutation, e.g. $321=s_2s_1s_2$, i.e. $\mathbf{b'}=(2,1,2)$, we obtain the same result using Billey's formula 
\begin{align*}
\sigma^T_{s_1}(321) &= r(2, \mathbf{b'}) = s_2(\alpha_1) = s_2(t_2-t_1) = t_3 - t_1
\end{align*}
as we should. 
\end{Example}

\section{ A variation on the theme: the equivariant cohomology rings of Hessenberg varieties} 
\label{section: A variation on the theme: the equivariant cohomology rings of Hessenberg varieties}

In this section, we briefly discuss how the methods outlined in Section~\ref{subsec: background equivariant cohomology} 
and Section~\ref{section: flag} 
can be adjusted appropriately to analyze the case of regular nilpotent Hessenberg varieties.  
We focus on general discussion here; this will lay the groundwork for the detailed analysis given in the later sections. 
 
First, we must specify the group which acts. 
In Section~\ref{section: flag} we considered the torus $T$ of diagonal matrices and its action on $\fln$, but 
this $T$-action does not, in general, preserve a regular nilpotent Hessenberg (sub)variety sitting inside $\fln$.  
It turns out to be useful to consider, instead, the following $1$-dimensional subtorus $\mathsf{S}$ of $\Tn$: 
\begin{equation} \label{eq:def of S}
\mathsf{S}:= \left\{   \diag(g,g^2,\dots,g^n)
\mid g\in\C^* \right\}.
\end{equation}
By using the explicit choice of $\mathsf{N}$ given by
$$
 \mathsf{N}:=\left(
          \begin{array}{ccccc}
           0 & 1& 0 &\cdots & 0 \\
            0&0&1& \cdots &0  \\
            \vdots & \vdots& \vdots &\ddots & \vdots \\
            0& 0& 0&\cdots  & 1\\
            0 & 0& 0&\cdots & 0 \\
         \end{array}
       \right).
$$
It is not hard to check directly that the regular nilpotent Hessenberg variety $\Hess(\mathsf{N},h)$ 
defined by
$$
\Hess(\mathsf{N},h) := \{ V_\bullet = (\{0\} \subset V_1 \subset V_2 \subset \cdots \subset V_{n-1} \subset \C^n \, \mid \, \dim_{\C}(V_i)=i, \forall i\}
$$
 is preserved by $\mathsf{S}$. Therefore,  in our considerations below, we can think about $\Hess(\mathsf{N},h)$ 
equipped with this $\mathsf{S}$-action.

Let $\chi$ denote the group homomorphism $\chi: diag(g,g^2,\dots,g^n)\mapsto g$
and let $\C_{\chi}$ denote the corresponding 
1-dimensional representation of $\mathsf{S}$. 
Consider the 
associated line bundle 
$E\mathsf{S}\times _{\mathsf{S}}\C_\chi \to B\mathsf{S}$. 
We denote by $t$ the first Chern class of the dual of this line bundle, i.e., 
\begin{align}\label{deg of t in coh}
t=c_1^\mathsf{S}(\C_{\chi}^*) \in H^2(B\mathsf{S}).
\end{align} 
As in the case of $BT$ discussed in Section~\ref{section: flag}, we may 
identify $H^*(B\mathsf{S})$ with the polynomial ring $\mathbb{Q}[t]$ in the single variable $t$.

Since the one-dimensional torus group $\mathsf{S}$ is a subgroup of the torus $T$ of diagonal matrices, there is a natural restriction homomorphism
$H^*(BT) \to H^*(B\mathsf{S})$ induced by the inclusion $\mathsf{S} \hookrightarrow T$. Indeed, we have the restriction map
$\text{Hom}(T,\C^*) \to \text{Hom}(\mathsf{S},\C^*)$, which then induces a linear map $H^2(BT) \to H^2(B\mathsf{S})$ via the 
identification in \eqref{eq:equiv coh of pt degree two}.  This then yields 
the restriction homomorphism
\begin{align} \label{eq:restriction map from BT to BS}
\pi: H^*(BT) \to H^*(B\mathsf{S})
\end{align}
which can also be viewed as a ring homomorphism
\begin{equation}\label{eq: restriction map from BT to BS in coord}
\pi: \Q[t_1, \cdots, t_n] \to \Q[t] 
\end{equation} 
via the isomorphisms $H^*(BT) \cong \Q[t_1,\cdots, t_n]$ and $H^*(B\mathsf{S}) \cong \Q[t]$ as above. 
Indeed, we can compute explicitly where the homomorphism $\pi$ in~\eqref{eq: restriction map from BT to BS in coord}
sends the variables $t_i$ in terms of $t$. This is 
the content of the next lemma.

\begin{Lemma} \label{lemma: ti goes to it}
The homomorphism $\pi$ of~\eqref{eq:restriction map from BT to BS} (or equivalently~\eqref{eq: restriction map from BT to BS in coord}) maps $t_i$ in \eqref{deg of ti in coh} to $i \cdot t$ for all $i \in [n]$. 
\end{Lemma}

\begin{proof}
By the definitions of $t_i$ and $t$ in~\eqref{deg of ti in coh} and~\eqref{deg of t in coh} respectively, 
it suffices to show that the homomorphism $\text{Hom}(T,\C^*) \to \text{Hom}(\mathsf{S},\C^*)$ maps $\chi_i^{-1}$ to $\chi^{-i}=(\chi^{-1})^i$ where $\chi_i^{-1}$ is defined by $\chi_i^{-1}(\diag(g_1,\dots,g_n))= g_i^{-1}$ and $\chi^{-i}$ is defined as $\chi^{-i}(\diag(g,g^2,\dots,g^n)) = g^{-i}$.
But this is clear by the definition of $\mathsf{S}$ in~\eqref{eq:def of S}.
Therefore, we have $\pi(t_i)=it$ for all $i \in [n]$, as desired.
\end{proof}

Next, we observe that the localization techniques described in Section~\ref{subsec: background equivariant cohomology}  holds for regular nilpotent Hessenberg varieties with respect to the action of the $1$-dimensional torus $\mathsf{S}$ defined above. More specifically, 
  as we have mentioned above, it is known that a Hessenberg variety (in Lie type A) admits a paving by complex
  affine cells \cite[Theorem 7.1]{ty}, so their cohomology
  rings are concentrated in even degrees. Hence,
  as in~\eqref{eq:free module over H(BT)}, 
   we can conclude that 
$H_\mathsf{S}^*(\Hess(\mathsf{N},h))$ is a free $H_\mathsf{S}^*(\pt)$-module and we have an isomorphism 
  \begin{equation} \label{eq:equivariant cohomology of Hess(N,h) free}
H_\mathsf{S}^*(\Hess(\mathsf{N},h)) \cong H_\mathsf{S}^*(\pt)\otimes_{\Q} H^*(\Hess(\mathsf{N},h)) \quad \text{as $H_\mathsf{S}^*(\pt)=H^*(B\mathsf{S})$-modules}. 
\end{equation}
Moreover, by Corollary~\ref{corollary:localization_theorem} we also have an injection
\begin{equation}
  \label{eq:localization for HessN}
  \iota_2:H^*_\mathsf{S}(\Hess(\mathsf{N},h)) \into H^*_\mathsf{S}(\Hess(\mathsf{N},h)^\mathsf{S}) 
\end{equation}
where the map is induced from the inclusion $\Hess(\mathsf{N},h)^\mathsf{S} \hookrightarrow \Hess(\mathsf{N},h)$.

Thus, just as in the case of $H^*_T(\fln)$, in order to analyze $H^*_\mathsf{S}(\Hess(\mathsf{N},h))$ it suffices to understand the
restriction to the $\mathsf{S}$-fixed point set $\Hess(\mathsf{N},h)^\mathsf{S}$. 
Recall that the $\fln^T$ 
can be identified with the
permutation group $S_n$ on $n$ letters. 
  Restricting 
to the subtorus $\mathsf{S} \subset  T$, it is straightforward to check that
the $\mathsf{S}$-fixed point set $\fln^\mathsf{S}$ of the flag variety
$\fln$ is in fact equal to $\fln^T$, i.e. 
\begin{align}\label{S-fixed = T-fixed} 
\fln^\mathsf{S} =  \fln^T.
\end{align}
From here it also quickly follows that 
\begin{align}\label{S-fixed Hess(N,h)} 
\Hess(\mathsf{N},h)^\mathsf{S}=\mathcal \Hess(\mathsf{N},h)\cap (\fln)^T.
\end{align}
Thus the set of $\mathsf{S}$-fixed point set $\Hess(\mathsf{N},h)^\mathsf{S}$ is a subset of
$\fln^T$, and through our fixed identification $\fln^T \cong S_n$
we may henceforth view $\Hess(\mathsf{N},h)^\mathsf{S}$ as a subset of
$S_n$. 
In a later section we will describe, 
in explicit combinatorial terms, the fixed point sets $\Hess(\mathsf{N},h)^\mathsf{S}$ (cf. Lemmas~\ref{lemma:fixed_point_Peterson} and~\ref{lemma:Hess fixed points}).

We may now consider the commutative diagram 
\begin{equation}\label{eq:cd}
\begin{CD}
H^{\ast}_{T}(\fln)@>{\iota_1}>> H^{\ast}_{T}(\fln^{T})\cong\displaystyle \bigoplus_{w\in S_n} \Q[t_1,\dots,t_n]\\
@V{}VV @VV{\bigoplus_w \pi}V\\
H^{\ast}_{\mathsf{S}}(\fln)@>{\iota'_1}>> H^{\ast}_{\mathsf{S}}(\fln^{\mathsf{S}})\cong\displaystyle \bigoplus_{w\in S_n} \Q[t]\\
@V{}VV @VV{\mathsf{pr}_{\Hess}}V\\
H^{\ast}_{\mathsf{S}}(\Hess(\mathsf{N},h))@>{\iota_2}>> H^{\ast}_{\mathsf{S}}(\Hess(\mathsf{N},h)^{\mathsf{S}})\cong\displaystyle \bigoplus_{w\in \Hess(\mathsf{N},h)^{\mathsf{S}}\subset S_n} \Q[t]
\end{CD}
\end{equation}
where all maps are induced from inclusion maps on underlying
spaces.  Note that all of $\iota_1$, $\iota'_1$, and $\iota_2$ are
injective, since $\iota'_1$ is a special case of
$\iota_2$. The vertical map $\bigoplus_w \pi$ is by definition the direct sum over $w \in S_n$ 
of the map $\pi: \Q[t_1, \cdots, t_n] \cong H^*_T(w) \to \Q[t] \cong H^*_{\mathsf{S}}(w)$ 
  sending each $t_i$ to $it$ by Lemma~\ref{lemma: ti goes to it}.
  Finally, the vertical map $\mathsf{pr}_{\Hess}$ in~\eqref{eq:cd} is the projection which forgets the coordinates 
  corresponding to $w$ for $w \not \in \Hess(\mathsf{N},h)^\mathsf{S}$.  It is induced by the inclusion map 
  of $\mathsf{S}$-fixed points $\Hess(\mathsf{N},h)^{\mathsf{S}} \into \fln^{\mathsf{S}}$.

  \begin{Remark}\label{remark: pi alpha i} 
 It is useful to note that the map $\pi: \Q[t_1,\cdots,t_n] \to \Q[t]$ above has the property 
\begin{equation}\label{eq:2-1}
\pi(\alpha_i)=t\quad  \textup{ for } \, i \in [n-1] 
\end{equation}
for a simple root $\alpha_i$ defined in \eqref{eq:simple root in type A}, as can be immediately checked using the fact that $\pi(t_i) = it$.  
\end{Remark}

The above discussion shows that the map $\iota_2$ in~\eqref{eq:cd} plays a very similar role with respect to $H^*_\mathsf{S}(\Hess(\mathsf{N},h))$ to that played by $\iota_1$ in the analysis of $H^*_T(\fln)$. In the following sections, we will exploit this 
analogy in order to analyze the equivariant and ordinary cohomology rings of $\Hess(\mathsf{N},h)$.

\section{ Baby case: the cohomology rings of Peterson varieties in type A } 
\label{section: peterson in type A}

Equipped with the background material developed in the previous sections, we are now in a position 
to prove some concrete results about the (equivariant and ordinary) cohomology rings of regular nilpotent Hessenberg varieties. As the title to this section suggests, we begin with the ``baby'' case of the Peterson variety, which is the ``smallest'' interesting regular nilpotent Hessenberg variety, in the sense that any regular nilpotent Hessenberg variety which is not a direct product of smaller-rank regular nilpotent Hessenberg varieties must satisfy $h(i) \geq i+1$ for all $i \in [n-1]$, and the Peterson case is precisely when this inequality is an equality, i.e., $h(i)=i+1$ for all $i \in [n-1]$.  We decided to present this case separately because, due to its simplicity, we can prove some results in this case that are currently out of reach in the general case ; moreover, again because of the simplicity of Petersons, even the generally applicable techniques are perhaps most easily understood and digested in the Peterson case.  Thus we felt that they serve as a useful setting to get an initial glimpse of the flavor, and the techniques, behind these results.

The outline of the section is as follows. In Section~\ref{subsec: background petersons} we give some necessary computations
related to the fixed point sets of Peterson varieties. Then in Section~\ref{subsection:basis_Peterson_typeA} we derive a $H^*_\mathsf{S}(\pt)$-module basis of the $\mathsf{S}$-equivariant cohomology rings of Peterson varieties in Lie type A. 
Then, in Section~\ref{subsec: presentation of Peterson cohomology type A}, we give explicit presentations by generators-and-relations of the (equivariant and ordinary) cohomology rings of Peterson varieties in Lie type A.

\subsection{ The set of $\mathsf{S}$-fixed points in the Peterson variety }
\label{subsec: background petersons}

As mentioned in the Introduction, the Peterson variety $\Y$ is the special case of a regular nilpotent Hessenberg variety $\Hess(\mathsf{N},h)$ where we take
$h = (2,3,4,\cdots, n,n)$, i.e., $h(i)=i+1$ for $i \in [n-1]$.
 Since we consistently assume that $h$ is indecomposable in this exposition, so $h(i) \geq i+1$ for all $i<n$ and $h(n)=n$, it is clear -- as was explained above -- that the Hessenberg function $h$ for the Peterson case is the ``smallest'' example of an (indecomposable) regular nilpotent Hessenberg variety, or in other words, $h(i) \leq h'(i)$ for any other indecomposable $h'$. Not surprisingly, the Peterson variety is also one of the simplest cases to analyze combinatorially, as we shall now see. 
As we mentioned above, the overall goal is to use the techniques discussed thus far in this paper and give an explicit 
and combinatorial description of the equivariant cohomology rings of Peterson varieties. From the diagram~\eqref{eq:cd} 
it is not hard to see that, in order to do so, we must first compute explicitly the $\mathsf{S}$-fixed points $\Y^\mathsf{S}$. This is the content of the following lemma. Here and in what follows, we denote permutations $w \in S_n$ in standard one-line notation 
\[
w = w(1) \, w(2) \cdots w(n-1) \, w(n). 
\]
Also recall that, by slight abuse of notation, we denote a permutation flag $wE_{\bullet}$ simply as $w$, thus identifying 
$\fln^T$ with $S_n$. 

\begin{Lemma} \label{lemma:fixed_point_Peterson}
Let $\mathsf{S}$ be the subtorus of the standard maximal torus $T$ defined by~\eqref{eq:def of S}.  Then $\mathsf{S}$ preserves $\Y$. (See the beginning of Section~\ref{section: A variation on the theme: the equivariant cohomology rings of Hessenberg varieties}.) Moreover, 
the set of $\S1$-fixed points $\Y^{\S1}$ consists of permutation flags $w = w E_{\bullet}$ such that the one-line notation of $w$ is of the form
\begin{equation} \label{eq:2.4}
w=\underbrace{j_1\ j_1-1\cdots 1}_\text{$j_1$ entries}\ \underbrace{j_2\ j_2-1\cdots j_1+1}_\text{$j_2-j_1$ entries} \cdots \underbrace{n\ n-1\cdots j_m+1}_\text{$n-j_m$ entries}, 
\end{equation}
for some $1\le j_1 < j_2 < \cdots < j_m < n$.
\end{Lemma}

\begin{Example}
When $n=3$, we have 
$$
\Y^{\S1}=\{123, 132, 213, 321\}.
$$
\end{Example}

Before proving Lemma~\ref{lemma:fixed_point_Peterson} we make 
an observation about the $\S1$-fixed points of $\Y$, since this characterization 
will be used in later sections. Moreover, this 
observation allows us to generalize the above lemma to all Lie types. 

\begin{Remark}\label{remark: Pet fixed points and subsets}
The description of the $\S1$-fixed points $\Y^{\S1}$ given above can be interpreted as saying that $\Y^{\S1}$ consists of products of the Bruhat-maximal permutations in certain permutation subgroups of $S_n$. Indeed, from Lemma~\ref{lemma:fixed_point_Peterson}, it is immediate that a $w$ of the form given in~\eqref{eq:2.4} is the Bruhat-maximal permutation in the product subgroup 
$$
S_{\{1,2,\cdots,j_1\}} \times S_{\{j_1+1,\cdots,j_2\}} \times \cdots \times S_{\{j_m+1, \cdots,n\}}
$$
for $1 \leq j_1 < j_2 < \cdots < j_m < n$. In what follows, we associate to such a $w$ the complement set 
$$
\mathcal{A} := [n] \setminus \{j_1, j_2, \cdots, j_m, n\} \subseteq [n-1]
$$ 
and, given a subset $\mathcal{A} \subseteq [n-1]$, denote the corresponding (unique) $w$ by $w_{\mathcal{A}}$. With this notation in place, the content of Lemma~\ref{lemma:fixed_point_Peterson} can be restated as 
$$
\Y^{\S1} = \{w_{\mathcal{A}} \, \mid \, \mathcal{A} \subseteq [n-1] \}.
$$ 
\end{Remark}

We now prove Lemma~\ref{lemma:fixed_point_Peterson}. 

\begin{proof}[Proof of Lemma~\ref{lemma:fixed_point_Peterson}]
We leave the proof of the first claim to the reader. 
For the second claim, we first recall from~\eqref{S-fixed Hess(N,h)} that 
\begin{equation*}
\Y^{\S1} = \Y \cap \fln^{T} = \Y \cap S_n.
\end{equation*}
Hence, it is enough to show that a permutation flag $w E_{\bullet}$ belongs to $\Y$ if and only if $w$ is of the form in \eqref{eq:2.4}. 
We can see this as follows: 
\begin{align*}
w E_{\bullet} \in \Y &\iff Nw E_{i} \subset wE_{i+1} \, \textup{ for all } \, i \in [n-1] \\
                           &\iff w^{-1}Nw E_{i} \subset E_{i+1} \, \textup{ for all } \, i \in [n-1] \\
                           &\iff w^{-1}Nw \langle e_{1}, e_{2}, \cdots ,e_{i}\rangle \subset \langle e_{1}, e_{2}, \cdots ,e_{i+1}\rangle \, \textup{ for all } \, i \in [n-1] \\
                           &\iff w^{-1}Nw e_{i} \in \langle e_{1}, e_{2}, \cdots ,e_{i+1}\rangle \, \textup{ for all } \, i \in [n-1] \\
                           &\iff e_{w^{-1}(w(i)-1)} \in \langle e_{1}, e_{2}, \cdots ,e_{i+1}\rangle \, \textup{ for all } \, i \in [n-1]\\
                           & \hspace{80pt} \text{ (with convention } e_0=0 \text{ and } w^{-1}(0)=0) \\
                           &\iff w^{-1}(w(i)-1) \leq i+1 \, \textup{ for all } \, i \in [n-1] \\
                           &\iff w \text{ is of the form in \eqref{eq:2.4}}
\end{align*}
which completes the proof. 
\end{proof}

\subsection{A basis for $H^*_{\S1}(\Y)$} \label{subsection:basis_Peterson_typeA}

It is well-known that the (ordinary) Schubert classes $\{\sigma_w\}_{w \in S_n}$ form an 
additive basis for the cohomology ring $H^*(\fln)$ of $\fln$. 
An analogous statement is true 
for $T$-equivariant cohomology; more precisely, the $T$-equivariant Schubert classes $\{\sigma^T_w\}_{w \in S_n}$ form 
a module basis for the free $H^*_T(\pt)$-module $H^*_T(\fln) \cong H^*_T(\pt) \otimes H^*(\fln)$ by Proposition~\ref{proposition:generators equivariant cohomology}-(2). 
Moreover, when viewed as elements of $\bigoplus_{w \in S_n} H^*_T(w)$ via the injection $\iota_1$ as in~\eqref{eq:cd}, the $T$-equivariant Schubert classes $\sigma^T_w$ satisfy an upper-triangularity property which make them very useful for explicit combinatorial computations. 
It turns out we can build a $H^*_\mathsf{S}(\pt)$-module basis of the $\mathsf{S}$-equivariant cohomology $H^*_\mathsf{S}(\Y)$ by using 
an appropriately chosen subset of the $T$-equivariant Schubert classes $\{\sigma^T_w\}_{w \in S_n}$ and the commutative diagram~\eqref{eq:cd}, and these classes will satisfy a upper-triangularity property analogous to that of the equivariant Schubert classes when viewed as elements of $\bigoplus_{w \in \Y^\mathsf{S}} H^*_\mathsf{S}(w)$ via the injection $\iota_2$ in~\eqref{eq:cd}. The goal of this section is to briefly explain these assertions and to obtain several significant consequences from this fact. 
The main reference of this section is \cite{HarTym-Monk}.

First we recall the computationally convenient upper-triangularity conditions, as alluded to above, 
which are satisfied by the $T$-equivariant Schubert
classes. 
As discussed in Example~\ref{example: billey upper triangular}, 
 the images of the $T$-equivariant Schubert classes in 
$\bigoplus_{w \in S_n} H^*_T(w)$ via the injection $\iota_1$ satisfy the condition 
\begin{align} \label{eq:GmodB-upper-triangular}
\sigma^T_w(v) = 0 \, \, \textup{  if  } v \not \geq w  
\end{align}
with respect to Bruhat order on $S_n$. Moreover, it is also known (and not hard to check directly using Billey's formula, Theorem~\ref{theorem:billey}) that 
\begin{align}\label{eq: upper triangular part 2} 
\sigma^T_w(w) \neq 0. 
\end{align} 
Together, we refer to these two conditions as the ``upper-triangularity properties'' (with respect to Bruhat order) 
of the $T$-equivariant Schubert classes. These properties make computations with the classes $\sigma^T_w$ much easier.  

Our next goal is to define the classes which will play the role in $H^*_\mathsf{S}(\Y)$ analogous to the equivariant Schubert classes. 
We have already seen that $\Y^{\mathsf{S}}$ is enumerated by the set of subsets ${\mathcal{A}}$ of $[n-1]$, cf. Remark~\ref{remark: Pet fixed points and subsets}. 
We now associate to each subset ${\mathcal{A}} \subset [n-1]$ a permutation $v_{\mathcal{A}} \in S_n$, which is \emph{not} (necessarily) contained 
in $\Y^\mathsf{S}$.  We have the following.

\begin{Definition}\label{def:vA}
  Let \({\mathcal{A}} = \{i_1 < i_2 < \cdots < i_k\}\)
 be a subset of $[n-1]$ as in Remark~\ref{remark: Pet fixed points and subsets}. We define the element $v_{\mathcal{A}} \in S_n$ 
 to be the product of simple transpositions whose indices
 are in ${\mathcal{A}}$, in increasing order, i.e. 
\[
v_{\mathcal{A}} := s_{i_1} s_{i_2} \cdots s_{i_k} = \prod_{j=1}^k s_{i_j}.
\]
\end{Definition}

It is clear from the definition above that each subset $\mathcal{A} \subset [n-1]$
corresponds to a unique permutation of the form $v_{\mathcal{A}}$. 
We now define 
\begin{equation}\label{eq: definition pvA}
p_{v_{\mathcal{A}}} := \textup{ image of $\sigma^T_{v_{\mathcal{A}}}$ under the restriction map $H^*_T(\fln) \to H^*_{\mathsf{S}}(\Y)$ in~\eqref{eq:cd}} 
\end{equation} 
for every subset $\mathcal{A}$. We call $p_{v_{\mathcal{A}}}$ the \textbf{Peterson Schubert class} (associated to $\mathcal{A}$). 
The set of Peterson Schubert classes
$\{p_{v_{\mathcal{A}}}\}$ for all subsets ${\mathcal{A} \subset [n-1]}$
gives rise to a collection of elements in $H^*_{\S1}(\Y)$ in
one-to-one correspondence with the $\S1$-fixed points of $\Y$. 
We will show below that these form a module basis of $H^*_\S1(\Y)$ analogous to Schubert classes, as we already hinted at above.

Our first claim is that the Peterson Schubert classes are upper-triangular in the following sense. 
Here, by slight abuse of language 
we also use the term \textbf{Bruhat order} to refer to the partial order on $\Y^\S1 \subset S_n$ obtained by restricting the standard Bruhat order on $S_n$ to $\Y^\S1$.  With this terminology in place, we can state the upper-triangular property of the $p_{v_{\mathcal{A}}}$ as follows: 
\begin{equation}\label{eq:Peterson-upper-triangular} 
  p_{v_{\mathcal{A}}}(w_{\mathcal{B}}) = 0 \quad \mbox{if} \; \; w_{\mathcal{B}} \not \geq v_{\mathcal{A}} 
\end{equation}
and 
\begin{equation}\label{eq:Peterson-upper-triangular-nonzero}
 p_{v_{\mathcal{A}}}(w_{\mathcal{A}}) \neq 0. 
\end{equation}
The remainder of this section is devoted to proving the above, and from it 
to briefly sketch how one may derive the accompanying statement that these Peterson Schubert classes form a module basis of $H^*_\S1(\Y)$. 
We first note that the definition of $v_{\mathcal{A}}$ immediately implies that 
\[
s_i \leq v_{\mathcal{A}} \quad \textup{ for all } i \in {\mathcal{A}} \, \, \textup{ in Bruhat order.} 
\]
We record some basic facts below which will be important in what
follows. The proofs are straightforward, but we provide brief sketches of proofs to give the reader a sense of their flavor.

\begin{Lemma}\label{fact:vA-length}
The Bruhat-length of $v_{\mathcal{A}}$ is the size of the set
$\mathcal{A}$, i.e. \(\ell(v_{\mathcal{A}}) = |\mathcal{A}|.\) In
particular, the decomposition in Definition~\ref{def:vA} is
minimal-length. 
\end{Lemma} 

\begin{proof} 
From the Definition~\ref{def:vA} of $v_{\mathcal{A}}$ it is straightforward to compute its Bruhat length; this can be done by, for instance, counting the number of descents in the one-line notation of $v_{\mathcal{A}}$. 
The reader may consult, for example, \cite{BjornerBrenti} for basic information concerning descents 
and its relation to Bruhat order. 
\end{proof}

We denote a consecutive string from $a$ to $b$ with $a<b$ by $[a,b]:= \{a,a+1,\ldots,b-1,b\}$.   We additional say that 
$[a,b]$ is a \textbf{maximal consecutive string} in $\mathcal{A}$ if neither $a-1$ nor $b+1$ are in $\mathcal{A}$ and $[a,b]$ is a subset of $\mathcal{A}$. 

\begin{Lemma}\label{fact:vA-reduced-word-unique}
If $\mathcal{A} = [a,b]$ is a maximal consecutive string in $\mathcal{A}$, then the word in
Definition \ref{def:vA} is the unique reduced word decomposition for
$v_{[a,b]}$. 
\end{Lemma} 

\begin{proof} 
From the definition of $v_{[a,b]}$ and the presentation of $S_n$ via the simple transpositions and the braid relations (see e.g. \cite{BjornerBrenti}), it can be seen that there are no possible alternative representations of $v_{[a.b]}$ in terms of the simple transpositions. 
\end{proof} 

\begin{Lemma}\label{fact:vA-product}
If ${\mathcal{A}} =  [a_1, b_1] \cup [a_2,b_2] \cup \cdots \cup
[a_m,b_m]$ is a decomposition 
of $\mathcal{A}$ into maximal consecutive substrings, 
then 
\[v_{\mathcal{A}} = v_{  [a_1, b_1] }v_{[a_2, b_2] } \cdots v_{  [a_m, b_m] }.\]
Moreover, for any reduced word decomposition of $v_{\mathcal{A}}$, 
there exists exactly one subword of that word decomposition which
is equal to $v_{[a_\ell, b_\ell]}$ for any $\ell, 1 \leq \ell \leq m$. 
\end{Lemma} 

\begin{proof} 
The first claim is immediate from Definition~\ref{def:vA}. The second claim follows from Lemma~\ref{fact:vA-reduced-word-unique}. 
\end{proof}

It follows from Remark~\ref{remark: Pet fixed points and subsets} that for a consecutive string $[a_\ell, b_\ell]=\{a_\ell, a_{\ell}+1,\ldots,
b_\ell\}$, the element $w_{[a_\ell, b_\ell]}$ is the largest element with respect to
Bruhat order in the subgroup $S_{\{a_\ell,a_\ell+1,\cdots,b_\ell, b_\ell+1\}}$ of permutations of the subset
$\{a_\ell, a_{\ell}+1, \cdots, .b_\ell, b_\ell+1\}$.  For our computations below, we fix the following
reduced-word decomposition of $w_{[a_\ell, b_\ell]}$:
\begin{equation}
  \label{eq:wjk-reduced-word}
  w_{[a_\ell, b_\ell]} = \prod_{s=0}^{b_\ell-a_\ell} \left( \prod_{i=0}^{b_\ell-a_\ell-s}
    s_{a_\ell+i} \right).
\end{equation}
Here we take the convention that a product is always composed from the
left to the right, so $\prod_{i=0}^{p} \beta_i = \beta_0 \cdot
\beta_1 \cdots \beta_p$ for any expressions $\beta_i$. 
Then for ${\mathcal{A}} =  [a_1, b_1] \cup [a_2, b_2] \cup \cdots \cup
[a_m, b_m]$, 
a reduced-word decomposition for $w_{\mathcal{A}}$ is given by taking
the product of the above reduced-word decompositions of $w_{[a_\ell, b_\ell]}$ for each of the maximal consecutive
substrings $[a_\ell, b_\ell]$ of ${\mathcal{A}}$, ordered so that maximal consecutive substrings increase from left 
to right. 
In other words, 
\begin{equation}\label{eq:wA-reduced-word} 
  w_{\mathcal{A}} = w_{[a_1, b_1]} w_{[a_2, b_2]} w_{[a_3, b_3]}
  \cdots w_{[a_m, b_m]}
\end{equation}
where $\mathcal{A} = [a_1, b_1] \cup [a_2, b_2] \cup \cdots \cup
[a_m, b_m]$ as above. 
With these conventions in place, we have the following.

\begin{Lemma}\label{fact:unique-vA-in-wA}
  There exists exactly one reduced subword in the reduced word
  decomposition~\eqref{eq:wA-reduced-word} of $w_{\mathcal{A}}$ which
  is equal to $v_{\mathcal{A}}$.  
  \end{Lemma} 

\begin{proof}
There is a unique reduced word decomposition for $v_{ [a_\ell, b_\ell] }$ for each maximal 
  consecutive string in $\mathcal{A}$. The claim then follows from the choice of reduced word decomposition of
  $w_{\mathcal{A}}$.
\end{proof}

The next lemma is the crucial observation which allows us to show that
the Peterson Schubert classes $p_{v_{\mathcal{A}}}$ corresponding to these
special Weyl group elements $v_{\mathcal{A}}$ are a $H^*_{\S1}(\pt)$-module
basis for $H^*_{\S1}(\Y)$. The point is that the Bruhat order on
$\Y^{\S1}$ can be translated to the usual partial order on sets given
by containment.

\begin{Lemma}\label{lemma:containment}
Let ${\mathcal{A}},{\mathcal{B}}$ be subsets in $[n-1]$.  Then $v_{\mathcal{A}} \leq w_{\mathcal{B}}$ if and only if ${\mathcal{A}} \subset {\mathcal{B}}$.
\end{Lemma}

\begin{proof}
  If ${\mathcal{A}} \subset {\mathcal{B}}$ then $v_{\mathcal{A}}
  \leq w_{\mathcal{B}}$ by definition of $v_{\mathcal{A}}$ and
  $w_{\mathcal{B}}$.  Now suppose that $v_{\mathcal{A}} \leq
  w_{\mathcal{B}}$.  In particular this means that $s_i \leq
  v_{\mathcal{A}}$ for all $i \in {\mathcal{A}}$ and Bruhat order is transitive, 
  so $s_i \leq w_{\mathcal B}$.  By definition of
  $w_{\mathcal{B}}$ this means $i \in {\mathcal{B}}$. Hence
  ${\mathcal{A}} \subset {\mathcal{B}}$ as desired. 
\end{proof}

Our next step is to develop  
tools to compute restrictions of
$p_{v_{\mathcal{A}}}$ at various fixed points $w_{\mathcal{B}} \in
\Y^{\S1}$.  These methods allow us to prove the upper-triangularity
condition~\eqref{eq:Peterson-upper-triangular} with
respect to the partial order on sets (equivalent to the restriction of
Bruhat order by
Lemma~\ref{lemma:containment}). The upper-triangularity properties then allow us, in turn, 
to prove that the Peterson Schubert classes are a module basis for $H^*_{\S1}(\Y)$ in Theorem~\ref{theorem:pvA-basis}.

The following is a well-known consequence of Billey's formula (Theorem~\ref{theorem:billey}) and concerns the Schubert classes $\sigma_v^T$ for general $v \in S_n$. 

\begin{Lemma}\label{fact:billey-positive}
Each summand in Billey's formula for $\sigma_v^T(w)$ is a degree $2\ell(v)$ polynomial in
the simple roots $\{t_{i+1} - t_i \}_{i=1}^{n-1}$ with {\em nonnegative} integer coefficients.
Here, we recall that $\deg t_i =2$ for all $i \in [n]$. 
\end{Lemma}

\begin{proof} 
As noted in the discussion after Definition~\ref{definition.term.in.billey}, it is known that each $r(i, \mathbf{b})$ is a positive root \cite[Equation 4.1 and surrounding discussion]{Billey}, and thus 
a non-negative integral linear combination
of simple positive roots.  Being the product of such, each term in Billey's formula is a polynomial in the simple positive roots with non-negative coefficients, as desired. 
\end{proof} 

The positivity phenomenon expressed in Lemma~\ref{fact:billey-positive} implies that if any
summand in Billey's formula for $\sigma_v^T(w)$ is nonzero, then the
entire sum is nonzero. From this we derive the following. 

\begin{Corollary}\label{cor:pvAwA-nonzero} 
Let ${\mathcal{A}} \subset [n-1]$. Then $p_{v_{\mathcal{A}}}(w_{\mathcal{A}}) \neq 0$. 
\end{Corollary}

\begin{proof}
We noted in Lemma~\ref{fact:unique-vA-in-wA}   
that $v_{\mathcal{A}}$ can be found as a subword of $w_{\mathcal{A}}$.
This implies that
$\sigma_{v_{\mathcal{A}}}^T(w_{\mathcal{A}}) \neq 0$, by the positivity phenomenon in Lemma~\ref{fact:billey-positive},
and since there is at least one non-zero term in the Billey formula for $\sigma^T_{v_{\mathcal{A}}}(w_{\mathcal{A}})$. By Remark~\ref{remark: pi alpha i} we know $\pi(t_{i+1}-t_i) = t$, so 
the image in $\Q[t]$ of any nonzero polynomial in the
$t_{i+1}-t_{i}$ with positive coefficients is also nonzero in $\Q[t]$. 
\end{proof}

The proof of the corollary above also shows the following.

\begin{Proposition}\label{prop:deg-pvA}
Let $\mathcal{A}, \mathcal{B}$ be subsets of $[n-1]$. 
Then 
\begin{enumerate} 
\item the restriction $p_{v_{\mathcal{A}}}(w_{\mathcal{B}})$ of the Peterson
Schubert class $p_{v_{\mathcal{A}}}$ at any $w_{\mathcal{B}}$ has 
degree $2\ell(v_{\mathcal{A}}) = 2\lvert {\mathcal{A}} \rvert$ as a polynomial in
$\Q[t]$ where we recall that $\deg t=2$, 
\item $p_{v_{\mathcal{A}}}$ has cohomology degree $2 \lvert {\mathcal{A}} \rvert$, and
\item $p_{v_{\mathcal{A}}}(w_{\mathcal{B}})$ is non-zero if
$\sigma^T_{v_{\mathcal{A}}}(w_{\mathcal{B}})$ is non-zero. 
\end{enumerate} 
\end{Proposition}

We can now give most of the proof of the assertion that the Peterson Schubert classes are a module basis of $H^*_{\S1}(\Y)$. 
The only step which we quote without proof is a technical result from \cite[Appendix]{HarTym-Monk}, needed 
at the very end of the argument. 

\begin{Theorem} \label{theorem:pvA-basis}  (\cite[Theorem~4.12]{HarTym-Monk})
Let $\Y$ be the Peterson variety of type $A_{n-1}$ equipped with the
natural $\S1$-action defined by~\eqref{eq:def of S}. For ${\mathcal{A}}
\subset [n-1]$, let $v_{\mathcal{A}} \in S_n$ be the permutation
given in Definition~\ref{def:vA}, and let $p_{v_{\mathcal{A}}}$ be the
corresponding Peterson Schubert class in $H^*_{\S1}(\Y)$. 
Then the classes $\{p_{v_{\mathcal{A}}}: {\mathcal{A}} \subset [n-1] \}$ in
$H^*_{\S1}(\Y)$ 
\begin{itemize}
\item form an $H^*_{\S1}(\pt)$-module basis for $H^*_{\S1}(\Y)$, and 
\item satisfy the upper-triangularity conditions 
\begin{equation}\label{eq:pvA-at-wB}
p_{v_{\mathcal{A}}}(w_{\mathcal{B}}) = 0 \quad \mbox{if } {\mathcal{B}} \not \supset {\mathcal{A}},  
\end{equation}
and
\begin{equation}\label{eq:pvA-at-wA-nonzero}
p_{v_{\mathcal{A}}}(w_{\mathcal{A}}) \neq 0.
\end{equation}
\end{itemize}
\end{Theorem}

\begin{proof}
  We begin with a proof of the upper-triangularity
  condition~\eqref{eq:pvA-at-wB}. 
  Recall that $v_{\mathcal{A}} \leq w_{\mathcal{B}}$ if and only if
  ${\mathcal{A}} \subset {\mathcal{B}}$ by
  Lemma~\ref{lemma:containment}.  The image of zero
  under the map $\pi$ in \eqref{eq:restriction map from BT to BS} is still zero, so
  it suffices to show that $\sigma^T_{v_{\mathcal{A}}}(w_{\mathcal{B}})
  = 0$ if $v_{\mathcal{A}} \not \leq w_{\mathcal{B}}$. This follows from the upper-triangularity of
  equivariant Schubert classes~\eqref{eq:GmodB-upper-triangular} (or can be proven
  directly from Billey's formula). 

  Next, the assertion that $p_{v_{\mathcal{A}}}(w_{\mathcal{A}}) \neq 0$ is
  the content of Corollary~\ref{cor:pvAwA-nonzero}. 

  We now show that assertions~\eqref{eq:pvA-at-wB}
  and~\eqref{eq:pvA-at-wA-nonzero} imply that the
  $\{p_{v_{\mathcal{A}}}\}$, ranging over all subsets ${\mathcal{A}}$ of
  $[n-1]$, are $H^*_{\S1}(\pt)$-linearly independent.  Indeed, suppose $\sum_{\mathcal{A}} c_{\mathcal{A}}
  p_{v_{\mathcal{A}}} = 0 \in H^*_{\S1}(\Y)$ for some coefficients \(c_{\mathcal{A}} \in
  H^*_{\S1}(\pt).\) If any subset ${\mathcal{A}}$ has
  $c_{\mathcal{A}} \neq 0,$ then there must exist a minimal such, say
  ${\mathcal{B}}$.  Evaluating at $w_{\mathcal{B}}$,
  we conclude that 
\begin{equation}\label{eq:localization-at-wB}
\sum_{\mathcal{A}} c_{\mathcal{A}} \cdot p_{v_{\mathcal{A}}}(w_{\mathcal{B}}) =0.
\end{equation}
  By hypothesis on
  ${\mathcal{B}},$ we have $c_{\mathcal{A}} = 0$ for all ${\mathcal{A}}
  \subsetneq {\mathcal{B}}$. On the other hand, by~\eqref{eq:pvA-at-wB} we know
  that $p_{v_{\mathcal{A}}}(w_{\mathcal{B}})=0$ for all ${\mathcal{A}} \not
  \subseteq {\mathcal{B}}$.  Hence the equality~\eqref{eq:localization-at-wB} simplifies to
\[
c_{\mathcal{B}} \cdot p_{v_{\mathcal{B}}}(w_{\mathcal{B}})=0.
\]
 From~\eqref{eq:pvA-at-wA-nonzero} and the fact that $H^*_{\S1}(\pt)
 = \Q[t]$ is an integral domain,  we conclude $c_{\mathcal{B}} = 0$, a
  contradiction. Hence the $\{p_{v_{\mathcal{A}}}\}$ are linearly
  independent over $H^*_{\S1}(\pt)$. 

  Lemmas~\ref{fact:vA-length} and~\ref{fact:billey-positive} show that for any \(w \in \Y^{\S1}\) the degree of
  the polynomial $p_{v_{\mathcal{A}}}(w)$ is
  $2\lvert {\mathcal{A}} \rvert$. 
  The polynomial variable $t$ has
  cohomology degree $2$ so the cohomology degree of
  $p_{v_{\mathcal{A}}}$ in $H^*_{\S1}(\Y)$ is $2\lvert {\mathcal{A}} \rvert$.
  Since the $p_{v_{\mathcal{A}}}$ are enumerated precisely by the
  subsets $\mathcal{A}$ of $[n-1]$, 
  we 
  conclude that there are $\binom{n-1}{\ell}$ distinct Peterson Schubert
  classes $p_{v_{\mathcal{A}}}$ of cohomology degree precisely $2\ell$. 
A result of Sommers and Tymoczko \cite{ST} states that
$\binom{n-1}{\ell}$  is precisely the $2\ell$-th Betti number of $H^*(\Y)$. 
Hence by \cite[Appendix]{HarTym-Monk}, 
 the $\{p_{v_{\mathcal{A}}}\}$ form an $H^*_{\S1}(\pt)$-module
basis for $H^*_{\S1}(\Y)$, as desired.  
\end{proof}

The following surjectivity result is a straightforward consequence of the above.
Here we use that, for a $T$-space $X$, the forgetful map $H^*_T(X) \to H^*(X)$ is surjective if $H^*(X)$ is concentrated in even degrees. (This is because, in such a situation, the Serre spectral sequence collapses at the $E_2$-stage \cite[Chapter III, Theorem 2.10]{MiTo}.) Since both $\fln$ and $\Y$ admit pavings by complex affine cells, their cohomology is generated in even degrees, so we know that both $H^*_T(\fln) \to H^*(\fln)$ and $H^*_{\S1}(\Y) \to H^*(\Y)$ are surjective. We will use this in the argument below.

\begin{Corollary} \label{coro:restriction surjective in type A}
The restriction map $H^*_T(\fln) \to H^*_{\S1}(\Y)$ on equivariant cohomology is surjective.
Moreover, the restriction map $H^*(\fln) \to H^*(\Y)$ on ordinary cohomology is also surjective.
\end{Corollary}

\begin{proof}
Since the restriction map $H^*_T(\fln) \to H^*_{\S1}(\Y)$ sends $\sigma^T_{v_A}$ to $p_{v_{\mathcal{A}}}$, it follows from Theorem~\ref{theorem:pvA-basis} that the restriction map $H^*_T(\fln) \to H^*_{\S1}(\Y)$ is surjective.
The last statement follows because, as indicated above, the horizontal maps are surjective in the following commutative diagram
\begin{equation*} 
\begin{CD}
\HT(\fln) @>>> H^*(\fln) \\
@VVV @VVV \\ 
\HS(\Y) @>>> H^*(\Y) \\
\end{CD}
\end{equation*}
where each horizontal map is induced from the Borel fibration.
\end{proof}

\begin{Remark}
Corollary~\ref{coro:restriction surjective in type A} also follows from \cite[Theorem~2]{Insko-Tymoczko}.
In fact, the homomorphism $H_*(\Y;\Z) \rightarrow H_*(\fln;\Z)$ induced from the inclusion $\Y \hookrightarrow \fln$ is injective by \cite[Theorem~2]{Insko-Tymoczko} which implies that the restriction map $\rho: H^*(\fln) \to H^*(\Y)$ with rational coefficients is surjective. 
Then we can prove that the restriction map $\tilde\rho: H^*_T(\fln) \to H^*_{\S1}(\Y)$ on equivariant cohomology is surjective by a similar argument as in the proof of Proposition~\ref{proposition:generators equivariant cohomology}-(1).
Indeed, consider the following commutative diagram of exact sequences
\[
\xymatrix{
  0 \ar[r] & \left(H^{>0}(BT)\right) \ar[d]_-{} \ar[r]^-{} & \HT(\fln) \ar[d]_-{\tilde \rho} \ar[r]^-{} & H^*(\fln) \ar[d]_-{\rho}^-{\textrm{surj}} \ar[r] & 0 \\
  0 \ar[r] & \left(H^{>0}(B\S1)\right) \ar[r]_-{} & \HS(\Y) \ar[r]_-{} & H^*(\Y) \ar[r] & 0
}
\]
where $\left(H^{>0}(BT)\right)$ denotes the ideal of $\HT(\fln)$ generated by elements of $H^{>0}(BT)$ (and similarly for the ideal $\left(H^{>0}(B\S1)\right)$ in $\HS(\Y)$).
Consider a homogeneous element $g \in H^{*}_{\S1}(\Y)$ and let $\overline{g} \in H^*(\Y)$ denote its image in $H^*(\Y)$ under the forgetful map. By the surjectivity of $\rho$, the surjectivity of $H^*_T(\fln) \to H^*(\fln)$, and the commutativity of the diagram, it follows that there exists a homogeneous element $f \in H^{*}_{T}(\fln)$ such that $g-\tilde\rho(f)$ is the kernel of the forgetful map $\HS(\Y) \to H^*(\Y)$. In other words, we can write $g = \tilde{\rho}(f) + t g'$ for some homogeneous element $g' \in H^{*}_{\S1}(\Y)$ with $\deg(g') = \deg(g) - \deg(t) = \deg(g)-2$. 
By induction on degree, it follows that $g'$ can be written as $g' = \tilde\rho(f')$ for some $f' \in H^{*}_{T}(\fln)$ and hence we have $g = \tilde{\rho}(f + t_1 f')$. 
In other words, we have just shown that the restriction map $\tilde\rho: H^*_T(\fln) \to H^*_{\S1}(\Y)$ on equivariant cohomology is surjective.
Since the injectivity of $H_*(\Y;\Z) \rightarrow H_*(\fln;\Z)$ is in fact stated in \cite[Theorem~2]{Insko-Tymoczko} for general Lie types, one can see that Corollary~\ref{coro:restriction surjective in type A} holds for all Lie types.
In Section~\ref{section: peterson general Lie type}, we will give another explanation of this surjectivity for general Lie types using a result of \cite{Dre}, using the theory of Peterson Schubert classes. 
\end{Remark}

Now that we know that the $\{p_{v_{\mathcal{A}}}\}$ form a module basis for $H^*_{\S1}(\Y)$ as given by Theorem~\ref{theorem:pvA-basis}, we can ask -- in the spirit of (and in analogy to) classical Schubert calculus for $H^*_T(\fln)$ -- for the structure constants of the multiplication of the Peterson Schubert classes. 
More specifically, since the $\{p_{v_{\mathcal{A}}}\}$ are a module basis, we know that there exist $c_{\mathcal{A'}\mathcal{A}}^\mathcal{B} \in \Q[t]$ such that 
\begin{equation*}
p_{v_{\mathcal{A'}}} \cdot p_{v_{\mathcal{A}}} = \sum_{\mathcal{B} \subset [n-1]} c_{\mathcal{A'}\mathcal{A}}^\mathcal{B} p_{v_{\mathcal{B}}}
\end{equation*}
and we can ask for explicit computations of the $c_{\mathcal{A'}\mathcal{A}}^\mathcal{B}$. We call this problem \textbf{($\S1$-equivariant) Peterson Schubert calculus}. 

Below we recount two results in this direction, namely, two explicit formulas which we call the Monk's formula and Giambelli's formula for Peterson varieties, and which are analogues of known results with this nomenclature in the case of $\fln$. We refer the reader to \cite{BH} and \cite{HarTym-Monk} respectively for details. 
In order to explain Monk's formula for the Peterson variety, we need some notation. 
For any subset $\mathcal{C} \subset [n-1]$ and $\ell \in \mathcal{C}$, we define
\begin{align*}
\mathcal{H}_\mathcal{C}(\ell) &:= \textrm{the maximal element in the maximal consecutive substring of $\mathcal{C}$ containing $\ell$}; \\
\mathcal{T}_\mathcal{C}(\ell) &:= \textrm{the minimal element in the maximal consecutive substring of $\mathcal{C}$ containing $\ell$}. 
\end{align*}

The Monk formula for $\Y$ gives a complete computation of the structure constants for products of the form $p_{s_i} \cdot p_{v_\mathcal{A}}$ as follows.  We refer to \cite{HarTym-Monk} for the proof. 

\begin{Theorem} $($Monk's formula for the Peterson variety, \cite[Lemma~6.4 and Theorem~6.12]{HarTym-Monk}$)$ \label{theorem:Monk_typeA}
Suppose $i \in [n-1]$, 
$v_{\mathcal{A}} \in S_n$ is the permutation
given in Definition~\ref{def:vA}, and $p_{v_\mathcal{A}}$ is the corresponding Peterson Schubert class. 
Then 
\begin{align*} 
p_{s_i} \cdot p_{v_\mathcal{A}} =p_{s_i}(w_\mathcal{A})p_{v_\mathcal{A}} + \sum_{\mathcal{B} \supsetneq \mathcal{A} \atop |\mathcal{B}|=|\mathcal{A}|+1} c_{i,\mathcal{A}}^\mathcal{B} \, p_{v_\mathcal{B}}, 
\end{align*}
where, for a subset $\mathcal{B} \subset [n-1]$ which is a disjoint union $\mathcal{B} = \mathcal{A}  \sqcup \{j \}$,
\begin{enumerate}
\item if $i \not\in \mathcal{A}$ then $p_{s_i}(w_\mathcal{A})=0$,
\item if $i \in \mathcal{A}$ then $p_{s_i}(w_\mathcal{A})=(\mathcal{H}_\mathcal{A}(i)-i+1)(i-\mathcal{T}_\mathcal{A}(i)+1)t, $
\item if $i \not\in \mathcal{B}$ then $c_{i,\mathcal{A}}^\mathcal{B}=0$,
\item if $i \in \mathcal{B}$ and $i \not\in [\mathcal{T}_\mathcal{B}(j), \mathcal{H}_\mathcal{B}(j)]$, then $c_{i,\mathcal{A}}^\mathcal{B}=0$,
\item if $j \leq i \leq  \mathcal{H}_\mathcal{B}(j)$, then 
$$
c_{i,\mathcal{A}}^\mathcal{B}=(\mathcal{H}_\mathcal{B}(j)-i+1)\dbinom{\mathcal{H}_\mathcal{B}(j)-\mathcal{T}_\mathcal{B}(j)+1}{j-\mathcal{T}_\mathcal{B}(j)},
$$
\item if $\mathcal{T}_\mathcal{B}(j) \leq i \leq j-1$, then 
$$
c_{i,\mathcal{A}}^\mathcal{B}=(i-\mathcal{T}_\mathcal{B}(j)+1)\dbinom{\mathcal{H}_\mathcal{B}(j)-\mathcal{T}_\mathcal{B}(j)+1}{j-\mathcal{T}_\mathcal{B}(j)+1}.
$$
\end{enumerate}
In particular, the structure constants $p_{s_i}(w_\mathcal{A})$ and $c_{i,\mathcal{A}}^\mathcal{B}$ are non-negative integers.
\end{Theorem}

The last result of this section is the Peterson analogue of the classical Giambelli formula. It gives a description of the classes $p_{v_\mathcal{A}}$ in terms of the 
degree-$2$ classes $p_{s_i}$, as follows. We refer the reader to \cite{BH} for the proof.

\begin{Theorem}$($Giambelli's formula for the Peterson variety in Lie type A, \cite[Theorem~3.2]{BH}$)$ \label{theorem_Giambelli_typeA}
If $\mathcal{A} = \coprod_{j=1}^m \mathcal{A}_{j}$ is a decomposition of $\mathcal{A}$ into maximal consecutive substrings, then we have 
\begin{align} \label{eq:Giambelli}
 p_{v_\mathcal{A}} = \frac{1}{|\mathcal{A}_1| ! |\mathcal{A}_2| ! \cdots |\mathcal{A}_m| !} \prod_{i\in \mathcal{A}} p_{s_i}.
\end{align}
\end{Theorem}

\subsection{A presentation of the cohomology ring of Peterson varieties in Lie type A } 
\label{subsec: presentation of Peterson cohomology type A}

The main result of this section is Theorem~\ref{theo:3.1}
below, which gives an explicit ring presentation via generators and relations of
the ${\S1}$-equivariant cohomology ring of the Peterson variety $\Y$ in Lie type A. 
We will use the essential tools introduced in previous sections -- namely, the combinatorial description of the $\S1$-equivariant cohomology rings of Peterson varieties and the theory of Hilbert series from commutative algebra -- to their full effect.

More specifically, the idea is as follows.  We use the commutative diagram~\eqref{eq:3.1} below, together with knowledge of the generators of $H^*_T(\fln)$, in order to obtain candidates for a set of ring generators of $H^*_{\S1}(\Y)$.  (Note that in this section, we are looking for a set of \emph{ring} generators, as distinct from the question asked in the previous section, which was for \emph{module} generators.) Once we have a (candidate set of) generators in hand, we use the combinatorial data of the restrictions of these classes to the fixed points $\Y^{\S1}$ in order to derive algebraic relations which must be satisfied by the generators. Finally, in order to prove that the set of relations thus derived is sufficient to determine the full ideal of relations, we use the arguments with Hilbert series and dimension-counting as developed in Section~\ref{sec: background}.

Following the sketch given above, we first 
consider the following commutative diagram, which is a version of~\eqref{eq:cd} but specialized to the Peterson case: 
\begin{equation} \label{eq:3.1}
\begin{CD}
\HT(\fln) @>\iota_1 >> \bigoplus_{w\in Fl(n)^T=S_n}\HT(w) \\
@V VV @V  V  \mathsf{pr}_{\Hess} \circ (\oplus_w \pi ) V  \\ 
\HS(\Y) @>\iota_2 >> \bigoplus_{w \in \Y^{\S1}\subset S_n}\HS(w). \\
\end{CD}
\end{equation}
Since both $H^{odd}(\fln)$ and
$H^{odd}(\Y)$ vanish, the maps $\iota_1$ and $\iota_2$ above are both
injective.  Also, left vertical map above is surjective by Corollary~\ref{coro:restriction surjective in type A}.  Therefore, $\HS(\Y)$ is isomorphic to the image of $\HT(\fln)$ in~\eqref{eq:3.1}. 

As mentioned above, our goal in this section is to derive a presentation by ring generators and relations of $H^*_{\S1}(\Y)$. To accomplish this, we will first derive a generating set, motivated by~\eqref{eq:3.1}. In fact, our ring generators will be a subset of 
the module basis found in Theorem~\ref{theorem:pvA-basis}. More specifically,  
we let $p_k$ for $1 \leq k \leq n-1$ denote the 
 (equivariant) Peterson Schubert class $p_{s_k}$ defined in 
Section~\ref{subsection:basis_Peterson_typeA}.  Our first step is the following.

\begin{Proposition}\label{prop: generators}
The elements $\xi_1, \ldots, \xi_{n-1}, t$ generate $H^*_{\S1}(\Y)$ as $\Q$-algebras. 
\end{Proposition} 

\begin{proof}
Since the $\{p_{v_{\mathcal{A}}}\}$ form a $\HS(\pt)$-module basis of $H^*_{\S1}(\Y)$ by Theorem~\ref{theorem:pvA-basis}, the result follows from Theorem~\ref{theorem_Giambelli_typeA}. 
\end{proof}

Now recall from Lemma~\ref{lemma: ti goes to it} that $\pi(t_i)=it$ 
and from Lemma~\ref{lemma:flag1} that $\tau^T_i \vert_w = t_{w(i)}$. Similarly it is 
straightforward that $t_i \vert_w = t_i$ (where by abuse of notation we denote $\iota_1(\tau^T_i)$ and $\iota_1(t_i)$ by $\tau^T_i$ and $t_i$ respectively). 
Using this, we can compute the images of the $\xi_k$ in $H^*_{\S1}(\Y^{\S1})
\cong \bigoplus_{w \in \Y^{\S1}} H^*_{\S1}(w)$ under the map $\iota_2$
in~\eqref{eq:3.1}.

\begin{Lemma}\label{lemma:xik}
Let $\xi_k \in H^*_{\S1}(\Y)$ for $1 \le k \le n-1$ as above. Then 
\[
\iota_2(\xi_k)|_{w} =\sum_{i=1}^k(w(i)-i)t. 
\]
\end{Lemma}

\begin{proof}
For $w \in \Y^{\S1}$ and $1 \le k \le n-1$ we have 
\begin{equation} \label{eq:3.8}
\begin{split}
\iota_2(\xi_k)|_{w}&
= (\mathsf{pr}_{\Hess} \circ \bigoplus_{w \in S_n} \pi) \circ ( \iota_1(\sum_{i=1}^k(\tau^T_i-t_i)))|_w \quad \textup{ by commutativity
  of~\eqref{eq:3.1} and definition of $p_k$} \\
&= \pi \left( \sum_{i=1}^k(t_{w(i)}-t_{i}) \right) \quad \textup{ by the remarks before the proof}\\
&=\sum_{i=1}^k(w(i)-i)t \quad \textup{ by Lemma~\ref{lemma: ti goes to it} } 
\end{split}
\end{equation}
as desired. 
\end{proof}

Proposition~\ref{prop: generators} gives a set of ring generators for the ring $H^*_{\S1}(\Y)$. Next, we 
wish to derive a (minimal) set of relations which defines the ring structure. A collection of such relations is obtained in the next lemma. 
Note that $\iota_2(t) \vert_w = t$ for all $w \in \Y^{\S1}$.

\begin{Lemma}\label{lemm:3.1}
Let $\xi_k \in H^*_{\S1}(\Y)$ for $1 \le k \le n-1$ be defined as above. Then 
$$
\xi_k \left(\xi_k-\frac{1}{2}\xi_{k-1}-\frac{1}{2}\xi_{k+1}-t \right)=0
$$
 for $k=1,2,\dots,n-1$, where we take the convention $\xi_0=\xi_n=0$.  
\end{Lemma}

\begin{proof}
Let $w\in \Y^{\S1}\subset S_n$.  It follows from Lemma~\ref{lemma:xik} that  
\begin{equation} \label{eq:3.10}
\begin{split}
&\iota_2(\xi_k-\frac{1}{2}\xi_{k-1}-\frac{1}{2}\xi_{k+1}-t)|_w\\
=& \sum_{i=1}^k(w(i)-i)t-\frac{1}{2}\sum_{i=1}^{k-1}(w(i)-i)t-\frac{1}{2}\sum_{i=1}^{k+1}(w(i)-i)t-t\\
=& \frac{1}{2}(w(k)-w(k+1)-1)t.
\end{split}
\end{equation}
Since $w$ is in $\Y^{\S1}$, we know it must be of the form given
in~\eqref{eq:2.4}. If $k=j_q$ for some $1 \le q \le m$, then
$\sum_{i=1}^k i=\sum_{i=1}^k w(i)$. Otherwise, $w(k+1)=w(k)-1$.
Therefore, for any $w \in \Y^{\S1}$ and for any $k$, either \eqref{eq:3.8} or \eqref{eq:3.10} vanishes.  This implies the lemma because $\iota_2$ is injective. 
\end{proof}

It turns out that the relations obtained in the lemma above are sufficient to describe the ring.   We have the following \cite[Theorem~3.3]{FukukawaHaradaMasuda}.

\begin{Theorem} \label{theo:3.1} 
Let $n$ be a positive integer, $n \ge 2$ and $\Y \subset \fln$ the associated Peterson variety, equipped with the $\S1$-action above.  
Then the
$\S1$-equivariant cohomology ring of $\Y$ has a presentation 
\begin{equation} \label{eq:Cohomology Peterson type A}
\HS(\Y) \cong \Q[z_1,\dots,z_{n-1}, t]/J
\end{equation}
where the identification is given by sending $z_k$ to $\xi_k$ for all $k \in [n-1]$ and $t$ to $t$, 
and $J$ is the ideal generated by the quadratic polynomials
\begin{equation}\label{eq:defining relations}
z_k\left(z_k-\frac{1}{2}z_{k-1}-\frac{1}{2}z_{k+1}-t\right)\quad\text{for $k \in [n-1]$}
\end{equation}
where we take the convention $z_0=z_n=0$. 
\end{Theorem}

The proof of Theorem~\ref{theo:3.1} will use heavily the techniques of Hilbert series developed in Section~\ref{sec: background}. However, before delving into those arguments, we take a moment to observe that from Theorem~\ref{theo:3.1} we may readily obtain a presentation also of the ordinary cohomology ring $H^*(\Y)$, cf. \cite[Corollary~3.4]{FukukawaHaradaMasuda}.

\begin{Corollary} \label{coro:3.1} 
Let $\bxi_k$ denote the restriction of $\xi_k$ to $H^*(\Y)$. (In other words, $\bxi_k$ is the image of the Schubert class $\sigma_{s_k}$ under the restriction map $H^*(\fln) \to H^*(\Y)$.) Then 
\[
H^*(\Y) \cong \Q[z_1,\dots,z_{n-1}]/\check{J},
\]
which sends $z_k$ to $\bxi_k$ for all $k \in [n-1]$.
Here, $\check{J}$ is the ideal generated by 
\[
z_k\left(z_k-\frac{1}{2}z_{k-1}-\frac{1}{2}z_{k+1}\right)\quad\text{for $k \in [n-1]$}
\]
with $z_0=z_n=0$. 
\end{Corollary}

The remainder of this section is devoted to a proof of Theorem~\ref{theo:3.1}.    As we have already mentioned, 
we will use the Hilbert series techniques developed in Section~\ref{sec: background} to 
show that the relations derived in Lemma~\ref{lemm:3.1} are precisely \emph{all} of the necessary relations. 
Let us give a sketch of the structure of our argument below (and which serves as a blueprint for many of the arguments of this manuscript). Let us denote the (equivariant or ordinary) cohomology ring which we want to analyze by $R$. Here is the strategy.  First, we find a finite set of generators $\beta_1, \cdots, \beta_N$ for $R$. This immediately implies that we have a surjective ring homomorphism
\[
\varphi: \Q[x_1, \ldots, x_N] \twoheadrightarrow R,   \hspace{5mm} x_i \mapsto \beta_i. 
\]
Now suppose in addition that we find some relations $f_1, \ldots, f_s$ which must be satisfied by the $\{\beta_1,\cdots, \beta_N\}$, i.e. the $f_j$ are polynomials in $x_1, \ldots, x_N$ which have the property that $f_j(\beta_1, \cdots, \beta_N)=0$ in $R$.  Let $J$ denote the ideal in $\Q[x_1,\ldots, x_N]$ generated by the relations $f_j$. Since these are relations among the $\beta_i$, we know there is a well-defined and surjective homomorphism, which by slight abuse of notation we still denote by $\varphi$, given by 
\[
\varphi: \Q[x_1,\ldots, x_N]/J \twoheadrightarrow R, \hspace{5mm} \bar{x_i} \mapsto \beta_i
\]
where $\bar{x_i}$ denotes the equivalence class of $x_i$ in the quotient ring.   We may now wonder whether the relations $J$ are sufficient to determine the ring, i.e., whether or not $\varphi$ in the equation above is an isomorphism. This is where Hilbert series come into the picture, as we now explain.  

The basic idea is quite simple. Suppose we have a homomorphism $\varphi: S \to R$ of graded rings (so in particular, $\varphi$ preserves the grading on both sides). Suppose in addition that we know that $\varphi$ is a surjection.  This means that, as an additive homomorphism on each graded piece, the map $\varphi_\ell: S_\ell \to R_\ell$ is a surjective linear map of vector spaces. Thus it follows that 
$\dim_\Q(S_\ell) \geq \dim_\Q(R_\ell)$ for all $\ell$.  Recall that the Hilbert series $\Hilb(S, q)$ of $S$ is defined as $\sum_\ell \dim_\Q(S_\ell) q^\ell$ and similarly for $R$. We will write 
\[
\Hilb(S,q) \geq \Hilb(R,q)
\]
if and only if, by definition, $\dim_\C(S_\ell) \geq \dim_\C(R_\ell)$ for all $\ell$. With this notation in hand, if we know that there exists a surjection $\varphi: S \to R$ of graded rings, it immediately follows that $\Hilb(S,q) \geq \Hilb(R,q)$. Moreover, it is also evident that in order to prove that $\varphi$ is an isomorphism, it suffices to prove that $\dim_\Q(S_\ell) = \dim_\Q(R_\ell)$ for all $\ell$, or in other words, that 
\[
\Hilb(S,q) = \Hilb(R,q).
\]
Therefore, if we are able to compute by some other means the LHS and the RHS of the above equation, and then verify that they are equal, then we have shown that 
\[
R \cong S
\]
as desired. 
Applying the above reasoning to $S := \Q[x_1,\ldots, x_N]/J$ and $R$, we can therefore obtain the desired presentation 
\[
R \cong \Q[x_1,\ldots, x_N]/J
\]
of $R$.

We now apply the sketch of reasoning above to the case $R = H^*_S(\Y)$ and the ring $S = \Q[z_1, \cdots, z_{n-1}, t]/J$ where $J$ is the ideal given in Theorem~\ref{theo:3.1}. In fact, it will turn out that it is beneficial to simultaneously consider the presentations of the equivariant and the ordinary cohomology rings $H^*_S(\Y)$ and $H^*(\Y)$, for reasons which will become clear. 
First, we have from Proposition~\ref{prop: generators} and Lemma~\ref{lemm:3.1} that the natural homomorphism of graded rings  
\begin{equation} \label{eq:4.1}
\varphi: \Q[z_1,\dots,z_{n-1},t]/J\twoheadrightarrow \HS(\Y); \ z_k \mapsto \xi_k \ {\rm and} \ t \mapsto t 
\end{equation} 
is surjective, and by forgetting the $S$-action (as in the discussion before the statement of Corollary~\ref{coro:restriction surjective in type A}) we also have that the graded ring homomorphism 
\begin{equation} \label{eq:4.2}
\check{\varphi}: \Q[z_1,\dots,z_{n-1}]/\check{J}\twoheadrightarrow H^*(\Y); \ z_k \mapsto \bxi_k
\end{equation} 
is also surjective.

As explained above, this means that the Hilbert series of the rings appearing in \eqref{eq:4.1} and \eqref{eq:4.2} satisfy 
\begin{eqnarray} 
\Hilb(\Q[z_1,\dots,z_{n-1},t]/J,q)&\ge \Hilb(\HS(\Y),q) \label{eq:4.3}\\
\Hilb(\Q[z_1,\dots,z_{n-1}]/\check{J},q)&\ge \Hilb(H^*(\Y),q). \label{eq:4.4}
\end{eqnarray}
We have seen that $\varphi$ (resp. $\check{\varphi}$) is an isomorphism if and only if the inequality in~\eqref{eq:4.3} (resp. ~\eqref{eq:4.4}) is in fact an equality.  Thus, we are now reduced to showing equalities of Hilbert series, and to do so, we can use
the techniques outlined in Section~\ref{subsec: background on Hilbert series and regular sequences}. 

Indeed, the Hilbert series of the right hand sides of~\eqref{eq:4.3} and~\eqref{eq:4.4} are known. For instance, in \cite{ST}
 it is shown that  
\begin{equation} \label{eq:4.5}
\Hilb(H^*(\Y),q)=(1+q^2)^{n-1}
\end{equation}
and since $\HS(\Y)=H^*(B{\S1})\otimes H^*(\Y)$ as $H^*(B{\S1})$-modules, \eqref{eq:4.5}  immediately implies 
\begin{equation} \label{eq:4.5-1}
\Hilb(\HS(\Y),q)=\frac{(1+q^2)^{n-1}}{1-q^2}.
\end{equation}  
Thus the remaining task is to compute the Hilbert series of the LHS of~\eqref{eq:4.3} and~\eqref{eq:4.4} and checking that they are equal to the RHS of~\eqref{eq:4.5} and~\eqref{eq:4.5-1} respectively. 
The following lemma computes the left hand side of~\eqref{eq:4.4}.

\begin{Lemma} \label{lemm:4.1}
$\Hilb(\Q[z_1,\dots,z_{n-1}]/\check{J},q)=(1+q^2)^{n-1}$.  
\end{Lemma}

In what follows, we take a slightly unconventional approach, in the following sense. We first prove Theorem~\ref{theo:3.1} by assuming Lemma~\ref{lemm:4.1}.  The argument is a straightforward application of Hilbert series ideas. We will then return to
a discussion of the proof of Lemma~\ref{lemm:4.1}, where in fact we will see it as a special case of a more general statement.

\begin{proof}[Proof of Theorem~\ref{theo:3.1} assuming Lemma~\ref{lemm:4.1}]
We first show that in the polynomial ring $\Q[z_1,\dots,z_{n-1},t]$, the sequence 
\[
\begin{split}
\f_k:&=z_k(z_k-\frac{1}{2}z_{k-1}-\frac{1}{2}z_{k+1}-t) \quad \text{for $k \in [n-1]$},\\
\f_n:&=t.
\end{split}
\]
is regular. 
To see this, first observe that since $\f_n=t,$ 
from the definitions of $\f_k$ and the ideals $J$ and $\check{J}$ given in the statements of Theorem~\ref{theo:3.1} and Corollary~\ref{coro:3.1} it follows that 
\[
\begin{split}
&\Hilb(\Q[z_1,\dots,z_{n-1},t]/(\f_1,\dots,\f_{n-1},\f_n),q)\\
=&\Hilb(\Q[z_1,\dots,z_{n-1}]/\check{J},q)\\
=&(1+q^2)^{n-1}
\end{split}
\]
where the last equality follows from Lemma~\ref{lemm:4.1}. This implies that \eqref{eq:4.6} is satisfied in our setting because $\deg \f_i=4$ for $ i \in [n-1]$, $\deg \f_n=2$ and 
\begin{equation} \label{eq:4.6-2}
\Hilb(\Q[z_1,\dots,z_{n-1},t],q)=\frac{1}{(1-q^2)^n}. 
\end{equation} 
Thus the claim follows from Lemma~\ref{lemma: hilbert regular criterion}.

Now, from the definition of a regular sequence it is clear that the subsequence $\f_1,\dots,\f_{n-1}$ of a regular sequence $\f_1,\dots,\f_n$ is again a regular sequence. Hence it follows from Lemma~\ref{lemma: hilbert regular criterion} and \eqref{eq:4.6-2} that 
\[
\begin{split}
\Hilb(\Q[z_1,\dots,z_{n-1},t]/J,q)&=\Hilb(\Q[z_1,\dots,z_{n-1},t]/(\f_1,\dots,\f_{n-1}),q)\\
&=\frac{1}{(1-q^2)^n}\prod_{i=1}^{n-1}(1-q^{\deg \f_i})\\
&=\frac{(1+q^2)^{n-1}}{1-q^2}.
\end{split}
\]
This together with \eqref{eq:4.5-1} shows that the equality holds in \eqref{eq:4.3}. Hence the map $\varphi$ in \eqref{eq:4.1} is an isomorphism, as desired. 
\end{proof}

To complete the argument, we must now prove Lemma~\ref{lemm:4.1}.  In fact, we will put the lemma in a more general setting. 
Note first that by Lemma~\ref{lemma: solution regular criterion} and the computation in Example~\ref{example.Hilbert.polynomial.ring}, in order to prove Lemma~\ref{lemm:4.1} it suffices to show that 
the solution set in $\C^{n-1}$ of the equations 
\begin{equation}\label{eq:5.0} 
z_k \left(z_k-\frac{1}{2}z_{k-1}-\frac{1}{2}z_{k+1} \right) \, \, \textup{ for } k \in [n-1]
\end{equation} 
(where $z_0=z_n=0$ by convention) consists only of the origin $\{0\}$. 
For the remainder of this argument, we will not recount all details and refer the reader to \cite{FukukawaHaradaMasuda}. 
The idea is to consider a more general set of 
equations in $\C^{n-1}$ (for $n \ge 3$), namely:   
\begin{equation} \label{eq:5.1}
\begin{split}
z_1^2&=b_1z_1z_2\\
z_k^2 &=z_k(a_{k-1}z_{k-1} +b_kz_{k+1}) \quad  \textup{ for } 2 \leq k \leq n-2 \\
z_{n-1}^2&=a_{n-2}z_{n-2}z_{n-1}
\end{split}
\end{equation}
where $a_k,b_k$ for $k=1,2,\dots,n-2$ are fixed complex numbers. We have the following. 

\begin{Lemma} \label{lemm:5.1}
In the setting above, set $c_i :=a_ib_i$ for $i=1,2,\dots,n-2$.  If 
\begin{equation} \label{eq:5.2}
 1-\cfrac{c_i}{1-\cfrac{c_{i+1}}{\qquad \cfrac{\ddots}{1-\cfrac{c_{j-1}}{1-c_j}}}}\not=0
\end{equation}
for all $1\le i\le j\le n-2$, 
then the solution set of the equations \eqref{eq:5.1} consists of only the origin in $\C^{n-1}$. 
\end{Lemma}

The proof of the above lemma proceeds by induction and uses only elementary techniques; see \cite[Lemma~5.2]{FukukawaHaradaMasuda}.  

Now, we return to the special case for which $a_k=b_k=\frac{1}{2}$ and hence $c_k=\frac{1}{4}$ for all $k \in [n-2]$.  Below, we give a sufficient condition for \eqref{eq:5.2} to be satisfied when $c_k=a_kb_k$ (for $k \in [n-2]$) is a constant $c$ independent of $k$.  This special case suffices to prove Lemma~\ref{lemm:4.1}. For this purpose, consider the numerical sequence $\{x_m\}_{m=0}^\infty$ defined by the following recurrence relation and with $x_0=1$: 
\begin{equation} \label{eq:5.5}
x_{m}=1-\frac{c}{x_{m-1}} \quad\text{for $m\ge 1$}.  
\end{equation}
In the situation when the $c_k$ are all equal, it is straightforward to see that the condition \eqref{eq:5.2} is equivalent to the statement that $x_m\not=0$ for $m \in [n-2]$.  
We have the following; see \cite[Lemma~5.4]{FukukawaHaradaMasuda} for the proof. 

\begin{Lemma}\label{lemm:5.2}
Let $\{x_m\}$ be the sequence defined in~\eqref{eq:5.5}.  
If $c$ is any real number $\le 1/4$, then $x_m>0$ for all  $m\ge 1$.  
\end{Lemma}

The proof of Lemma~\ref{lemm:4.1} is now straightforward. 

\begin{proof}[Proof of Lemma~\ref{lemm:4.1}]
The statement of Lemma~\ref{lemm:4.1} follows from Lemmas~\ref{lemma: hilbert regular criterion}, \ref{lemma: solution regular criterion}, \ref{lemm:5.1} and \ref{lemm:5.2}.
\end{proof}

\section{ The cohomology rings of Peterson varieties in general Lie type} 
\label{section: peterson general Lie type} 

In Section~\ref{section: peterson in type A},  we focused on the special case of Lie type A, when $G=GL_n(\C)$. 
In this section we briefly recount analogous results that can be obtained for (equivariant) cohomology rings of Peterson varieties in arbitrary Lie type. (In the type A case, because we can make explicit arguments with matrices and linear algebra we often blur the distinction between the reductive $GL_n(\C)$ and the semisimple $SL_n(\C)$, but in arbitrary Lie type, we restrict to semisimple groups.)
We keep the account brief, since in many ways it parallels the results in the $GL_n(\C)$ case. We refer the reader to \cite{Dre, HarHorMas} for details.

For this section, we fix $G$ a complex semisimple linear algebraic group of rank $n$. 
We fix $B$ a Borel subgroup and $T$ a maximal torus of $G$ such that $T \subset B \subset G$. 
We denote their Lie algebras by $\mathfrak{t}\subset\mathfrak{b}\subset\mathfrak{g}$ and $W$ is the associated Weyl group.
We also denote by $\mathfrak{g}_{\alpha} \subset \mathfrak{g}$ the root space for each root $\alpha$ and fix a set of simple positive roots $\Delta=\{\alpha_1,\dots,\alpha_n\}$.
An element $x \in \mathfrak{g}$ is \textbf{regular} if its $G$-orbit of the adjoint action has the largest possible dimension.

\begin{Definition} \label{defpet}
Let $E_{\alpha}$ be a basis element of the root space $\mathfrak{g}_{\alpha}$ and let $N_0 :=\sum_{\alpha\in\Delta}E_{\alpha }$, a regular nilpotent element. (We refer the reader to \cite[Theorem~5.3]{Kos59} for a proof of the regularity of $N_0$.) In this setting we define the \textbf{Peterson variety (associated to $\mathfrak{g}$)} as 
\begin{equation*} \label{eqS}
Pet:=\{gB\in G/B \mid Ad(g^{-1})(N_0)\in \mathfrak{b}\oplus\displaystyle\bigoplus_{\alpha\in-\Delta}\mathfrak{g}_{\alpha}\}, 
\end{equation*}
where $-\Delta$ denotes the set of simple negative roots, obtained by negating the simple positive roots. \footnote{It is also possible to define this Hessenberg subspace without reference to the choice $-\Delta$ of simple negative roots. For instance, it can be defined as the annihilator of $[\mathfrak{n}, \mathfrak{n}]$ with respect to the Killing form, where $\mathfrak{n}$ denotes the nilpotent radical of $\mathfrak{b}$.} 
\end{Definition}

As in the previous section, the action of the maximal torus $T$ on $G/B$ does not in general preserve 
the Peterson variety. However, using the homomorphism 
$\phi:T\rightarrow(\mathbb{C}^*)^n$ defined by $g\mapsto (\alpha_1(g),\dots,\alpha_n(g))$, we define
a one-dimensional subtorus $\mathsf{S}$ as the connected component of the identity in
\begin{equation*} 
\phi^{-1}(\{(c,c,\dots,c) \mid c\in\mathbb{C}^*\}). 
\end{equation*}
The restriction of the $T$-action on $G/B$ to the subgroup $\mathsf{S}$ does preserve $Pet$ \cite[Lemma~5.1-(3)]{HarTym-Poset}. 
The simple root $\alpha_i : T \to \C^*$ determines the associated line bundle $ET\times _{T}\C_{\alpha_i} \to BT$. 
By abuse of notation, we denote by $\alpha_i$ the first Chern class of this line bundle 
\begin{align}\label{deg of alphai in coh all Lie type}
\alpha_i=c_1^T(\C_{\alpha_i}) \in H^2(BT), \ \ \ i \in [n].
\end{align} 
We then identify $H^*(BT)$ with the polynomial ring $\mathbb{Q}[\alpha_1,\ldots,\alpha_n]$. 
Let $\nu$ be the character of $\mathsf{S}$ defined by the composition of the inclusion $\mathsf{S} \hookrightarrow T$ and $\phi:T\rightarrow(\mathbb{C}^*)^n$ where we may regard the image $\{(c,c,\dots,c) \mid c\in \C^*\}$ as isomorphic to $\C^*$ in the natural way.
Considering the associated line bundle $E\mathsf{S}\times _{\mathsf{S}}\C_\nu \to B\mathsf{S}$, we denote by $t$ the first Chern class of this line bundle, namely
\begin{align}\label{deg of t in coh all Lie type}
t=c_1^\mathsf{S}(\C_{\nu}) \in H^2(B\mathsf{S}).
\end{align} 
We also have the identification $H^*(B\mathsf{S})\cong \mathbb{Q}[t]$. 
By a similar argument to that for Lemma~\ref{lemma: ti goes to it}, we obtain the following lemma.

\begin{Lemma} \label{lemma: alphai goes to t in all types}
The homomorphism $\pi: H^*(BT) \to H^*(B\mathsf{S})$ induced from the inclusion $\mathsf{S} \hookrightarrow T$ sends $\alpha_i$ in \eqref{deg of alphai in coh all Lie type} to $t$ for all $i \in [n]$. 
\end{Lemma}

\subsection{Peterson Schubert classes in general Lie type}\label{subsec.drellich}

As in Section~\ref{subsection:basis_Peterson_typeA}, the 
$\mathsf{S}$-fixed points $Pet^\mathsf{S}$ of the Peterson variety satisfy the relation
\[
Pet^\mathsf{S}=Pet\cap (G/B)^T
\]
and hence we may view $Pet^\mathsf{S}$ as a subset of the Weyl group $W$. Indeed, the fixed point set $Pet^\mathsf{S}$ may be described concretely as follows. For 
a subset $K$ of the set $\Delta$ simple roots, let $W_K$ denote the parabolic subgroup generated by the simple reflections indexed by $K$ and let $w_K$ denote 
the longest element of $W_K$. Then it is known \cite[Proposition 5.8]{HarTym-Poset} that 
\begin{equation*}
Pet^\mathsf{S}=\{w_K \mid K\subset\Delta \}.
\end{equation*}

In \cite{Dre}, Drellich defines an element $v_{K}$ in the Weyl group $W$ which are analogous to the $v_{\mathcal{A}}$ which appeared in the previous section. To define them, we need some preliminaries. First recall that there is a one-to-one correspondence between the set of vertices of the Dynkin diagram of $\Phi$ and the set of the simple roots $\Delta=\{\alpha_1,\ldots,\alpha_n\}$. 
Throughout this section, we assume to be fixed an ordering of the simple roots as given in \cite[p.58]{Humphreys-LieAlgebra}. (See Figure~\ref{pic: Dynkin diagrams}.)

\begin{figure}[h]
\setlength{\unitlength}{1mm}
\begin{center} 
  \begin{picture}(50, 15)(0,0)
  \put(-25,9){Type $A_n$:}
  \put(0,10){\circle{2}}
  \put(10,10){\circle{2}}
  \put(20,10){\circle{2}}
  \put(50,10){\circle{2}}
  \put(60,10){\circle{2}}
  \put(1,10){\line(1,0){8}}
  \put(11,10){\line(1,0){8}}
  \put(21,10){\line(1,0){8}}
  \put(33,8.8){$\cdots$}
  \put(41,10){\line(1,0){8}}
  \put(51,10){\line(1,0){8}}
  \put(-2,3){$\alpha_1$}
  \put(8,3){$\alpha_2$}
  \put(18,3){$\alpha_3$}
  \put(48,3){$\alpha_{n-1}$}
  \put(58,3){$\alpha_n$}
  \end{picture}
\end{center}  

\setlength{\unitlength}{1mm}
\begin{center}
  \begin{picture}(50, 15)(0,0)
  \put(-25,9){Type $B_n$:}
  \put(0,10){\circle{2}}
  \put(10,10){\circle{2}}
  \put(20,10){\circle{2}}
  \put(50,10){\circle{2}}
  \put(60,10){\circle{2}}
  \put(1,10){\line(1,0){8}}
  \put(11,10){\line(1,0){8}}
  \put(21,10){\line(1,0){8}}
  \put(33,8.8){$\cdots$}
  \put(41,10){\line(1,0){8}}
  \put(51,10.5){\line(1,0){8}}
  \put(51,9.5){\line(1,0){8}}
  \put(-2,3){$\alpha_1$}
  \put(8,3){$\alpha_2$}
  \put(18,3){$\alpha_3$}
  \put(48,3){$\alpha_{n-1}$}
  \put(58,3){$\alpha_n$}
  \qbezier(54.5,11.5)(57.5,10)(54.5,8.5)
  \end{picture}
\end{center}  

\setlength{\unitlength}{1mm}
\begin{center}
  \begin{picture}(50, 15)(0,0)
  \put(-25,9){Type $C_n$:}
  \put(0,10){\circle{2}}
  \put(10,10){\circle{2}}
  \put(20,10){\circle{2}}
  \put(50,10){\circle{2}}
  \put(60,10){\circle{2}}
  \put(1,10){\line(1,0){8}}
  \put(11,10){\line(1,0){8}}
  \put(21,10){\line(1,0){8}}
  \put(33,8.8){$\cdots$}
  \put(41,10){\line(1,0){8}}
  \put(51,10.5){\line(1,0){8}}
  \put(51,9.5){\line(1,0){8}}
  \put(-2,3){$\alpha_1$}
  \put(8,3){$\alpha_2$}
  \put(18,3){$\alpha_3$}
  \put(48,3){$\alpha_{n-1}$}
  \put(58,3){$\alpha_n$}
  \qbezier(55.5,11.5)(52.5,10)(55.5,8.5)
  \end{picture}
\end{center}  

\setlength{\unitlength}{1mm}
\begin{center}
  \begin{picture}(50, 20)(0,0)
  \put(-25,9){Type $D_n$:}
  \put(0,10){\circle{2}}
  \put(10,10){\circle{2}}
  \put(20,10){\circle{2}}
  \put(50,10){\circle{2}}
  \put(60,4){\circle{2}}
  \put(60,16){\circle{2}}
  \put(1,10){\line(1,0){8}}
  \put(11,10){\line(1,0){8}}
  \put(21,10){\line(1,0){8}}
  \put(33,8.8){$\cdots$}
  \put(41,10){\line(1,0){8}}
  \put(51,10.5){\line(3,2){8}}
  \put(51,9.5){\line(3,-2){8}}
  \put(-2,3){$\alpha_1$}
  \put(8,3){$\alpha_2$}
  \put(18,3){$\alpha_3$}
  \put(48,3){$\alpha_{n-2}$}
  \put(63,15){$\alpha_{n-1}$}
  \put(63,3){$\alpha_n$}
  \end{picture}
\end{center}  

\setlength{\unitlength}{1mm}
\begin{center}
  \begin{picture}(50, 25)(0,0)
  \put(-25,9){Type $E_6$:}
  \put(0,10){\circle{2}}
  \put(10,10){\circle{2}}
  \put(20,10){\circle{2}}
  \put(30,10){\circle{2}}
  \put(40,10){\circle{2}}
  \put(20,20){\circle{2}}
  \put(1,10){\line(1,0){8}}
  \put(11,10){\line(1,0){8}}
  \put(21,10){\line(1,0){8}}
  \put(31,10){\line(1,0){8}}
  \put(20,11){\line(0,1){8}}
  \put(-2,3){$\alpha_1$}
  \put(8,3){$\alpha_3$}
  \put(18,3){$\alpha_4$}
  \put(28,3){$\alpha_5$}
  \put(38,3){$\alpha_6$}
  \put(22,19){$\alpha_2$}
  \end{picture}
\end{center}  

\setlength{\unitlength}{1mm}
\begin{center}
  \begin{picture}(50, 25)(0,0)
  \put(-25,9){Type $E_7$:}
  \put(0,10){\circle{2}}
  \put(10,10){\circle{2}}
  \put(20,10){\circle{2}}
  \put(30,10){\circle{2}}
  \put(40,10){\circle{2}}
  \put(50,10){\circle{2}}
  \put(20,20){\circle{2}}
  \put(1,10){\line(1,0){8}}
  \put(11,10){\line(1,0){8}}
  \put(21,10){\line(1,0){8}}
  \put(31,10){\line(1,0){8}}
  \put(41,10){\line(1,0){8}}
  \put(20,11){\line(0,1){8}}
  \put(-2,3){$\alpha_1$}
  \put(8,3){$\alpha_3$}
  \put(18,3){$\alpha_4$}
  \put(28,3){$\alpha_5$}
  \put(38,3){$\alpha_6$}
  \put(48,3){$\alpha_7$}
  \put(22,19){$\alpha_2$}
  \end{picture}
\end{center}  

\setlength{\unitlength}{1mm}
\begin{center}
  \begin{picture}(50, 25)(0,0)
  \put(-25,9){Type $E_8$:}
  \put(0,10){\circle{2}}
  \put(10,10){\circle{2}}
  \put(20,10){\circle{2}}
  \put(30,10){\circle{2}}
  \put(40,10){\circle{2}}
  \put(50,10){\circle{2}}
  \put(60,10){\circle{2}}
  \put(20,20){\circle{2}}
  \put(1,10){\line(1,0){8}}
  \put(11,10){\line(1,0){8}}
  \put(21,10){\line(1,0){8}}
  \put(31,10){\line(1,0){8}}
  \put(41,10){\line(1,0){8}}
  \put(51,10){\line(1,0){8}}
  \put(20,11){\line(0,1){8}}
  \put(-2,3){$\alpha_1$}
  \put(8,3){$\alpha_3$}
  \put(18,3){$\alpha_4$}
  \put(28,3){$\alpha_5$}
  \put(38,3){$\alpha_6$}
  \put(48,3){$\alpha_7$}
  \put(58,3){$\alpha_8$}
  \put(22,19){$\alpha_2$}
  \end{picture}
\end{center}  

\setlength{\unitlength}{1mm}
\begin{center}
  \begin{picture}(30, 15)(0,0)
  \put(-35,9){Type $F_4$:}
  \put(0,10){\circle{2}}
  \put(10,10){\circle{2}}
  \put(20,10){\circle{2}}
  \put(30,10){\circle{2}}
  \put(1,10){\line(1,0){8}}
  \put(11,10.5){\line(1,0){8}}
  \put(11,9.5){\line(1,0){8}}
  \put(21,10){\line(1,0){8}}
  \put(-2,3){$\alpha_1$}
  \put(8,3){$\alpha_2$}
  \put(18,3){$\alpha_3$}
  \put(28,3){$\alpha_4$}
  \qbezier(14.5,11.5)(17.5,10)(14.5,8.5)
  \end{picture}
\end{center}  

\setlength{\unitlength}{1mm}
\begin{center}
  \begin{picture}(10, 15)(0,0)
  \put(-45,9){Type $G_2$:}
  \put(0,10){\circle{2}}
  \put(10,10){\circle{2}}
  \put(1,10){\line(1,0){8}}
  \put(1,10.7){\line(1,0){8}}
  \put(1,9.3){\line(1,0){8}}
  \put(-2,3){$\alpha_1$}
  \put(8,3){$\alpha_2$}
  \qbezier(5.5,11.5)(2.5,10)(5.5,8.5)
  \end{picture}
\end{center}  
\caption{The Dynkin diagrams.}
\label{pic: Dynkin diagrams}
\end{figure}
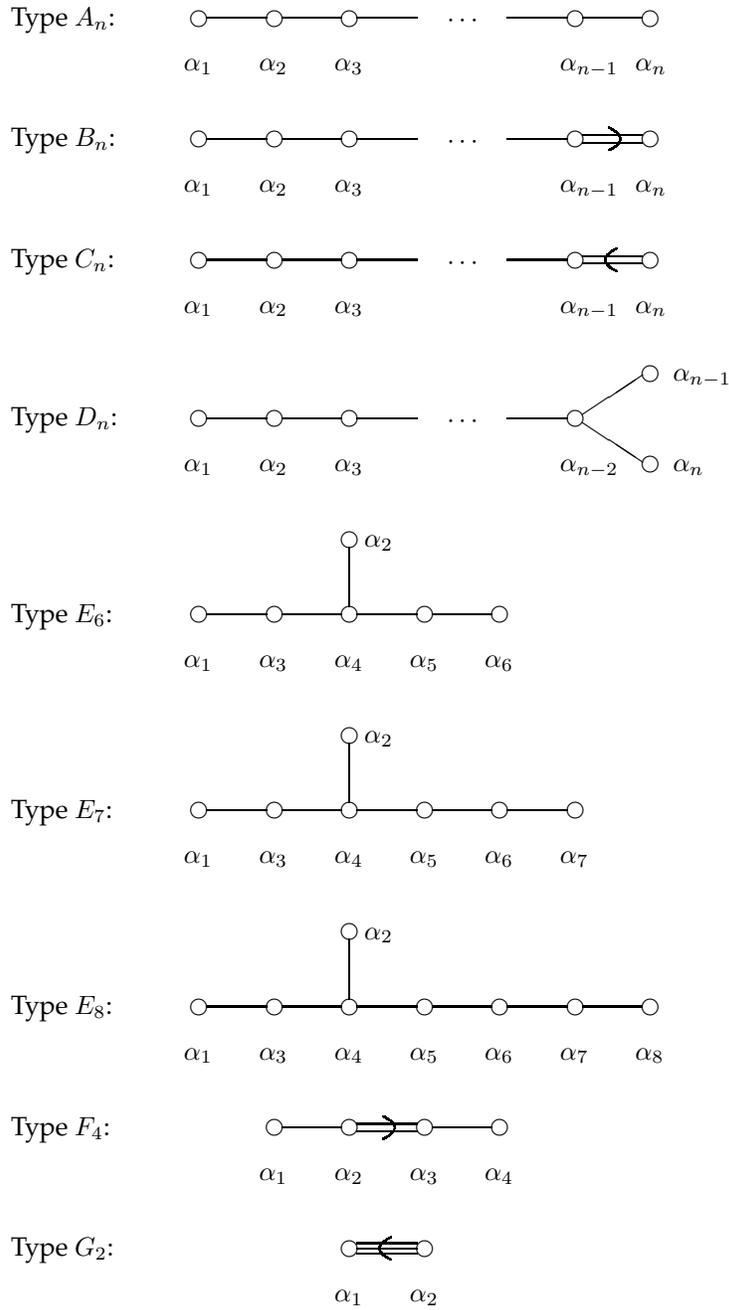  We say that a subset $K \subset \Delta$ of simple roots is \textbf{connected} if the induced Dynkin diagram of $K$ is a connected subgraph of the Dynkin diagram of $\Delta$. 
We denote by $\Phi_K$ the root subsystem associated with a connected subset $K \subset \Delta$.

\begin{Example}
Suppose $\Phi$ is of type $E_7$. We take the ordering of the simple roots $\alpha_1,\ldots,\alpha_7$ given in Figure~\ref{pic: Dynkin diagrams}.
If we take a connected subset $K=\{\alpha_2,\alpha_4,\alpha_5,\alpha_6,\alpha_7 \}$, then $\Phi_K$ is of type $A_5$.
If we take $K'=\{\alpha_2,\alpha_3,\alpha_4,\alpha_5 \}$, then $\Phi_{K'}$ is of type $D_4$.
\end{Example}

Every subset $K \subset \Delta$ can be written as $K = K_1 \times \cdots \times K_m$ where for each $i \in [m]$, $K_i$ is a maximal connected subset. 
Note that each connected subset corresponds to its own Lie type, and these may be distinct. 

\begin{Definition} \label{def:vK}
For a subset $K\subset \Delta$, we define $v_K \in W$ as follows.
\begin{enumerate}
\item For a connected subset $K \subset \Delta$, we fix an ordering of the simple roots in $\Phi_K$ as in Figure~\ref{pic: Dynkin diagrams}.
Then we define $v_K \in W$ as 
\begin{align*}
v_K := \prod_{{\rm{Root}}_K(i)=1}^{|K|} s_i 
\end{align*}
where ${\rm{Root}}_K(i)$ is the fixed index of the corresponding root in the root system $\Phi_K$. 
\item If a subset $K\subset \Delta$ can be decomposed into maximal connected subsets as $K = K_1 \times K_2 \times \cdots \times K_m$, then we define $v_K := v_{K_1} v_{K_2} \cdots v_{K_m}$.
\end{enumerate}
\end{Definition}

\begin{Example} \label{ex:vK generalLietype}
Continuing with the previous example, for the case $K=\{\alpha_2,\alpha_4,\alpha_5,\alpha_6,\alpha_7 \}$, we fix an ordering of elements of $K$ by ${\rm{Root}}_K(2)=1, {\rm{Root}}(4)=2, {\rm{Root}}_K(5)=3, {\rm{Root}}_K(6)=4, {\rm{Root}}_K(7)=5$. 
Then we have $v_K=s_2s_4s_5s_6s_7$.
For the case $K'=\{\alpha_2,\alpha_3,\alpha_4,\alpha_5 \}$ in $\Phi_{E_7}$, we fix an ordering of elements of $K'$ by ${\rm{Root}}_{K'}(5)=1, {\rm{Root}}_{K'}(4)=2, {\rm{Root}}_{K'}(2)=3, {\rm{Root}}_{K'}(3)=4$, one has $v_{K'}=s_5s_4s_2s_3$.
\end{Example}

Drellich proves an analogue of Theorem~\ref{theorem:pvA-basis} for arbitrary Lie type as follows; see \cite[Theorem~3.5]{Dre}. 

\begin{Theorem} 
\label{thm:basis}
Let $Pet$ be the Peterson variety variety associated to $\mathfrak{g}$. Then 
the Peterson Schubert classes $\lbrace p_{v_K} : K\subset \Delta \rbrace$ are a $\HS(\pt)$-module basis of $H_{\mathsf{S}}^*(Pet)$. 
\end{Theorem}

\begin{Remark}
The definition of $v_K$ depends on a choice of an ordering of elements of $K$, but the Peterson Schubert class $p_{v_K}$ does not depend on these choices \cite[Remark~7.12]{Horiguchi-mixedEuler}. 
For instance, in Example~\ref{ex:vK generalLietype}, if we change the ordering of the elements of $K$ to be ${\rm{Root}}_{K}(7)=1, {\rm{Root}}_{K}(6)=2, {\rm{Root}}_{K}(5)=3, {\rm{Root}}_{K}(4)=4, {\rm{Root}}_{K}(2)=5$, then we obtain $v_{K}=s_7s_6s_5s_4s_2$, which is different from the element $s_2s_4s_5s_6s_7$. 
However, the Peterson Schubert classes associated to these elements are equal.  
\end{Remark}

The proof of Theorem~\ref{thm:basis}, which we omit, uses techniques analogous to those of Section~\ref{section: peterson in type A}, among other techniques. As a corollary, we obtain the following result.

\begin{Corollary} \label{coro:restriction surjective in all Lie type}
The restriction homomorphism $H^*_T(G/B) \to H^*_\mathsf{S}(Pet)$ is surjective.
In particular, the (ordinary) restriction homomorphism $H^*(G/B) \to H^*(Pet)$ is also surjective.
\end{Corollary}

Finally, we introduce the Monk formula and the Giambelli formula for $Pet$, as given by Drellich \cite[Theorem~4.2, Theorem~5.3, Theorem~5.5]{Dre}.   For a brief account of further developments since the work of Drellich, we refer the reader to Section~\ref{section.further.developments}.

\begin{Theorem} \label{theorem:Monk in all types} (\cite[Theorem~4.2]{Dre})
Let $Pet$ be a Peterson variety associated to $\mathfrak{g}$. 
Let $i \in [n]$,
let $v_K \in S_n$ be the permutation given in Definition~\ref{def:vK}, and let $p_{v_K}$ the corresponding Peterson Schubert class. 
Then 
\begin{align*} 
p_{s_i} \cdot p_{v_K} =p_{s_i}(w_K)p_{v_K} + \sum_{J \supsetneq K \atop |J|=|K|+1} c_{i,K}^J \, p_{v_J}, 
\end{align*}
where the coefficients $c_{i,K}^J$ are non-negative rational numbers given by 
\begin{align*}
c_{i,K}^J = \big(p_{s_i}(w_J)-p_{s_i}(w_K)\big) \cdot \frac{p_{v_K}(w_J)}{p_{v_J}(w_J)}.
\end{align*}
\end{Theorem}

\begin{Theorem}\label{theorem_Giambelli in all types} (\cite[Theorems~5.3 and 5.5]{Dre})
In the setting of the theorem above, 
\begin{enumerate}
\item if $K$ and $J$ are connected with $K \cap J \neq \emptyset$ and $K \cup J$ is disconnected, then we have 
\begin{align*} 
 p_{v_{K \cup J}} = p_{v_K} \cdot p_{v_J} 
\end{align*}
\item if $K$ is connected, then 
\begin{align*} 
p_{v_K} = \frac{|\mbox{Red}(v_K)|}{|K|!} \prod_{\alpha_i \in K} p_{s_i}
\end{align*} 
where $\mbox{Red}(v_K)$ is the set of reduced words for $v_K$.
\end{enumerate}
\end{Theorem}

\subsection{A presentation of the cohomology ring of Peterson varieties in arbitrary Lie type}

We maintain the notation and setup of algebraic groups and Lie algebras given above. In this section we quickly explain the presentation, by generators and relations, of the cohomology ring of Peterson varieties $Pet$ in arbitrary Lie type. As in Section~\ref{subsec.drellich} we keep the discussion brief and refer the reader to \cite{HarHorMas} for details. 

As before, we fix a choice of $G$. We begin with the cohomology ring of the flag variety $G/B$. 
We have already defined above the characters $\chi: T \to \C^*$ from the torus $T$. Note that a character on $T$ can be easily extended to the Borel subgroup $B$ by using the abelianization map $B \to B/[B,B] \cong T$. 
For $\chi \in \mathrm{Hom}(T, \C^*)$ we denote by $\tilde{\chi}$ the extension to $B$ defined by the composition $\tilde{\chi}: B \to B/[B,B] \cong T \xrightarrow{ \ \chi \ } \C^*$. With this notation in place, 
we can associate to a character $\chi: T \to \C^*$ the complex line bundle $L_\chi:=G \times_B \C_{\tilde{\chi}}$ over $G/B$, where $\C_{\tilde{\chi}}$ is the complex one-dimensional $B$-representation defined by the character $\tilde{\chi}$, and $G \times_B \C_{\tilde{\chi}}$ is the quotient space of $G \times \C_{\tilde{\chi}}$ by the $B$-action $b \cdot (g,z):=(gb^{-1},\tilde{\chi}(b)z)$ for $(g,z) \in G \times \C_{\tilde{\chi}}$ and $b \in B$.
Denote the symmetric algebra of the $\Q$-vector space $\mathfrak{t}^*_\Z \otimes_{\Z}\Q$ by 
\begin{align*}
\mathcal{R}:=\text{Sym}_{\Q}(\mathfrak{t}^*_\Z \otimes_{\Z}\Q). 
\end{align*}
Then  the assignment $\chi \mapsto c_1(L_{\chi}^*) \in H^2(G/B)$ induces a homomorphism
\begin{align} \label{eq:the map to flag in all Lie types}
\mathcal{R} \to H^*(G/B)
\end{align}
(where the map above doubles the grading on $\mathcal{R}$).  
The following is a famous result of Borel \cite{Borel}. 

\begin{Theorem}\label{theorem:Borel} (\cite{Borel})
The map in \eqref{eq:the map to flag in all Lie types} is surjective and its kernel is the ideal in $\mathcal{R}$ generated by $W$-invariants. 
\end{Theorem}

It follows from Corollary~\ref{coro:restriction surjective in all Lie type} that the surjective map in \eqref{eq:the map to flag in all Lie types} composed with the restriction map  \begin{align} \label{eq:the map to pet in all Lie types}
\mathcal{R} \twoheadrightarrow H^*(G/B) \twoheadrightarrow H^*(Pet)
\end{align}
is also surjective. Next, we wish to analyze the kernel. We have the following from \cite{HarHorMas}. 

\begin{Theorem} \label{mainthm Lie terminology} (\cite[Theorem~4.1]{HarHorMas})
The kernel of the surjective map in \eqref{eq:the map to pet in all Lie types} is generated by the quadratic forms 
\begin{align*}
\alpha_i \varpi_i \ \ \ {\rm for} \ i \in [n]
\end{align*}
where $\varpi_1, \ldots, \varpi_n$ denote the fundamental weights. 
In particular, we obtain the following presentation  
\begin{align*}
H^*(Pet) \cong \mathcal{R}/(\alpha_i \varpi_i \mid i \in [n]).
\end{align*}
\end{Theorem}

In what follows, we take the fundamental weights $\varpi_1, \ldots, \varpi_n$ as our chosen basis of $\mathfrak{t}^*_\Z \otimes_{\Z}\Q$.
Then we have 
\begin{align*}
\mathcal{R} \cong \Q[\varpi_1,\ldots,\varpi_n].
\end{align*}
Since the simple root $\alpha_i$ can be written as $\alpha_i=\sum_{j=1}^n \langle\alpha_i,\alpha_j\rangle \varpi_j$ where $\langle\alpha_i,\alpha_j\rangle$ denotes the Cartan integer, the quadratic form $\alpha_i \varpi_i$ appearing in Theorem~\ref{mainthm Lie terminology} can be written 
\begin{align*}
\alpha_i \varpi_i = \sum_{j=1}^n \langle\alpha_i,\alpha_j\rangle \varpi_i \varpi_j.
\end{align*}
It is known that the ring homomorphism in \eqref{eq:the map to flag in all Lie types} sends $\varpi_k$ to the Schubert class $\sigma_{s_k} \in H^2(G/B)$, see e.g. \cite{BGG}. \footnote{We can also see it directly if we know about divided difference operators. Since one can write $\varpi_k=\sum_{j=1}^n c_j \sigma_{s_j}$ in $H^2(G/B)$ where $\varpi_k$ denotes the image of $\varpi_k$ under the map in \eqref{eq:the map to flag in all Lie types}, applying the divided difference operator $\partial_j$ on both sides yields $\delta_{kj}=c_j$ for $j \in [n]$, from which we get $\varpi_k=\sigma_{s_k}$.}
In order to prove Theorem~\ref{mainthm Lie terminology}, it is necessary to first show that the ring homomorphism 
\begin{equation} 
\check\varphi: \Q[\varpi_1,\dots,\varpi_n]/(\sum_{j=1}^n \langle\alpha_i,\alpha_j\rangle \varpi_i \varpi_j \mid i \in [n]) \twoheadrightarrow H^*(Pet); \ \ \ \varpi_k \mapsto \bxi_k 
\end{equation}
is well-defined where $\bxi_k$ denotes the image of the (ordinary) Schubert class $\sigma_{s_k}$ under the restriction map $H^*(G/B) \to H^*(Pet)$.
For this purpose, we begin with some elementary computations involving Peterson Schubert classes. First, from 
Monk's formula (Theorem~\ref{theorem:Monk in all types}) applied to the case $K=\{\alpha_i\}$ and $v_K=s_i$ we obtain 
\begin{equation} \label{eq:3-1}
p_{s_i}^2=p_{s_i}(s_i)\cdot p_{s_i}+\displaystyle\sum_{j\neq i} c_i^j\cdot p_{v_{\{\alpha_i,\alpha_j\}}} 
\end{equation}
where 
\begin{equation} \label{eq:3-2}
c_i^j=(p_{s_i}(w_{\{\alpha_i,\alpha_j\}})-p_{s_i}(s_i))\cdot \frac {p_{s_i}(w_{\{\alpha_i,\alpha_j\}})}{p_{v_{\{\alpha_i,\alpha_j\}}}(w_{\{\alpha_i,\alpha_j\}})}. 
\end{equation}
More specifically, since Theorem~\ref{theorem:billey} implies that  $\sigma_{s_i}(s_i)=\alpha_i$, from Lemma~\ref{lemma: alphai goes to t in all types} we conclude 
\begin{equation} \label{eq:3-3}
p_{s_i}(s_i)=t. 
\end{equation}

We record the following. 
\begin{Lemma} \label{lemcommute}
In~\eqref{eq:3-1}, if $s_i$ and $s_j$ commute, then $c_i^j=0$.
\end{Lemma}

\begin{proof}
Since 
$s_i$ and $s_j$ commute, we have $w_{\{\alpha_i,\alpha_j\}}=s_is_j$. Moreover  
from Theorem~\ref{theorem:billey} we can compute 
\[
\sigma_{s_i}(w_{\{\alpha_i,\alpha_j\}})=\sigma_{s_i}(s_is_j)=\alpha_i. 
\]
From Lemma~\ref{lemma: alphai goes to t in all types} we get $p_{s_i}(w_{\{\alpha_i,\alpha_j\}})=t$.  Then the equations~\eqref{eq:3-2}, \eqref{eq:3-3} imply $c_i^j=0$ as desired. 
\end{proof}

In the case when $s_i$ and $s_j$ do not commute, the Dynkin diagram corresponding to the subset 
$K=\{\alpha_i,\alpha_j\}$ is connected, so Giambelli's formula (Theorem~\ref{theorem_Giambelli in all types}) yields
\begin{equation} \label{eq:3-4}
p_{v_{\{\alpha_i,\alpha_j\}}}=\frac{1}{2}p_{s_i}p_{s_j}.
\end{equation}

In this case, the coefficient appearing in~\eqref{eq:3-1} can be expressed in terms of the Cartan integer. 

\begin{Lemma}\label{lemcartan}
In~\eqref{eq:3-1}, if $s_i$ and $s_j$ do not commute, then  
\begin{equation*}
c_i^j=-\langle\alpha_i,\alpha_j\rangle.
\end{equation*}
\end{Lemma}

\begin{proof}
From \eqref{eq:3-2}, \eqref{eq:3-3}, \eqref{eq:3-4} we can compute 
\begin{equation} \label{eq:eq:3-5}
c_i^j=\frac {2(p_{s_i}(w_{\{\alpha_i,\alpha_j\}})-t)}{p_{s_j}(w_{\{\alpha_i,\alpha_j\}})}
\end{equation}
so it suffices to compute 
$p_{s_i}(w_{\{\alpha_i,\alpha_j\}})$ and 
$p_{s_j}(w_{\{\alpha_i,\alpha_j\}})$. 
In what follows we use the notation 
\[
a_{ij}:=\langle \alpha_i,\alpha_j\rangle \ \ (i\not=j), \quad a:=a_{ij}a_{ji}. 
\]
With this notation in place, note that by definition of the Cartan integers we have that the action of the simple reflections $s_j$ on the simple roots $\alpha_j$ may be expressed as 
\begin{equation} \label{eq:3-6}
s_j(\alpha_i)=\begin{cases} \alpha_i-a_{ij}\alpha_j \quad&(i\not=j),\\
                              -\alpha_i \quad&(i=j).
                              \end{cases}
\end{equation}

In order to prove the lemma, we consider each of the possible cases. 
                            
(i) In the case when the Dynkin diagram corresponding to $\{\alpha_i, \alpha_j\}$ is of the form 
\begin{picture}(35,10)                   
                       \put(5,5){\circle{5}}    
                       \put(7.5,5){\line(1,0){20}}
                       \put(30,5){\circle{5}}
                       
                       \put(3,-7){$i$} 
                       \put(27,-7){$j$}                                                                     
\end{picture}
the order of $s_is_j$ is $3$ so we have 
\[
w_{\{\alpha_i,\alpha_j\}}=s_is_js_i=s_js_is_j.
\]
Using 
Theorem~\ref{theorem:billey} and~\eqref{eq:3-6} in this case we can compute that
\begin{align*}
&\sigma_{s_i}(w_{\{\alpha_i,\alpha_j\}})=\sigma_{s_i}(s_is_js_i)=\alpha_i+s_is_j(\alpha_i)=a\alpha_i-a_{ij}\alpha_j,\\
&\sigma_{s_j}(w_{\{\alpha_i,\alpha_j\}})=\sigma_{s_j}(s_js_is_j)=\alpha_j+s_js_i(\alpha_j)=a\alpha_j-a_{ji}\alpha_i. 
\end{align*}
Then Lemma~\ref{lemma: alphai goes to t in all types} implies 
\begin{equation*}
p_{s_i}(w_{\{\alpha_i,\alpha_j\}})=(a-a_{ij})t,\quad p_{s_j}(w_{\{\alpha_i,\alpha_j\}})=(a-a_{ji})t.
\end{equation*}
Finally~\eqref{eq:eq:3-5} yields 
\begin{equation} \label{eq:3-7}
c_i^j= \frac{2(a-a_{ij}-1)}{(a-a_{ji})}
\end{equation}
and substituting $a=a_{ij}a_{ji}$, $a_{ij}=-1$ we obtain $c_i^j=-a_{ij}$ as desired.  

(ii) In the case 
\begin{picture}(35,10)                   
                       \put(5,5){\circle{5}}    
                       \put(7.5,6){\line(1,0){20}}
                       \put(7.5,4){\line(1,0){20}}
                       \put(30,5){\circle{5}}
                       
                       \put(3,-7){$i$} 
                       \put(27,-7){$j$}                                                                    
\end{picture}
the order of $s_is_j$ is $4$ so we have 
\[
w_{\{\alpha_i,\alpha_j\}}=s_is_js_is_j=s_js_is_js_i.
\]
Using the above together with 
Theorem~\ref{theorem:billey} and~\eqref{eq:3-6} we may compute
\begin{align*}
&\sigma_{s_i}(w_{\{\alpha_i,\alpha_j\}})=\sigma_{s_i}(s_is_js_is_j)=\alpha_i+s_is_j(\alpha_i)=a\alpha_i-a_{ij}\alpha_j,\\
&\sigma_{s_j}(w_{\{\alpha_i,\alpha_j\}})=\sigma_{s_j}(s_js_is_js_i)=\alpha_j+s_js_i(\alpha_j)=a\alpha_j-a_{ji}\alpha_i  
\end{align*}
which is the same as case (i) above. Thus~\eqref{eq:3-7}
also holds in this case and since $a=a_{ij}a_{ji}=2$ we obtain $c_i^j=-a_{ij}$ as required. 

(iii) Finally, in the case 
\begin{picture}(35,10)                   
                       \put(5,5){\circle{5}}  
                       \put(7,6.5){\line(1,0){20}}
                       \put(7.5,5){\line(1,0){20}}
                       \put(7,3.5){\line(1,0){20}}
                       \put(30,5){\circle{5}} 
                                                                   
                       \put(3,-7){$i$} 
                       \put(27,-7){$j$}
                                                                    
\end{picture}
the element $s_is_j$ has order $6$ and thus 
\[
w_{\{\alpha_i,\alpha_j\}}=s_is_js_is_js_is_j=s_js_is_js_is_js_i.
\]
In this case we have $a=3$ so Theorem~\ref{theorem:billey} and~\eqref{eq:3-6} yield that 
\begin{align*}
\sigma_{s_i}(w_{\{\alpha_i,\alpha_j\}})&=\sigma_{s_i}(s_is_js_is_js_is_j)
=\alpha_i+s_is_j(\alpha_i)+(s_is_j)^2(\alpha_i)=4\alpha_i-2a_{ij}\alpha_j,\\
\sigma_{s_j}(w_{\{\alpha_i,\alpha_j\}})&=\sigma_{s_j}(s_js_is_js_is_js_i)
=\alpha_j+s_js_i(\alpha_j)+(s_js_i)^2(\alpha_j)=4\alpha_j-2a_{ji}\alpha_i.
\end{align*}
Then from Lemma~\ref{lemma: alphai goes to t in all types} we compute 
\begin{equation*}
p_{s_i}(w_{\{\alpha_i,\alpha_j\}})=(4-2a_{ij})t,\quad 
p_{s_j}(w_{\{\alpha_i,\alpha_j\}})=(4-2a_{ji})t.
\end{equation*}
Equation~\eqref{eq:eq:3-5} then implies 
\[
c_i^j=\frac{2(3-2a_{ij})}{4-2a_{ji}}
\]
and finally using that $a_{ij}a_{ji}=3$ we get that $c_i^j=-a_{ij}$ as desired. 
\end{proof}

The main theorem of \cite{HarHorMas} is the following. 

\begin{Theorem}\label{Propqua} (\cite[Theorem~4.1]{HarHorMas})
Following the notation of this section, we have 
\begin{equation*} 
H^*_{\S1}(Pet) \cong \Q[\varpi_1,\cdots, \varpi_n, t] \bigg/ 
\left\langle \displaystyle\sum_{j=1}^{n}\langle\alpha_i,\alpha_j\rangle \varpi_i \varpi_j -2t \varpi_i,  \ \ 1\leq i\leq n \right\rangle 
\end{equation*}
and 
\begin{equation*} 
H^*(Pet) \cong \Q[\varpi_1, \cdots, \varpi_n] \bigg/
\left\langle \displaystyle\sum_{j=1}^{n}\langle\alpha_i,\alpha_j\rangle \varpi_i \varpi_j ,  \ \ 1\leq i\leq n \right\rangle 
\end{equation*}
where the identifications are given by $\varpi_i \mapsto p_{s_i}$ and $\varpi_i \mapsto \check{p}_{s_i}$ respectively. 
\end{Theorem}

\begin{proof}
We give a sketch only. 
If $s_i$ and $s_j$ commute then $\langle \alpha_i,\alpha_j\rangle=0$, so by 
Lemma~\ref{lemcommute}, the conclusion of 
Lemma~\ref{lemcartan} holds in this case. From this and \eqref{eq:3-4} we see that 
\eqref{eq:3-1} can be expressed as 
\begin{equation*}
p_{s_i}^2=t\cdot p_{s_i}-\frac{1}{2}\sum_{j\neq i} \langle\alpha_i,\alpha_j\rangle p_{s_i}p_{s_j}. 
\end{equation*}
Since $\langle \alpha_i,\alpha_i\rangle=2$
for any $i$, the above equation can be re-written to agree with one of the relations given among the $\varpi_i$.

From Theorems~\ref{thm:basis} and \ref{theorem_Giambelli in all types} it follows that $p_{s_1},\ldots p_{s_n},t$ generate $H^*_{\S1}(Pet)$ as $\Q$-algebras. 
Thus, the maps sending $\varpi_i$ to $p_{s_i}$ and $\check{p}_{s_i}$ respectively give rise to surjective ring maps $\varphi$ and $\check{\varphi}$ respectively. The remainder of the argument uses Hilbert series techniques similar to previous proofs given in this manuscript. 
\end{proof}

\section{ The cohomology rings of regular nilpotent Hessenberg varieties in type A  } 
\label{section: reg nilp Hess}

In the previous sections, we gave an explicit ring presentation of the equivariant and ordinary cohomology rings of Peterson varieties, which are special cases of regular nilpotent Hessenberg varieties. 
We now turn  to a similar such derivation -- at least, for the case of Lie type A -- for the general case of \emph{any} regular nilpotent Hessenberg variety.  The restriction to Lie type A is partly for simplicity, but is also due to the fact that the situation for other Lie types turns out to be best discussed using substantially different perspectives. Our exposition for this section will be based on \cite{AHHM2019}.  As already mentioned, the Peterson case ca be seen as a ``motivating'' case, and we hope that the reader will find that some of the features of the analysis given below are familiar from the previous sections.  Indeed, in broad strokes the structure of our argument remains the same: namely, after coming up with a set of generators for the relevant rings (which come from the cohomology ring of the flag variety), we first lean on the theory of localization in equivariant cohomology to aid us in showing that certain relations hold among these generators, thus giving a well-defined homomorphism 
\[
\varphi: \Q[x_1,\ldots, x_n]/I_h \to H^*_S(\Hess(\mathsf{N},h))
\]
where $x_1,\ldots,x_n$ are the generators and the ideal $I_h$ is generated by the relations we have found. Then, in order to show that the given relations suffice to describe the ring, i.e., to show that $\varphi$ is an isomorphism, we turn to the theory of Hilbert series and show that the domain and the codomain have identical Hilbert series. 

Based on the above brief sketch, the reader could be forgiven for wondering if there is anything at all that is different about the general case as compared to the Peterson case.  Our reply to this question is a typical one in mathematics: namely, the general case turns out to present technical difficulties not seen in the special case, and hence requires some genuinely new insights and arguments. 

With this in mind, our aim in the exposition below is to make this general case comprehensible by making as clear as possible both the similarities and the disparities between the Peterson case and the general case. It is explicitly \emph{not} our aim to reproduce, here, the most technical (and lengthy) portions of the actual proofs of some of the statements, since these may be readily found in the original manuscript \cite{AHHM2019}. Instead, we have sought to provide the reader with motivation, examples, and overall perspective, armed with which an interested reader would be able to absorb the proofs in \cite{AHHM2019} without undue difficulty. 

In particular, and from our point of view most importantly, we give a leisurely exposition concerning the derivation of our set of relations, denoted below as $f_{i,j}$ (respectively $\cf_{i,j}$) for the equivariant (respectively ordinary) cohomology ring of $\Hess(\mathsf{N},h)$. This is because our $f_{i,j}$ (or $\cf_{i,j}$) are not completely straightforward generalizations of the quadratic relations found in the Peterson case in Theorem~\ref{theo:3.1}. In this sense, the derivation of these relations is one of the conceptually important contributions of \cite{AHHM2019}.

\subsection{The motivation and definition of the relations $f_{i,j}$ and statement of main results}\label{subsec: derivation fij}

As remarked at the end of the previous section, one of the important contributions of \cite{AHHM2019} is the 
derivation of the appropriate polynomials $f_{i,j}$ which generate the ideal of relations for $H^*_\mathsf{S}(\Hess(\mathsf{N},h))$. 
This section is devoted to the task of finding these polynomials. 
Let $\SChNil_i\in H_S^2(\Hess(\mathsf{N},h))$ denote the image of $\TChFlag_i$ in \eqref{def of T eq ch in flag} under the restriction homomorphism $H^*_T(\fln) \to H^*_S(\Hess(\mathsf{N},h))$ which is the composition of the two left vertical maps in~\eqref{eq:cd}. 
We note that this point in the exposition, we have not justified that the classes $\SChNil_i$ generate $H^*_S(\Hess(\mathsf{N},h))$. Our approach to the proof does not attempt to show this directly; rather, we find the relations satisfied by the $\SChNil_i$, and then a Hilbert series argument will in fact show, as a corollary, that the classes $\SChNil_i$ do indeed generate $H^*_S(\Hess(\mathsf{N},h))$ (as a $H^*(B\S1)$-algebra). 
More specifically, our first task below is to find polynomials $f(x_1,\ldots,x_n, t) \in \mathbb{Q}[x_1,\cdots, x_n, t]$ such that $f(\SChNil_1, \cdots, \SChNil_n, t) = 0 \in H^*_\mathsf{S}(\Hess(\mathsf{N},h))$.

We will find such relations among the $\SChNil_i$ inductively.  The base cases (``the building blocks'') are the polynomials defined as follows: 
\begin{equation} \label{eq:f-1}
g_j:=\sum_{k=1}^j(x_k-kt) \in \Q[x_1, \ldots, x_n, t] \quad \textup{ for } j
\in [n]. 
\end{equation} 
For convenience we also set $g_0 :=0$.   With this in hand, we can now define our polynomials $f_{i,j}$ in a recursive fashion as follows. 

\begin{Definition} \label{definition:fij}
Let $(i,j)$ be a pair of integers satisfying  $1 \leq j \leq i \leq n$. 
We define polynomials $f_{i,j}$ inductively as follows. As the base
case, when $i=j$, we
define 
\begin{equation*}
f_{j,j}:=g_j \quad \textup{ for } j \in [n]. 
\end{equation*}
Proceeding inductively, for $(i,j)$ with $1 \leq j < i \leq n$ we define 
\begin{equation} \label{eq:f-3}
f_{i,j}:=f_{i-1,j-1}+\big(x_j-x_i-t\big)f_{i-1,j}
\end{equation}
where we take the convention that $f_{*,0}:=0$ for any $*$. 
\end{Definition}

\begin{Example}\label{exam:f}
Suppose $n=4$. Then the $f_{i,j}$ have the following form. 
\begin{align*}
&f_{1,1} = g_1, f_{2,2}=g_2, f_{3,3} = g_3, f_{4,4} = g_4 \\
&f_{2,1}=(x_1-x_2-t)g_1 \\
&f_{3,2}=(x_1-x_2-t)g_1+(x_2-x_3-t)g_2  \\
&f_{4,3}=(x_1-x_2-t)g_1+(x_2-x_3-t)g_2+(x_3-x_4-t)g_3 \\
&f_{3,1}=(x_1-x_3-t)(x_1-x_2-t)g_1  \\
&f_{4,2}=(x_1-x_3-t)(x_1-x_2-t)g_1+(x_2-x_4-t)\{(x_1-x_2-t)g_1+(x_2-x_3-t)g_2\} \\
&f_{4,1}=(x_1-x_4-t)(x_1-x_3-t)(x_1-x_2-t)g_1
\end{align*}
\end{Example}

It helps to informally visualize each $f_{i,j}$ for $i \geq  j$ as being associated to the $(i,j)$-th entry in an $n \times n$ matrix, as follows: 
\begin{equation}\label{eq:matrix fij}
\begin{pmatrix}
f_{1,1} & 0 & \cdots & \cdots & 0 \\
f_{2,1} & f_{2,2} & 0 & \cdots & 0 \\
f_{3,1} & f_{3,2} & f_{3,3} & \ddots & \vdots \\
\vdots &  \vdots      &     \vdots            & \ddots & 0 \\
f_{n,1} & f_{n,2} & f_{n,3} & \cdots & f_{n,n} 
\end{pmatrix}.
\end{equation}
From Definition~\ref{definition:fij} we see that $f_{i,j}$ is determined by the entries ``right above'' it and 
``to its upper-left'' (and when there is no entry to its upper-left, we treat it as $0$).

Our next goal is to motivate the above definition of the $f_{i,j}$. In particular, as we already stated above, 
it is not straightforward to shift from the relations appearing in the presentation of the Petersons in Section~\ref{section: peterson in type A} to the relations $f_{i,j}$.  Some non-trivial new ideas are needed.  By illustration through concrete examples, recorded below, we aim to give the reader a feeling for how we derived the $f_{i,j}$ in Definition~\ref{definition:fij}, where we started our explorations by manipulating and expanding on the relations for the Peterson case. 
To describe this process, let us first recall that in the analysis of the Peterson case in Section~\ref{section: peterson in type A} we had defined the $p_{v_{\mathcal{A}}}$ in~\eqref{eq: definition pvA}, and using Lemma~\ref{lemma:equivariant Schubert class simple reflection} we can conclude 
\begin{equation}\label{eq.pj.relations.and.generator}
p_i := \pi_1(\sigma_{s_i}^T) =\pi_1 \left( \sum_{k=1}^i ( \tau_k^T - t_k ) \right) = \sum_{k=1}^i ( \SChNil_k - kt ), 
\end{equation} 
where $\pi_1$ is the restriction map $H^*_T(\fln) \to H^*_S(\Hess(\mathsf{N},h))$.
Under the map taking the variable $x_k$ to $\bar{\tau}^{\S1}_k$, we see from~\eqref{eq:f-1} that $p_i$ is represented by the polynomial $g_i$.  
Indeed, this motivates our definition of the $g_i$. 
Noting that $\sum_{k=1}^n ( \TChFlag_k - t_k )=0$ in $H^*_T(\fln)$ by Theorem~\ref{theorem:presentation flag typeA}, we always have the relation
\begin{equation}\label{eq.gn.is.zero}
p_n=0 \ \ \ {\rm in} \ H^*_\mathsf{S}(\Hess(\mathsf{N},h))
\end{equation}
for any Hessenberg function $h$, which translates to the relation $g_n=0$.

We now translate between the presentation of $H^*_{\S1}(\Y)$ in Theorem~\ref{theo:3.1}, which is stated in terms of variables $z_i$ which get sent to $p_i, i \in [n-1]$, and a presentation in terms of variables $x_j$ which correspond to the $\overline{\tau}^{\S1}_j, j \in [n]$. Keeping in mind the relation~\eqref{eq.gn.is.zero} that $g_n=0$ and using~\eqref{eq.pj.relations.and.generator} to see that $z_i$ corresponds to $g_i=\sum_{k=1}^i (x_k-kt)$, it is not hard to see that 
\begin{align*}
 H^*_{\S1}(\Y) & \cong \Q[z_1, \cdots, z_{n-1}, t] \bigg/  \langle z_k(z_k - \frac{1}{2} z_{k-1} - \frac{1}{2} z_{k+1} - t)  \, \mid \, k \in [n-1] \rangle \\ 
& \cong \Q[z_1, \cdots, z_{n-1}, t] \bigg/ \langle z_k(2z_k - z_{k-1}-z_{k+1}-2t) \, \mid \, k \in [n-1] \rangle \\
& \cong \Q[x_1, \cdots, x_n, t] \bigg/  \langle g_k(x_k - x_{k+1}-t) \, \mid \, k \in [n-1] \rangle + \langle g_n \rangle. 
\end{align*}
Thus we see that the fundamental relations for $\Y$ in terms of the $x_j$ variables have a simple relationship to the $f_{i,j}$ of Definition~\ref{definition:fij}. This is easiest to see in an explicit example. 

\begin{Example}\label{example.n.is.4.peterson}
Let $n=4, h=(2,3,4,4)$ for $Pet_4$. The relations are then 
\begin{equation}\label{eq:example.g.ideal}
\langle g_1(x_1-x_2-t), g_2(x_2-x_3-t), g_3(x_3-x_4-t), g_4 \rangle 
\end{equation}
Notice from Example~\ref{exam:f} that 
$$
f_{2,1} = g_1(x_1-x_2-t), f_{3,2} = f_{2,1} + g_2(x_2-x_3-t), f_{4,3} = f_{3,2} + g_3(x_3-x_4-t), f_{4,4} = g_4.
$$
So it is evident from the above that the two ideals 
$$ 
\langle f_{2,1}, f_{3,2}, f_{4,3}, f_{4,4} \rangle  \, \textup{ and } \, \langle g_1(x_1-x_2-t), g_2(x_2-x_3-t), g_3(x_3-x_4-t), g_4 \rangle
$$
are equal.  Thus, the ideal in~\eqref{eq:example.g.ideal} can also be described as $\langle f_{2,1}, f_{3,2}, f_{4,3}, f_{4,4}\rangle$, or in other words, we take the ideal generated by the $f_{i,j}$ polynomials corresponding to the lowest box in the Hessenberg diagram of $h$ for each column of the diagram. Our methods below attempt to generalize this beyond the Peterson case. 
\end{Example}

It is now our task to appropriately generalize the relations in Theorem~\ref{theo:3.1} to the case where the Hessenberg function is larger than the Peterson Hessenberg function. 
To accomplish this, it is useful to visualize a Hessenberg function as a ``collection of boxes'' in a Hessenberg diagram, which consists of an $n \times n$ grid of square boxes, and the Hessenberg diagram consists of the collection of $(i,j)$-th boxes for pairs $(i,j)$ with $i \leq h(j)$. (For an illustration in the case $n=4$, see Example~\ref{ex:h=(2,3,4,4)} below.)  In this context, the Peterson Hessenberg function $h(i)=i+1$ can be viewed as the minimal such collection of boxes satisfying $h(i) \geq i+1$, which means that all boxes immediately below the main diagonal are included in the Hessenberg diagram. More general indecomposable Hessenberg functions are obtained by simply ``adding boxes'' to the Peterson Hessenberg diagram.  Moreover, as we have just seen above, we can view the quadratic relations defining $H^*_S(\Y)$ correspond to the polynomials contained in the $(j+1,j)$-th boxes for $j \in [n-1]$, as well as the $g_n=0$ coming from the $(n,n)$-th box. 
Thus, another way of phrasing our task, below, is to ask: what happens to the quadratic Peterson relations as we ``add boxes''? It is natural to think of this process as inductive, i.e., we may add one box at a time, and consider the consequences of adding that single box. We illustrate this thought process through the concrete examples that follow.

In Section~\ref{section: peterson in type A}, when we analyzed the relations satisfied by the classes $p_i$ in $H^*_\mathsf{S}(\Y)$, we used the injectivity into the 
equivariant cohomology of the fixed points and concrete data about the restrictions of the classes to $w \in \Y^\mathsf{S}$. We continue 
using this strategy as we seek to understand the new relations that arise when we ``add boxes''. In particular, we will need -- in analogy to Lemma~\ref{lemma:fixed_point_Peterson} -- a concrete description of the $\mathsf{S}$-fixed points in $\Hess(\mathsf{N},h)$. We quote the following \cite[Lemma~2.3]{AHHM2019}. 

\begin{Lemma} \label{lemma:Hess fixed points}
The $\mathsf{S}$-fixed point set $\Hess(\mathsf{N},h)^\mathsf{S} \subset S_n$ of
$\Hess(\mathsf{N},h)$ is given by 
\begin{equation*} \label{eq:H-1}
 \Hess(\mathsf{N},h)^\mathsf{S}=\{w\in S_n \mid w^{-1}(w(i)-1)\le h(i)\text{ for all } i \in [n] \}.
\end{equation*}
Here, we use a convention 
\begin{equation*} 
v(0)=0 \quad \text{for all $v \in S_n$}.
\end{equation*}
\end{Lemma}

To begin our exploration, we first record a computation, for a small value of $n$, for the base case of the Peterson Hessenberg function. 

\begin{Example} \label{ex:h=(2,3,4,4)}
Let $n=4$ and $h=(2,3,4,4)$. This is precisely the Peterson case, with corresponding ``box'' diagram 
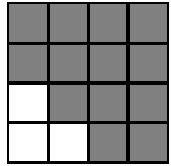
\begin{center}
\begin{picture}(60,60)
\put(0,48){\colorbox{gray}}
\put(0,52){\colorbox{gray}}
\put(0,57){\colorbox{gray}}
\put(4,48){\colorbox{gray}}
\put(4,52){\colorbox{gray}}
\put(4,57){\colorbox{gray}}
\put(9,48){\colorbox{gray}}
\put(9,52){\colorbox{gray}}
\put(9,57){\colorbox{gray}}

\put(15,48){\colorbox{gray}}
\put(15,52){\colorbox{gray}}
\put(15,57){\colorbox{gray}}
\put(19,48){\colorbox{gray}}
\put(19,52){\colorbox{gray}}
\put(19,57){\colorbox{gray}}
\put(24,48){\colorbox{gray}}
\put(24,52){\colorbox{gray}}
\put(24,57){\colorbox{gray}}

\put(30,48){\colorbox{gray}}
\put(30,52){\colorbox{gray}}
\put(30,57){\colorbox{gray}}
\put(34,48){\colorbox{gray}}
\put(34,52){\colorbox{gray}}
\put(34,57){\colorbox{gray}}
\put(39,48){\colorbox{gray}}
\put(39,52){\colorbox{gray}}
\put(39,57){\colorbox{gray}}

\put(45,48){\colorbox{gray}}
\put(45,52){\colorbox{gray}}
\put(45,57){\colorbox{gray}}
\put(49,48){\colorbox{gray}}
\put(49,52){\colorbox{gray}}
\put(49,57){\colorbox{gray}}
\put(54,48){\colorbox{gray}}
\put(54,52){\colorbox{gray}}
\put(54,57){\colorbox{gray}}

\put(0,33){\colorbox{gray}}
\put(0,37){\colorbox{gray}}
\put(0,42){\colorbox{gray}}
\put(4,33){\colorbox{gray}}
\put(4,37){\colorbox{gray}}
\put(4,42){\colorbox{gray}}
\put(9,33){\colorbox{gray}}
\put(9,37){\colorbox{gray}}
\put(9,42){\colorbox{gray}}

\put(15,33){\colorbox{gray}}
\put(15,37){\colorbox{gray}}
\put(15,42){\colorbox{gray}}
\put(19,33){\colorbox{gray}}
\put(19,37){\colorbox{gray}}
\put(19,42){\colorbox{gray}}
\put(24,33){\colorbox{gray}}
\put(24,37){\colorbox{gray}}
\put(24,42){\colorbox{gray}}

\put(30,33){\colorbox{gray}}
\put(30,37){\colorbox{gray}}
\put(30,42){\colorbox{gray}}
\put(34,33){\colorbox{gray}}
\put(34,37){\colorbox{gray}}
\put(34,42){\colorbox{gray}}
\put(39,33){\colorbox{gray}}
\put(39,37){\colorbox{gray}}
\put(39,42){\colorbox{gray}}

\put(45,33){\colorbox{gray}}
\put(45,37){\colorbox{gray}}
\put(45,42){\colorbox{gray}}
\put(49,33){\colorbox{gray}}
\put(49,37){\colorbox{gray}}
\put(49,42){\colorbox{gray}}
\put(54,33){\colorbox{gray}}
\put(54,37){\colorbox{gray}}
\put(54,42){\colorbox{gray}}


\put(15,18){\colorbox{gray}}
\put(15,22){\colorbox{gray}}
\put(15,27){\colorbox{gray}}
\put(19,18){\colorbox{gray}}
\put(19,22){\colorbox{gray}}
\put(19,27){\colorbox{gray}}
\put(24,18){\colorbox{gray}}
\put(24,22){\colorbox{gray}}
\put(24,27){\colorbox{gray}}

\put(30,18){\colorbox{gray}}
\put(30,22){\colorbox{gray}}
\put(30,27){\colorbox{gray}}
\put(34,18){\colorbox{gray}}
\put(34,22){\colorbox{gray}}
\put(34,27){\colorbox{gray}}
\put(39,18){\colorbox{gray}}
\put(39,22){\colorbox{gray}}
\put(39,27){\colorbox{gray}}

\put(45,18){\colorbox{gray}}
\put(45,22){\colorbox{gray}}
\put(45,27){\colorbox{gray}}
\put(49,18){\colorbox{gray}}
\put(49,22){\colorbox{gray}}
\put(49,27){\colorbox{gray}}
\put(54,18){\colorbox{gray}}
\put(54,22){\colorbox{gray}}
\put(54,27){\colorbox{gray}}

%

\put(30,3){\colorbox{gray}}
\put(30,7){\colorbox{gray}}
\put(30,12){\colorbox{gray}}
\put(34,3){\colorbox{gray}}
\put(34,7){\colorbox{gray}}
\put(34,12){\colorbox{gray}}
\put(39,3){\colorbox{gray}}
\put(39,7){\colorbox{gray}}
\put(39,12){\colorbox{gray}}

\put(45,3){\colorbox{gray}}
\put(45,7){\colorbox{gray}}
\put(45,12){\colorbox{gray}}
\put(49,3){\colorbox{gray}}
\put(49,7){\colorbox{gray}}
\put(49,12){\colorbox{gray}}
\put(54,3){\colorbox{gray}}
\put(54,7){\colorbox{gray}}
\put(54,12){\colorbox{gray}}

\put(0,0){\framebox(15,15)}
\put(15,0){\framebox(15,15)}
\put(30,0){\framebox(15,15)}
\put(45,0){\framebox(15,15)}
\put(0,15){\framebox(15,15)}
\put(15,15){\framebox(15,15)}
\put(30,15){\framebox(15,15)}
\put(45,15){\framebox(15,15)}
\put(0,30){\framebox(15,15)}
\put(15,30){\framebox(15,15)}
\put(30,30){\framebox(15,15)}
\put(45,30){\framebox(15,15)}
\put(0,45){\framebox(15,15)}
\put(15,45){\framebox(15,15)}
\put(30,45){\framebox(15,15)}
\put(45,45){\framebox(15,15)}
\end{picture}
\end{center}
and by Lemma~\ref{lemma:fixed_point_Peterson}, we may compute that its set of $\mathsf{S}$-fixed points is 
\begin{align} \label{eq:fixed_point_h=(2,3,4,4)}
1234,1243,1324,1432,2134,2143,3214,4321. 
\end{align}
In this case we have seen in Example~\ref{example.n.is.4.peterson} that the elements 
\begin{align*}
&\Delta_1:=\xi_1(\SChNil_1-\SChNil_2-t) \\
&\Delta_2:=\xi_2(\SChNil_2-\SChNil_3-t) \\
&\Delta_3:=\xi_3(\SChNil_3-\SChNil_4-t) \\
&\Delta_4:=\xi_4 
\end{align*}
vanish in $H^*_\mathsf{S}(\Y)$, and moreover, the corresponding polynomials generate the ideal of relations among the $\SChNil_i$ and $t$. Note that it is also straightforward to verify directly that $\Delta_i \vert_w=0$ for all the $\mathsf{S}$-fixed points $w\in \Y^\mathsf{S}$ in the list above. 
\end{Example}

The next example adds one box -- the box in the $(3,1)$-th position of the matrix -- to the previous example. 

\begin{Example} \label{ex:h=(3,3,4,4)} 
Consider again $n=4$ and this time, let us consider $h=(3,3,4,4)$. The corresponding configuration of boxes is below. 
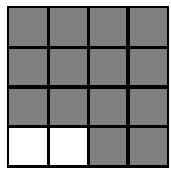
\begin{center}
\begin{picture}(60,60)
\put(0,48){\colorbox{gray}}
\put(0,52){\colorbox{gray}}
\put(0,57){\colorbox{gray}}
\put(4,48){\colorbox{gray}}
\put(4,52){\colorbox{gray}}
\put(4,57){\colorbox{gray}}
\put(9,48){\colorbox{gray}}
\put(9,52){\colorbox{gray}}
\put(9,57){\colorbox{gray}}

\put(15,48){\colorbox{gray}}
\put(15,52){\colorbox{gray}}
\put(15,57){\colorbox{gray}}
\put(19,48){\colorbox{gray}}
\put(19,52){\colorbox{gray}}
\put(19,57){\colorbox{gray}}
\put(24,48){\colorbox{gray}}
\put(24,52){\colorbox{gray}}
\put(24,57){\colorbox{gray}}

\put(30,48){\colorbox{gray}}
\put(30,52){\colorbox{gray}}
\put(30,57){\colorbox{gray}}
\put(34,48){\colorbox{gray}}
\put(34,52){\colorbox{gray}}
\put(34,57){\colorbox{gray}}
\put(39,48){\colorbox{gray}}
\put(39,52){\colorbox{gray}}
\put(39,57){\colorbox{gray}}

\put(45,48){\colorbox{gray}}
\put(45,52){\colorbox{gray}}
\put(45,57){\colorbox{gray}}
\put(49,48){\colorbox{gray}}
\put(49,52){\colorbox{gray}}
\put(49,57){\colorbox{gray}}
\put(54,48){\colorbox{gray}}
\put(54,52){\colorbox{gray}}
\put(54,57){\colorbox{gray}}

\put(0,33){\colorbox{gray}}
\put(0,37){\colorbox{gray}}
\put(0,42){\colorbox{gray}}
\put(4,33){\colorbox{gray}}
\put(4,37){\colorbox{gray}}
\put(4,42){\colorbox{gray}}
\put(9,33){\colorbox{gray}}
\put(9,37){\colorbox{gray}}
\put(9,42){\colorbox{gray}}

\put(15,33){\colorbox{gray}}
\put(15,37){\colorbox{gray}}
\put(15,42){\colorbox{gray}}
\put(19,33){\colorbox{gray}}
\put(19,37){\colorbox{gray}}
\put(19,42){\colorbox{gray}}
\put(24,33){\colorbox{gray}}
\put(24,37){\colorbox{gray}}
\put(24,42){\colorbox{gray}}

\put(30,33){\colorbox{gray}}
\put(30,37){\colorbox{gray}}
\put(30,42){\colorbox{gray}}
\put(34,33){\colorbox{gray}}
\put(34,37){\colorbox{gray}}
\put(34,42){\colorbox{gray}}
\put(39,33){\colorbox{gray}}
\put(39,37){\colorbox{gray}}
\put(39,42){\colorbox{gray}}

\put(45,33){\colorbox{gray}}
\put(45,37){\colorbox{gray}}
\put(45,42){\colorbox{gray}}
\put(49,33){\colorbox{gray}}
\put(49,37){\colorbox{gray}}
\put(49,42){\colorbox{gray}}
\put(54,33){\colorbox{gray}}
\put(54,37){\colorbox{gray}}
\put(54,42){\colorbox{gray}}

\put(0,18){\colorbox{gray}}
\put(0,22){\colorbox{gray}}
\put(0,27){\colorbox{gray}}
\put(4,18){\colorbox{gray}}
\put(4,22){\colorbox{gray}}
\put(4,27){\colorbox{gray}}
\put(9,18){\colorbox{gray}}
\put(9,22){\colorbox{gray}}
\put(9,27){\colorbox{gray}}

\put(15,18){\colorbox{gray}}
\put(15,22){\colorbox{gray}}
\put(15,27){\colorbox{gray}}
\put(19,18){\colorbox{gray}}
\put(19,22){\colorbox{gray}}
\put(19,27){\colorbox{gray}}
\put(24,18){\colorbox{gray}}
\put(24,22){\colorbox{gray}}
\put(24,27){\colorbox{gray}}

\put(30,18){\colorbox{gray}}
\put(30,22){\colorbox{gray}}
\put(30,27){\colorbox{gray}}
\put(34,18){\colorbox{gray}}
\put(34,22){\colorbox{gray}}
\put(34,27){\colorbox{gray}}
\put(39,18){\colorbox{gray}}
\put(39,22){\colorbox{gray}}
\put(39,27){\colorbox{gray}}

\put(45,18){\colorbox{gray}}
\put(45,22){\colorbox{gray}}
\put(45,27){\colorbox{gray}}
\put(49,18){\colorbox{gray}}
\put(49,22){\colorbox{gray}}
\put(49,27){\colorbox{gray}}
\put(54,18){\colorbox{gray}}
\put(54,22){\colorbox{gray}}
\put(54,27){\colorbox{gray}}

%

\put(30,3){\colorbox{gray}}
\put(30,7){\colorbox{gray}}
\put(30,12){\colorbox{gray}}
\put(34,3){\colorbox{gray}}
\put(34,7){\colorbox{gray}}
\put(34,12){\colorbox{gray}}
\put(39,3){\colorbox{gray}}
\put(39,7){\colorbox{gray}}
\put(39,12){\colorbox{gray}}

\put(45,3){\colorbox{gray}}
\put(45,7){\colorbox{gray}}
\put(45,12){\colorbox{gray}}
\put(49,3){\colorbox{gray}}
\put(49,7){\colorbox{gray}}
\put(49,12){\colorbox{gray}}
\put(54,3){\colorbox{gray}}
\put(54,7){\colorbox{gray}}
\put(54,12){\colorbox{gray}}

\put(0,0){\framebox(15,15)}
\put(15,0){\framebox(15,15)}
\put(30,0){\framebox(15,15)}
\put(45,0){\framebox(15,15)}
\put(0,15){\framebox(15,15)}
\put(15,15){\framebox(15,15)}
\put(30,15){\framebox(15,15)}
\put(45,15){\framebox(15,15)}
\put(0,30){\framebox(15,15)}
\put(15,30){\framebox(15,15)}
\put(30,30){\framebox(15,15)}
\put(45,30){\framebox(15,15)}
\put(0,45){\framebox(15,15)}
\put(15,45){\framebox(15,15)}
\put(30,45){\framebox(15,15)}
\put(45,45){\framebox(15,15)}
\end{picture}
\end{center}

Now the set of $\mathsf{S}$-fixed points is larger than that of the previous example.  Specifically, by Lemma~\ref{lemma:Hess fixed points}, in addition to the $\mathsf{S}$-fixed points $\Y^\mathsf{S}$ listed in~\eqref{eq:fixed_point_h=(2,3,4,4)}, we also have the following $\mathsf{S}$-fixed points 
\begin{align} \label{eq:fixed_point_h=(3,3,4,4)}
2314,3124,3421,4132. 
\end{align}
Now let us consider what happens to the elements $\Delta_1, \Delta_2, \Delta_3$, which vanished at the $\mathsf{S}$-fixed points $\Y^\mathsf{S}$, when we restrict them to these ``extra'' $4$ fixed points in $\Hess(\mathsf{N},h)^\mathsf{S}$. It is straightforward to compute the restrictions, and listing these in order we have 
\begin{align*}
&(\Delta_1|_{2314},\Delta_1|_{3124},\Delta_1|_{3421},\Delta_1|_{4132})=(-2t^2,2t^2,-4t^2,6t^2) \\
&(\Delta_2|_{2314},\Delta_2|_{3124},\Delta_2|_{3421},\Delta_2|_{4132})=(2t^2,-2t^2,4t^2,-6t^2) \\
&(\Delta_3|_{2314},\Delta_3|_{3124},\Delta_3|_{3421},\Delta_3|_{4132})=(0,0,0,0). 
\end{align*}
The above computation shows that while $\Delta_3$ still vanishes in $H^*_\mathsf{S}(\Hess(\mathsf{N},h))$ for $h=(3,3,4,4)$, the elements $\Delta_1, \Delta_2$, in contrast to the Peterson case, do not restrict to $0$ at all fixed points and therefore do not vanish. On the other hand, a glance at the restrictions of $\Delta_1$ and $\Delta_2$ immediately reveals that the \emph{sum} element $\Delta_1+\Delta_2$ does vanish for $h=(3,3,4,4)$. In summary, we have determined that 
\begin{align*}
&\Delta_1+\Delta_2=\xi_1(\SChNil_1-\SChNil_2-t)+\xi_2(\SChNil_2-\SChNil_3-t)=0 \\ 
&\Delta_3=\xi_3(\SChNil_3-\SChNil_4-t)=0 
\end{align*}
are relations in $H^*_\mathsf{S}(\Hess(\mathsf{N},h))$ for $h=(3,3,4,4)$.

Our next observation comes from an examination of properties of the 4 ``extra'' fixed points~\eqref{eq:fixed_point_h=(3,3,4,4)}. Specifically, observe that for all 4 of the new fixed points, the difference between the first entry in the one-line notation of $w$ and its third is always $1$, i.e., $w(1)-w(3)=1$ for all $w$ in~\eqref{eq:fixed_point_h=(3,3,4,4)}.  From this it immediately follows that 
\begin{align*}
(\SChNil_1-\SChNil_3-t)|_w=0  
\end{align*}
for the four fixed points $w$ in~\eqref{eq:fixed_point_h=(3,3,4,4)}. The element $\Delta_1$, as we have seen, vanishes on $\Y^\mathsf{S}$ but does not vanish on the $4$ fixed points in~\eqref{eq:fixed_point_h=(3,3,4,4)}. The element $(\SChNil_1-\SChNil_3-t)$ vanishes at all the $4$ fixed points in~\eqref{eq:fixed_point_h=(3,3,4,4)}. Putting these two together, then, we may conclude that 
\begin{align*}
\Delta_1(\SChNil_1-\SChNil_3-t)=\xi_1(\SChNil_1-\SChNil_2-t)(\SChNil_1-\SChNil_3-t)=0  \, \textup{ in } \, H^*_{\S1}(\Hess(\mathsf{N},h))
\end{align*}
since the product does vanish at all fixed points $w \in \Hess(\mathsf{N},h)^\S1$. 
Putting the above observations together, we have derived the following list of relations that hold in $H^*_\mathsf{S}(\Hess(\mathsf{N},h))$ for $h=(3,3,4,4)$: 
\begin{align*}
\Delta_1(\SChNil_1-\SChNil_3-t)&=0 \\
\Delta_1+\Delta_2&=0 \\ 
\Delta_3&=0 \\
\Delta_4&=0 
\end{align*}
and, if we were feeling optimistic, we might guess that this forms the complete list of relations.  
\end{Example}

The next example adds a box to the Peterson case of Example~\ref{ex:h=(2,3,4,4)} in a \emph{different} manner than that of Example~\ref{ex:h=(3,3,4,4)}. 

\begin{Example} \label{ex:h=(2,4,4,4)}
Consider again $n=4$ and this time, let $h=(2,4,4,4)$. The corresponding ``box'' diagram is

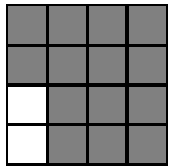
\begin{center}
\begin{picture}(60,60)
\put(0,48){\colorbox{gray}}
\put(0,52){\colorbox{gray}}
\put(0,57){\colorbox{gray}}
\put(4,48){\colorbox{gray}}
\put(4,52){\colorbox{gray}}
\put(4,57){\colorbox{gray}}
\put(9,48){\colorbox{gray}}
\put(9,52){\colorbox{gray}}
\put(9,57){\colorbox{gray}}

\put(15,48){\colorbox{gray}}
\put(15,52){\colorbox{gray}}
\put(15,57){\colorbox{gray}}
\put(19,48){\colorbox{gray}}
\put(19,52){\colorbox{gray}}
\put(19,57){\colorbox{gray}}
\put(24,48){\colorbox{gray}}
\put(24,52){\colorbox{gray}}
\put(24,57){\colorbox{gray}}

\put(30,48){\colorbox{gray}}
\put(30,52){\colorbox{gray}}
\put(30,57){\colorbox{gray}}
\put(34,48){\colorbox{gray}}
\put(34,52){\colorbox{gray}}
\put(34,57){\colorbox{gray}}
\put(39,48){\colorbox{gray}}
\put(39,52){\colorbox{gray}}
\put(39,57){\colorbox{gray}}

\put(45,48){\colorbox{gray}}
\put(45,52){\colorbox{gray}}
\put(45,57){\colorbox{gray}}
\put(49,48){\colorbox{gray}}
\put(49,52){\colorbox{gray}}
\put(49,57){\colorbox{gray}}
\put(54,48){\colorbox{gray}}
\put(54,52){\colorbox{gray}}
\put(54,57){\colorbox{gray}}

\put(0,33){\colorbox{gray}}
\put(0,37){\colorbox{gray}}
\put(0,42){\colorbox{gray}}
\put(4,33){\colorbox{gray}}
\put(4,37){\colorbox{gray}}
\put(4,42){\colorbox{gray}}
\put(9,33){\colorbox{gray}}
\put(9,37){\colorbox{gray}}
\put(9,42){\colorbox{gray}}

\put(15,33){\colorbox{gray}}
\put(15,37){\colorbox{gray}}
\put(15,42){\colorbox{gray}}
\put(19,33){\colorbox{gray}}
\put(19,37){\colorbox{gray}}
\put(19,42){\colorbox{gray}}
\put(24,33){\colorbox{gray}}
\put(24,37){\colorbox{gray}}
\put(24,42){\colorbox{gray}}

\put(30,33){\colorbox{gray}}
\put(30,37){\colorbox{gray}}
\put(30,42){\colorbox{gray}}
\put(34,33){\colorbox{gray}}
\put(34,37){\colorbox{gray}}
\put(34,42){\colorbox{gray}}
\put(39,33){\colorbox{gray}}
\put(39,37){\colorbox{gray}}
\put(39,42){\colorbox{gray}}

\put(45,33){\colorbox{gray}}
\put(45,37){\colorbox{gray}}
\put(45,42){\colorbox{gray}}
\put(49,33){\colorbox{gray}}
\put(49,37){\colorbox{gray}}
\put(49,42){\colorbox{gray}}
\put(54,33){\colorbox{gray}}
\put(54,37){\colorbox{gray}}
\put(54,42){\colorbox{gray}}


\put(15,18){\colorbox{gray}}
\put(15,22){\colorbox{gray}}
\put(15,27){\colorbox{gray}}
\put(19,18){\colorbox{gray}}
\put(19,22){\colorbox{gray}}
\put(19,27){\colorbox{gray}}
\put(24,18){\colorbox{gray}}
\put(24,22){\colorbox{gray}}
\put(24,27){\colorbox{gray}}

\put(30,18){\colorbox{gray}}
\put(30,22){\colorbox{gray}}
\put(30,27){\colorbox{gray}}
\put(34,18){\colorbox{gray}}
\put(34,22){\colorbox{gray}}
\put(34,27){\colorbox{gray}}
\put(39,18){\colorbox{gray}}
\put(39,22){\colorbox{gray}}
\put(39,27){\colorbox{gray}}

\put(45,18){\colorbox{gray}}
\put(45,22){\colorbox{gray}}
\put(45,27){\colorbox{gray}}
\put(49,18){\colorbox{gray}}
\put(49,22){\colorbox{gray}}
\put(49,27){\colorbox{gray}}
\put(54,18){\colorbox{gray}}
\put(54,22){\colorbox{gray}}
\put(54,27){\colorbox{gray}}


\put(15,3){\colorbox{gray}}
\put(15,7){\colorbox{gray}}
\put(15,12){\colorbox{gray}}
\put(19,3){\colorbox{gray}}
\put(19,7){\colorbox{gray}}
\put(19,12){\colorbox{gray}}
\put(24,3){\colorbox{gray}}
\put(24,7){\colorbox{gray}}
\put(24,12){\colorbox{gray}}

\put(30,3){\colorbox{gray}}
\put(30,7){\colorbox{gray}}
\put(30,12){\colorbox{gray}}
\put(34,3){\colorbox{gray}}
\put(34,7){\colorbox{gray}}
\put(34,12){\colorbox{gray}}
\put(39,3){\colorbox{gray}}
\put(39,7){\colorbox{gray}}
\put(39,12){\colorbox{gray}}

\put(45,3){\colorbox{gray}}
\put(45,7){\colorbox{gray}}
\put(45,12){\colorbox{gray}}
\put(49,3){\colorbox{gray}}
\put(49,7){\colorbox{gray}}
\put(49,12){\colorbox{gray}}
\put(54,3){\colorbox{gray}}
\put(54,7){\colorbox{gray}}
\put(54,12){\colorbox{gray}}

\put(0,0){\framebox(15,15)}
\put(15,0){\framebox(15,15)}
\put(30,0){\framebox(15,15)}
\put(45,0){\framebox(15,15)}
\put(0,15){\framebox(15,15)}
\put(15,15){\framebox(15,15)}
\put(30,15){\framebox(15,15)}
\put(45,15){\framebox(15,15)}
\put(0,30){\framebox(15,15)}
\put(15,30){\framebox(15,15)}
\put(30,30){\framebox(15,15)}
\put(45,30){\framebox(15,15)}
\put(0,45){\framebox(15,15)}
\put(15,45){\framebox(15,15)}
\put(30,45){\framebox(15,15)}
\put(45,45){\framebox(15,15)}
\end{picture}
\end{center}
which we see by comparison with Example~\ref{ex:h=(2,3,4,4)} that it is formed from the Peterson diagram by adding the single box at the $(4,2)$-th 
spot. This time, the $4$ extra $\mathsf{S}$-fixed points in $\Hess(\mathsf{N},h)^\mathsf{S}$ which do not appear in $\Y^\mathsf{S}$ are 

\begin{align} \label{eq:fixed_point_h=(2,4,4,4)}
1342,1423,3241,4312. 
\end{align}
As in Example~\ref{ex:h=(3,3,4,4)} with $h=(3,3,4,4)$, we can compute in a straightforward manner the restrictions of the classes $\Delta_1, \Delta_2, \Delta_3$ to these $4$ extra fixed points and we find 
\begin{align*}
&(\Delta_1|_{1342},\Delta_1|_{1423},\Delta_1|_{3241},\Delta_1|_{4312})=(0,0,0,0) \\
&(\Delta_2|_{1342},\Delta_2|_{1423},\Delta_2|_{3241},\Delta_2|_{4312})=(-2t^2,2t^2,-6t^2,4t^2) \\
&(\Delta_3|_{1342},\Delta_3|_{1423},\Delta_3|_{3241},\Delta_3|_{4312})=(2t^2,-2t^2,6t^2,-4t^2) 
\end{align*}
so again we see that one of the elements, $\Delta_1$, vanishes on the $4$ extra fixed points, but $\Delta_2$ and $\Delta_3$ do not. However, as in Example~\ref{ex:h=(3,3,4,4)} it is easy to see that the element $\Delta_2+\Delta_3$ does vanish. Moreover, by examining the $4$ extra fixed points as we did in Example~\ref{ex:h=(3,3,4,4)}, we see that $w(2)-w(4)=1$ for all four fixed points and thus 
\[
(\SChNil_2 - \SChNil_4-t) \vert_w = 0
\]
for these $4$ fixed points. Now by the same reasoning as in Example~\ref{ex:h=(3,3,4,4)} we find 
that 
\begin{align*}
\Delta_2(\SChNil_2-\SChNil_4-t)=\xi_2(\SChNil_2-\SChNil_3-t)(\SChNil_2-\SChNil_4-t)=0
\end{align*}
in $H^*_\mathsf{S}(\Hess(\mathsf{N},h))$ and obtain 
a list of relations satisfied in $H^*_\mathsf{S}(\Hess(\mathsf{N},h))$: 
\begin{align*}
\Delta_1&=0 \\ 
\Delta_2(\SChNil_2-\SChNil_4-t)&=0 \\
\Delta_2+\Delta_3&=0 \\
\Delta_4&=0 
\end{align*}
which is similar to what we found for $h=(3,3,4,4)$. 
\end{Example}

At this point we may take a step back and consider what happened in Examples~\ref{ex:h=(3,3,4,4)} and~\ref{ex:h=(2,4,4,4)}. Note that we began with the elements $\Delta_1, \Delta_2, \Delta_3, \Delta_4$, which generated the ideal of relations for the Peterson variety case $h=(2,3,4,4)$. When we added a single box at the $(3,1)$ spot, what happened is that these relations ``morphed'' to the relations 
\begin{align*}
&\Delta_1(\SChNil_1-\SChNil_3-t) \\
&\Delta_1+\Delta_2 \\ 
&\Delta_3 \\
&\Delta_4
\end{align*}
whereas, if we instead added to the Peterson the single box at the $(4,2)$-th spot, the relations instead ``morphed'' into 
\begin{align*}
&\Delta_1 \\ 
&\Delta_2(\SChNil_2-\SChNil_4-t) \\
&\Delta_2+\Delta_3 \\
&\Delta_4. 
\end{align*}
We note that adding a box at the $(i,j)$-th appears, at least in these two examples, to entail a multiplication of one of the relations by a factor of $(\SChNil_j - \SChNil_i - t)$. We also see that we sometimes need to take the sum of two previous generators, and it seems plausible that there would be a rule to determine which sums must be taken.

With this in mind, let us consider two more examples. 

\begin{Example} \label{ex:h=(3,4,4,4)} 
Let $n=4$ still and take $h=(3,4,4,4)$. 
Then the box diagram is 
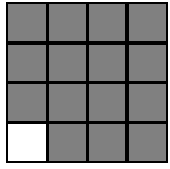
\begin{center}
\begin{picture}(60,60)
\put(0,48){\colorbox{gray}}
\put(0,52){\colorbox{gray}}
\put(0,57){\colorbox{gray}}
\put(4,48){\colorbox{gray}}
\put(4,52){\colorbox{gray}}
\put(4,57){\colorbox{gray}}
\put(9,48){\colorbox{gray}}
\put(9,52){\colorbox{gray}}
\put(9,57){\colorbox{gray}}

\put(15,48){\colorbox{gray}}
\put(15,52){\colorbox{gray}}
\put(15,57){\colorbox{gray}}
\put(19,48){\colorbox{gray}}
\put(19,52){\colorbox{gray}}
\put(19,57){\colorbox{gray}}
\put(24,48){\colorbox{gray}}
\put(24,52){\colorbox{gray}}
\put(24,57){\colorbox{gray}}

\put(30,48){\colorbox{gray}}
\put(30,52){\colorbox{gray}}
\put(30,57){\colorbox{gray}}
\put(34,48){\colorbox{gray}}
\put(34,52){\colorbox{gray}}
\put(34,57){\colorbox{gray}}
\put(39,48){\colorbox{gray}}
\put(39,52){\colorbox{gray}}
\put(39,57){\colorbox{gray}}

\put(45,48){\colorbox{gray}}
\put(45,52){\colorbox{gray}}
\put(45,57){\colorbox{gray}}
\put(49,48){\colorbox{gray}}
\put(49,52){\colorbox{gray}}
\put(49,57){\colorbox{gray}}
\put(54,48){\colorbox{gray}}
\put(54,52){\colorbox{gray}}
\put(54,57){\colorbox{gray}}

\put(0,33){\colorbox{gray}}
\put(0,37){\colorbox{gray}}
\put(0,42){\colorbox{gray}}
\put(4,33){\colorbox{gray}}
\put(4,37){\colorbox{gray}}
\put(4,42){\colorbox{gray}}
\put(9,33){\colorbox{gray}}
\put(9,37){\colorbox{gray}}
\put(9,42){\colorbox{gray}}

\put(15,33){\colorbox{gray}}
\put(15,37){\colorbox{gray}}
\put(15,42){\colorbox{gray}}
\put(19,33){\colorbox{gray}}
\put(19,37){\colorbox{gray}}
\put(19,42){\colorbox{gray}}
\put(24,33){\colorbox{gray}}
\put(24,37){\colorbox{gray}}
\put(24,42){\colorbox{gray}}

\put(30,33){\colorbox{gray}}
\put(30,37){\colorbox{gray}}
\put(30,42){\colorbox{gray}}
\put(34,33){\colorbox{gray}}
\put(34,37){\colorbox{gray}}
\put(34,42){\colorbox{gray}}
\put(39,33){\colorbox{gray}}
\put(39,37){\colorbox{gray}}
\put(39,42){\colorbox{gray}}

\put(45,33){\colorbox{gray}}
\put(45,37){\colorbox{gray}}
\put(45,42){\colorbox{gray}}
\put(49,33){\colorbox{gray}}
\put(49,37){\colorbox{gray}}
\put(49,42){\colorbox{gray}}
\put(54,33){\colorbox{gray}}
\put(54,37){\colorbox{gray}}
\put(54,42){\colorbox{gray}}

\put(0,18){\colorbox{gray}}
\put(0,22){\colorbox{gray}}
\put(0,27){\colorbox{gray}}
\put(4,18){\colorbox{gray}}
\put(4,22){\colorbox{gray}}
\put(4,27){\colorbox{gray}}
\put(9,18){\colorbox{gray}}
\put(9,22){\colorbox{gray}}
\put(9,27){\colorbox{gray}}

\put(15,18){\colorbox{gray}}
\put(15,22){\colorbox{gray}}
\put(15,27){\colorbox{gray}}
\put(19,18){\colorbox{gray}}
\put(19,22){\colorbox{gray}}
\put(19,27){\colorbox{gray}}
\put(24,18){\colorbox{gray}}
\put(24,22){\colorbox{gray}}
\put(24,27){\colorbox{gray}}

\put(30,18){\colorbox{gray}}
\put(30,22){\colorbox{gray}}
\put(30,27){\colorbox{gray}}
\put(34,18){\colorbox{gray}}
\put(34,22){\colorbox{gray}}
\put(34,27){\colorbox{gray}}
\put(39,18){\colorbox{gray}}
\put(39,22){\colorbox{gray}}
\put(39,27){\colorbox{gray}}

\put(45,18){\colorbox{gray}}
\put(45,22){\colorbox{gray}}
\put(45,27){\colorbox{gray}}
\put(49,18){\colorbox{gray}}
\put(49,22){\colorbox{gray}}
\put(49,27){\colorbox{gray}}
\put(54,18){\colorbox{gray}}
\put(54,22){\colorbox{gray}}
\put(54,27){\colorbox{gray}}


\put(15,3){\colorbox{gray}}
\put(15,7){\colorbox{gray}}
\put(15,12){\colorbox{gray}}
\put(19,3){\colorbox{gray}}
\put(19,7){\colorbox{gray}}
\put(19,12){\colorbox{gray}}
\put(24,3){\colorbox{gray}}
\put(24,7){\colorbox{gray}}
\put(24,12){\colorbox{gray}}

\put(30,3){\colorbox{gray}}
\put(30,7){\colorbox{gray}}
\put(30,12){\colorbox{gray}}
\put(34,3){\colorbox{gray}}
\put(34,7){\colorbox{gray}}
\put(34,12){\colorbox{gray}}
\put(39,3){\colorbox{gray}}
\put(39,7){\colorbox{gray}}
\put(39,12){\colorbox{gray}}

\put(45,3){\colorbox{gray}}
\put(45,7){\colorbox{gray}}
\put(45,12){\colorbox{gray}}
\put(49,3){\colorbox{gray}}
\put(49,7){\colorbox{gray}}
\put(49,12){\colorbox{gray}}
\put(54,3){\colorbox{gray}}
\put(54,7){\colorbox{gray}}
\put(54,12){\colorbox{gray}}

\put(0,0){\framebox(15,15)}
\put(15,0){\framebox(15,15)}
\put(30,0){\framebox(15,15)}
\put(45,0){\framebox(15,15)}
\put(0,15){\framebox(15,15)}
\put(15,15){\framebox(15,15)}
\put(30,15){\framebox(15,15)}
\put(45,15){\framebox(15,15)}
\put(0,30){\framebox(15,15)}
\put(15,30){\framebox(15,15)}
\put(30,30){\framebox(15,15)}
\put(45,30){\framebox(15,15)}
\put(0,45){\framebox(15,15)}
\put(15,45){\framebox(15,15)}
\put(30,45){\framebox(15,15)}
\put(45,45){\framebox(15,15)}
\end{picture}
\end{center}
The $\mathsf{S}$-fixed points $w$ for this choice of $h$ have all the fixed points appearing in Example~\ref{ex:h=(3,3,4,4)} above and, in addition, the following 
6 permutations: 
\begin{align} \label{eq:fixed_point_h=(3,4,4,4)}
1342,1423,2413,3241,4231,4312. 
\end{align}
Now consider the three elements which vanished in $H^*_\mathsf{S}(\Hess(\mathsf{N},(3,3,4,4)))$ in the $h=(3,3,4,4)$ case in Example~\ref{ex:h=(3,3,4,4)}, namely, 
\begin{align*}
f_1&:=\Delta_1(\SChNil_1-\SChNil_3-t)=\xi_1(\SChNil_1-\SChNil_2-t)(\SChNil_1-\SChNil_3-t),\\
f_2&:=\Delta_1+\Delta_2=\xi_1(\SChNil_1-\SChNil_2-t)+\xi_2(\SChNil_2-\SChNil_3-t),\\
f_3&:=\Delta_3=\xi_3(\SChNil_3-\SChNil_4-t). 
\end{align*}
Restricting the $f_1, f_2$ and $f_3$ to the 6 fixed points~\eqref{eq:fixed_point_h=(3,4,4,4)} we can compute that 
\begin{align*}
&(f_1|_{1342},f_1|_{1423},f_1|_{2413},f_1|_{3241},f_1|_{4231},f_1|_{4312})=(0,0,0,0,0,0) \\
&(f_2|_{1342},f_2|_{1423},f_2|_{2413},f_2|_{3241},f_2|_{4231},f_2|_{4312})=(-2t^2,2t^2,3t^2,-6t^2,-3t^2,4t^2) \\
&(f_3|_{1342},f_3|_{1423},f_3|_{2413},f_3|_{3241},f_3|_{4231},f_3|_{4312})=(2t^2,-2t^2,-3t^2,6t^2,3t^2,-4t^2). 
\end{align*}
Thus we find that $f_1$ still vanishes in the $h=(3,4,4,4)$ case, but $f_2$ and $f_3$ do not.  However, again we can easily see that 
$f_2+f_3 = \Delta_1 + \Delta_2 + \Delta_3$ does vanish. Moreover, notice that for all the $6$ ``new'' fixed points $w$ as listed in~\eqref{eq:fixed_point_h=(3,4,4,4)} we have $w(2)-w(4)=1$. This implies that 
\begin{align*}
(\SChNil_2-\SChNil_4-t)|_w=0  
\end{align*}
in $H^*_\mathsf{S}(\Hess(\mathsf{N},(3,4,4,4)))$. Following similar reasoning as in the above examples, we obtain that 
\begin{align*}
f_2(\SChNil_2-\SChNil_4-t)=(\Delta_1+\Delta_2)(\SChNil_2-\SChNil_4-t)=0 \in H^*_\mathsf{S}(\Hess(\mathsf{N},(3,4,4,4))).
\end{align*}
In summary, we have obtained the following relations 
\begin{align*}
\Delta_1(\SChNil_1-\SChNil_3-t)&=0 \\
(\Delta_1+\Delta_2)(\SChNil_2-\SChNil_4-t)&=0 \\
\Delta_1+\Delta_2+\Delta_3&=0 \\ 
\Delta_4&=0 
\end{align*}
in the $\mathsf{S}$-equivariant cohomology of $\Hess(\mathsf{N}, (3,4,4,4))$, and we may again guess that these are all the relations. 

\end{Example}

\begin{Example} \label{ex:h=(4,4,4,4)}
Finally, let $n=4$ still and consider $h=(4,4,4,4)$, the box diagram of which is 
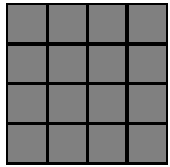
\begin{center}
\begin{picture}(60,60)
\put(0,48){\colorbox{gray}}
\put(0,52){\colorbox{gray}}
\put(0,57){\colorbox{gray}}
\put(4,48){\colorbox{gray}}
\put(4,52){\colorbox{gray}}
\put(4,57){\colorbox{gray}}
\put(9,48){\colorbox{gray}}
\put(9,52){\colorbox{gray}}
\put(9,57){\colorbox{gray}}

\put(15,48){\colorbox{gray}}
\put(15,52){\colorbox{gray}}
\put(15,57){\colorbox{gray}}
\put(19,48){\colorbox{gray}}
\put(19,52){\colorbox{gray}}
\put(19,57){\colorbox{gray}}
\put(24,48){\colorbox{gray}}
\put(24,52){\colorbox{gray}}
\put(24,57){\colorbox{gray}}

\put(30,48){\colorbox{gray}}
\put(30,52){\colorbox{gray}}
\put(30,57){\colorbox{gray}}
\put(34,48){\colorbox{gray}}
\put(34,52){\colorbox{gray}}
\put(34,57){\colorbox{gray}}
\put(39,48){\colorbox{gray}}
\put(39,52){\colorbox{gray}}
\put(39,57){\colorbox{gray}}

\put(45,48){\colorbox{gray}}
\put(45,52){\colorbox{gray}}
\put(45,57){\colorbox{gray}}
\put(49,48){\colorbox{gray}}
\put(49,52){\colorbox{gray}}
\put(49,57){\colorbox{gray}}
\put(54,48){\colorbox{gray}}
\put(54,52){\colorbox{gray}}
\put(54,57){\colorbox{gray}}

\put(0,33){\colorbox{gray}}
\put(0,37){\colorbox{gray}}
\put(0,42){\colorbox{gray}}
\put(4,33){\colorbox{gray}}
\put(4,37){\colorbox{gray}}
\put(4,42){\colorbox{gray}}
\put(9,33){\colorbox{gray}}
\put(9,37){\colorbox{gray}}
\put(9,42){\colorbox{gray}}

\put(15,33){\colorbox{gray}}
\put(15,37){\colorbox{gray}}
\put(15,42){\colorbox{gray}}
\put(19,33){\colorbox{gray}}
\put(19,37){\colorbox{gray}}
\put(19,42){\colorbox{gray}}
\put(24,33){\colorbox{gray}}
\put(24,37){\colorbox{gray}}
\put(24,42){\colorbox{gray}}

\put(30,33){\colorbox{gray}}
\put(30,37){\colorbox{gray}}
\put(30,42){\colorbox{gray}}
\put(34,33){\colorbox{gray}}
\put(34,37){\colorbox{gray}}
\put(34,42){\colorbox{gray}}
\put(39,33){\colorbox{gray}}
\put(39,37){\colorbox{gray}}
\put(39,42){\colorbox{gray}}

\put(45,33){\colorbox{gray}}
\put(45,37){\colorbox{gray}}
\put(45,42){\colorbox{gray}}
\put(49,33){\colorbox{gray}}
\put(49,37){\colorbox{gray}}
\put(49,42){\colorbox{gray}}
\put(54,33){\colorbox{gray}}
\put(54,37){\colorbox{gray}}
\put(54,42){\colorbox{gray}}

\put(0,18){\colorbox{gray}}
\put(0,22){\colorbox{gray}}
\put(0,27){\colorbox{gray}}
\put(4,18){\colorbox{gray}}
\put(4,22){\colorbox{gray}}
\put(4,27){\colorbox{gray}}
\put(9,18){\colorbox{gray}}
\put(9,22){\colorbox{gray}}
\put(9,27){\colorbox{gray}}

\put(15,18){\colorbox{gray}}
\put(15,22){\colorbox{gray}}
\put(15,27){\colorbox{gray}}
\put(19,18){\colorbox{gray}}
\put(19,22){\colorbox{gray}}
\put(19,27){\colorbox{gray}}
\put(24,18){\colorbox{gray}}
\put(24,22){\colorbox{gray}}
\put(24,27){\colorbox{gray}}

\put(30,18){\colorbox{gray}}
\put(30,22){\colorbox{gray}}
\put(30,27){\colorbox{gray}}
\put(34,18){\colorbox{gray}}
\put(34,22){\colorbox{gray}}
\put(34,27){\colorbox{gray}}
\put(39,18){\colorbox{gray}}
\put(39,22){\colorbox{gray}}
\put(39,27){\colorbox{gray}}

\put(45,18){\colorbox{gray}}
\put(45,22){\colorbox{gray}}
\put(45,27){\colorbox{gray}}
\put(49,18){\colorbox{gray}}
\put(49,22){\colorbox{gray}}
\put(49,27){\colorbox{gray}}
\put(54,18){\colorbox{gray}}
\put(54,22){\colorbox{gray}}
\put(54,27){\colorbox{gray}}

\put(0,3){\colorbox{gray}}
\put(0,7){\colorbox{gray}}
\put(0,12){\colorbox{gray}}
\put(4,3){\colorbox{gray}}
\put(4,7){\colorbox{gray}}
\put(4,12){\colorbox{gray}}
\put(9,3){\colorbox{gray}}
\put(9,7){\colorbox{gray}}
\put(9,12){\colorbox{gray}}

\put(15,3){\colorbox{gray}}
\put(15,7){\colorbox{gray}}
\put(15,12){\colorbox{gray}}
\put(19,3){\colorbox{gray}}
\put(19,7){\colorbox{gray}}
\put(19,12){\colorbox{gray}}
\put(24,3){\colorbox{gray}}
\put(24,7){\colorbox{gray}}
\put(24,12){\colorbox{gray}}

\put(30,3){\colorbox{gray}}
\put(30,7){\colorbox{gray}}
\put(30,12){\colorbox{gray}}
\put(34,3){\colorbox{gray}}
\put(34,7){\colorbox{gray}}
\put(34,12){\colorbox{gray}}
\put(39,3){\colorbox{gray}}
\put(39,7){\colorbox{gray}}
\put(39,12){\colorbox{gray}}

\put(45,3){\colorbox{gray}}
\put(45,7){\colorbox{gray}}
\put(45,12){\colorbox{gray}}
\put(49,3){\colorbox{gray}}
\put(49,7){\colorbox{gray}}
\put(49,12){\colorbox{gray}}
\put(54,3){\colorbox{gray}}
\put(54,7){\colorbox{gray}}
\put(54,12){\colorbox{gray}}

\put(0,0){\framebox(15,15)}
\put(15,0){\framebox(15,15)}
\put(30,0){\framebox(15,15)}
\put(45,0){\framebox(15,15)}
\put(0,15){\framebox(15,15)}
\put(15,15){\framebox(15,15)}
\put(30,15){\framebox(15,15)}
\put(45,15){\framebox(15,15)}
\put(0,30){\framebox(15,15)}
\put(15,30){\framebox(15,15)}
\put(30,30){\framebox(15,15)}
\put(45,30){\framebox(15,15)}
\put(0,45){\framebox(15,15)}
\put(15,45){\framebox(15,15)}
\put(30,45){\framebox(15,15)}
\put(45,45){\framebox(15,15)}
\end{picture}
\end{center}
The $\mathsf{S}$-fixed points in this case contain $6$ more permutations 
\begin{align} \label{eq:fixed_point_h=(4,4,4,4)}
2341,2431,3142,3412,4123,4213 
\end{align}
which were not present in the $h=(3,4,4,4)$ example. 
Let us consider the $3$ elements that vanished in Example~\ref{ex:h=(3,4,4,4)}, as listed at the end: 
\begin{align*}
g_1&:=\Delta_1(\SChNil_1-\SChNil_3-t)=\xi_1(\SChNil_1-\SChNil_2-t)(\SChNil_1-\SChNil_3-t),\\
g_2&:=(\Delta_1+\Delta_2)(\SChNil_2-\SChNil_4-t)=\xi_1(\SChNil_1-\SChNil_2-t)(\SChNil_2-\SChNil_4-t)+\xi_2(\SChNil_2-\SChNil_3-t)(\SChNil_2-\SChNil_4-t), \\
g_3&:=\Delta_1+\Delta_2+\Delta_3=\xi_1(\SChNil_1-\SChNil_2-t)+\xi_2(\SChNil_2-\SChNil_3-t)+\xi_3(\SChNil_3-\SChNil_4-t). 
\end{align*}
Following the examples we analyzed above, we can try restricting these elements to the new fixed points~\eqref{eq:fixed_point_h=(4,4,4,4)} and obtain 
\begin{align*}
&(g_1|_{2341},g_1|_{2431},g_1|_{3142},g_1|_{3412},g_1|_{4123},g_1|_{4213})=(6t^3,6t^3,-4t^3,-4t^3,6t^3,6t^3) \\
&(g_2|_{2341},g_2|_{2431},g_2|_{3142},g_2|_{3412},g_2|_{4123},g_2|_{4213})=(-6t^3,-6t^3,4t^3,4t^3,-6t^3,-6t^3) \\
&(g_3|_{2341},g_3|_{2431},g_3|_{3142},g_3|_{3412},g_3|_{4123},g_3|_{4213})=(0,0,0,0,0,0). 
\end{align*}
Thus we find that $g_1, g_2$ do not vanish but $g_3$ does. However, as before, $g_1+g_2=\Delta_1(\SChNil_1-\SChNil_3-t)+(\Delta_1+\Delta_2)(\SChNil_2-\SChNil_4-t)$ does vanish. Moreover, since $w(1)-w(4)=1$ for all $w$ in the list~\eqref{eq:fixed_point_h=(4,4,4,4)} we have as before that 
\begin{align*}
g_1(\SChNil_1-\SChNil_4-t)=\xi_1(\SChNil_1-\SChNil_2-t)(\SChNil_1-\SChNil_3-t)(\SChNil_1-\SChNil_4-t)=0. 
\end{align*}
Summarizing, we obtain the following list of relations for this (last) case of $h=(4,4,4,4)$: 
\begin{align*}
\Delta_1(\SChNil_1-\SChNil_3-t)(\SChNil_1-\SChNil_4-t)&=0 \\
\Delta_1(\SChNil_1-\SChNil_3-t)+(\Delta_1+\Delta_2)(\SChNil_2-\SChNil_4-t)&=0 \\
\Delta_1+\Delta_2+\Delta_3&=0 \\ 
\Delta_4&=0. 
\end{align*}
\end{Example}

This is where the fun begins.  Having computed the above examples, can we now guess a systematic method for constructing the appropriate relations for $H^*_\mathsf{S}(\Hess(\mathsf{N},h))$?  In the process of writing \cite{AHHM2019}, our approach (as hinted above) was to think of the process inductively, the base case being the Peterson variety $h=(2,3,4,\cdots, n,n)$, and then adding ``one box at a time''.  Suppose then that a Hessenberg function $h$ is obtained from $h'$ by adding the $(i,j)$-th box. We additionally visualized each box in the Hessenberg box diagram as containing a polynomial relation, where the relation for the $(i,j)$-th box can be inductively derived from the relations associated to the boxes in $h$. We also recall here that we had already guessed that adding a $(i,j)$-th box should somehow translate to a transformation which involves a multiplication by $(\SChNil_j - \SChNil_i - t)$.

\noindent \textbf{Suggested Exercise.} 
With the above discussion in mind, the reader is at this point invited to give it a try him/herself to attempt to come up with a systematic method for constructing a (candidate) list of relations for $H^*_\mathsf{S}(\Hess(\mathsf{N},h))$ which ``formalizes'' the phenomena we saw in the concrete examples above, before reading further. 

\medskip

We hope the reader has done the suggested exercise, and we now proceed to a more complete description of the thought process which lead us to Definition~\ref{definition:fij}. To do so, let us look again at the step in the inductive process when we add the $(4,1)$-box to $h'=(3,4,4,4)$ to obtain $h=(4,4,4,4)$:

\begin{align*}
\Delta_1(\SChNil_1-\SChNil_3-t) \quad & \quad \leadsto \quad \Delta_1(\SChNil_1-\SChNil_3-t)(\SChNil_1-\SChNil_4-t) \\
(\Delta_1+\Delta_2)(\SChNil_2-\SChNil_4-t)& \quad \leadsto \quad \Delta_1(\SChNil_1-\SChNil_3-t)+(\Delta_1+\Delta_2)(\SChNil_2-\SChNil_4-t)\\
\Delta_1+\Delta_2+\Delta_3& \quad \leadsto \quad \Delta_1+\Delta_2+\Delta_3\\ 
\Delta_4& \quad \leadsto \quad \Delta_4
\end{align*}
We think of the LHS as the relations for $h'=(3,4,4,4)$, and moreover, we associate the generators to the boxes at positions $(3,1),(4,2),(4,3),(4,4)$ respectively, reading from top to bottom. On the RHS, we think of these four generators as associated to the boxes at positions $(4,1), (4,2), (4,3), (4,4)$ respectively, also reading from top to bottom.

We now examine once again what happens under these transformations $\leadsto$ from the LHS to the RHS.  The top transformation is a multiplication by $\SChNil_1-\SChNil_4-t$, as we had suspected. 
In the next line, it is useful to recognize at this point that since $\Delta_1(\SChNil_1-\SChNil_3-t)+(\Delta_1+\Delta_2)(\SChNil_2-\SChNil_4-t)$ vanishes for $h'=(3,4,4,4)$ in any case, we could have taken on the LHS of the second line to be 
\[
\Delta_1(\SChNil_1-\SChNil_3-t)+(\Delta_1+\Delta_2)(\SChNil_2-\SChNil_4-t) 
\]
and we ask the reader to keep this in mind as we continue looking at further examples. 

Now let us examine the inductive step at which we added a box at position $(3,1)$ to $h'=(2,4,4,4)$ to obtain $h=(3,4,4,4)$.   Recall that we decided above that the second relation for $h=(3,4,4,4)$ could be taken to be $\Delta_1(\SChNil_1-\SChNil_3-t)+(\Delta_1+\Delta_2)(\SChNil_2-\SChNil_4-t)=0$. Writing the relations obtained in the above exposition for $h'=(2,4,4,4)$ on the LHS and those for $h = (3,4,4,4)$ on the RHS, we have 
\begin{align*}
\Delta_1& \quad \leadsto \quad \Delta_1(\SChNil_1-\SChNil_3-t) \\
\Delta_2(\SChNil_2-\SChNil_4-t)& \quad \leadsto \quad \Delta_1(\SChNil_1-\SChNil_3-t)+(\Delta_1+\Delta_2)(\SChNil_2-\SChNil_4-t)\\
\Delta_2+\Delta_3& \quad \leadsto \quad \Delta_1+\Delta_2+\Delta_3\\ 
\Delta_4& \quad \leadsto \quad \Delta_4
\end{align*}
We see that at the first (top) line, the transformation from $h'$ to $h$ involves a multiplication by $(\SChNil_1 - \SChNil_3-t)$ as suspected. 
For the second and third lines, from the above lists we see that there appear to be transformations occurring, but upon inspection we see (as in the case above) that since $\Delta_1=0$, the original relations for the $h'=(2,4,4,4)$ case could be replaced by 
$\Delta_1(\SChNil_1-\SChNil_3-t)+(\Delta_1+\Delta_2)(\SChNil_2-\SChNil_4-t)$ and $\Delta_1+\Delta_2+\Delta_3$ respectively without causing any disruption in the ideal of relations.  Thus henceforth we may imagine that the relations for $h'$ are altered in this way and the complete list is 
\begin{align*}
&\Delta_1 \\
&\Delta_1(\SChNil_1-\SChNil_3-t)+(\Delta_1+\Delta_2)(\SChNil_2-\SChNil_4-t)\\
&\Delta_1+\Delta_2+\Delta_3\\ 
&\Delta_4. 
\end{align*}
Here we are imagining that the above relations are associated to the boxes at positions 
$(2,1), (4,2), (4,3), (4,4)$ respectively.

As a final consideration, let us examine the step in which $h'=(3,3,4,4)$ and we add the $(4,2)$-th box to obtain $h=(3,4,4,4)$. We start with 
the transformation of relations as follows 
\begin{align*}
\Delta_1(\SChNil_1-\SChNil_3-t)& \quad \leadsto \quad \Delta_1(\SChNil_1-\SChNil_3-t) \\
\Delta_1+\Delta_2& \quad \leadsto \quad \Delta_1(\SChNil_1-\SChNil_3-t)+(\Delta_1+\Delta_2)(\SChNil_2-\SChNil_4-t)\\
\Delta_3& \quad \leadsto \quad \Delta_1+\Delta_2+\Delta_3\\ 
\Delta_4& \quad \leadsto \quad \Delta_4
\end{align*}
but following reasoning similar to the above cases, we may replace the third relation on the LHS with $\Delta_1+\Delta_2+\Delta_3$ (since $\Delta_1+\Delta_2=0$), resulting in the fact that only the second relation undergoes a non-trivial transformation.  In this case, this non-trivial transformation at the second line, corresponding to the box in position $(4,2)$, involves two operations: multiplication by $(\SChNil_2-\SChNil_4-t)$, and then an addition of the term $\Delta_1(\SChNil_1-\SChNil_3-t)$.  This is a new phenomenon that we have not seen before! How should we interpret this?  The reader may also be wondering whether we could simply take $(\Delta_1+\Delta_2)(\SChNil_2-\SChNil_4-t)$ as the relation associated to the box $(4,2)$, since this vanishes for $h=(3,4,4,4)$. However, thinking back to the computation for $h=(4,4,4,4)$, we also recall that $(\Delta_1+\Delta_2)(\SChNil_2-\SChNil_4-t)$ did \emph{not} vanish for $h=(4,4,4,4)$, whereas the sum $\Delta_1(\SChNil_1-\SChNil_3-t)+(\Delta_1+\Delta_2)(\SChNil_2-\SChNil_4-t)$ does vanish. This seems to suggest to us that the ``correct'' relation (polynomial) to associated to box $(4,2)$ is $\Delta_1(\SChNil_1-\SChNil_3-t)+(\Delta_1+\Delta_2)(\SChNil_2-\SChNil_4-t)$ and not $(\Delta_1+\Delta_2)(\SChNil_2-\SChNil_4-t)$.

As just commented, in the previous paragraph we have come across a process by which we both multiplied a previous relation by a linear term involving $\tau_i$'s and $t$, and then added another relation. 
It is now time to notice that this ``other relation'' that we added -- at least, in the example above -- is exactly the relation $\Delta_1(\SChNil_1-\SChNil_3-t)$ which is associated to the box $(3,1)$. At this point we have the following idea: what this process is doing (in this case) is: 
\textbf{ to obtain the relation for box $(4,2)$, take the relation for box $(3,2)$ and multiply by $(\SChNil_2-\SChNil_4-t)$, then add the relation for box $(3,1)$.}    This then immediately generalizes to the following guess for a general algorithm. Namely: \textbf{ to obtain the relation for box $(i,j)$, take the relation for the box $(i-1,j)$ and multiply by $(\SChNil_j-\SChNil_i-t)$, then add the relation for box $(i-1,j-1)$. } (Note that we saw the second part of this algorithm (``add the box $(i-1,j-1)$'') for the first time in the $h'=(3,3,4,4)$ and $h=(3,4,4,4)$ example because all the other examples had taken $j=1$, in which case there is no box $(i-1,j-1)$, so the second part of the algorithm does nothing.)

The general algorithm boldfaced in the previous paragraph is just a guess!  Let us know see how to start the process by defining the relations for the base cases, which we take to be the boxes $(j,j)$ for $1 \leq j \leq n$ along the main diagonal, and see what kind of relations we obtain.  For simplicity we illustrate with $n=4$, but it should be clear that nothing we say really depends on this restriction. 
To define the relations along the main diagonal, the most natural thing to try is 
\[
f_{j,j} := p_j \, \, \textup{ for } 1 \leq j \leq 4.
\]
In the case $h=(1,2,3,4)$, the corresponding regular nilpotent Hessenberg variety is $\Hess(\mathsf{N}, (1,2,3,4)) = \{\mathrm{id}\}$ is a single point and $\xi_j(id)=\sum_{k=1}^j (\SChNil_k(id)-kt)=\sum_{k=1}^j (kt-kt)=0$, so these are indeed appropriate relations for these boxes. Now if we follow the algorithm described above (boldfaced), we would then recursively define the $f_{i,j}$ for $i>j$ as
follows: 
\begin{align*}
f_{2,1}&=(\SChNil_1-\SChNil_2-t)f_{1,1}=(\SChNil_1-\SChNil_2-t)p_1 \\
f_{3,2}&=(\SChNil_2-\SChNil_3-t)f_{2,2}+f_{2,1}=(\SChNil_1-\SChNil_2-t)p_1+(\SChNil_2-\SChNil_3-t)p_2 \\
f_{4,3}&=(\SChNil_3-\SChNil_4-t)f_{3,3}+f_{3,2}=(\SChNil_1-\SChNil_2-t)p_1+(\SChNil_2-\SChNil_3-t)p_2+(\SChNil_3-\SChNil_4-t)p_3 \\
f_{3,1}&=(\SChNil_1-\SChNil_3-t)f_{2,1}=(\SChNil_1-\SChNil_3-t)(\SChNil_1-\SChNil_2-t)p_1 \\
f_{4,2}&=(\SChNil_2-\SChNil_4-t)f_{3,2}+f_{3,1}\\
&=(\SChNil_1-\SChNil_3-t)(\SChNil_1-\SChNil_2-t)p_1+(\SChNil_2-\SChNil_4-t)\left((\SChNil_1-\SChNil_2-t)p_1+(\SChNil_2-\SChNil_3-t)p_2\right) \\
f_{4,1}&=(\SChNil_1-\SChNil_4-t)f_{3,1}=(\SChNil_1-\SChNil_4-t)(\SChNil_1-\SChNil_3-t)(\SChNil_1-\SChNil_2-t)p_1 
\end{align*}
and for any given Hessenberg function $h$, it is then possible to check, following the methods in the examples above, that the four relations $f_{h(1),1}, f_{h(2),2}, f_{h(3),3}, f_{h(4),4}$ corresponding to the boxes ``at the bottom'' of the Hessenberg box diagram of $h$ are indeed $0$ in $H^*_S(\Hess(\mathsf{N},h))$.

From here it is straightforward to conjecture the statement given in our main theorem. 
At this point we invite the reader to go back and view Definition~\ref{definition:fij}, bearing in mind this motivating discussion. 
The natural conjecture to make at this point formalizes the guess, already hinted at above, that the $f_{h(j),j}$ for $j \in [n]$ generate the ideal of relations in the equivariant cohomology ring. 
Having come this far, we are finally in a position to state the result in \cite{AHHM2019}.

\begin{Theorem} \label{theorem:reg nilp Hess cohomology} (\cite[Theorem~3.3]{AHHM2019})
Let $n$ be a positive integer and $h: [n] \to [n]$ a Hessenberg function. Let $\Hess(\mathsf{N},h) \subset \fln$ denote the corresponding regular nilpotent Hessenberg variety equipped with the action of the 1-dimensional subtorus $\S1$. Then the restriction map 
\begin{equation*}
H^*_T(\fln) \to H^*_\mathsf{S}(\Hess(\mathsf{N},h))
\end{equation*} 
is surjective, and 
there is an isomorphism of graded $\mathbb{Q}[t]$-algebras
\begin{equation*}
H^{\ast}_{\mathsf{S}}(\mathcal \Hess(\mathsf{N},h))\cong \mathbb{Q}[x_1,\dots,x_n,t]/\Ih 
\end{equation*}
sending $x_i$ to $\SChNil_i$ and $t$ to $t$, 
where $\SChNil_i$ is the $\mathsf{S}$-equivariant first Chern class of the dual of the
tautological line bundle  restricted to $\Hess(\mathsf{N},h)$ and we
identify $H^{\ast}(B\mathsf{S}) \cong \mathbb{Q}[t]$, and the ideal $\Ih$ is defined by 
\begin{equation}\label{eq:def Ih} 
\Ih:= (f_{h(j),j} \mid j \in [n] ).
\end{equation} 
\end{Theorem}

The above result concerns equivariant cohomology.  As usual, we can also 
give the analogous statement for ordinary cohomology. 
For any pair $i, j$ with $1 \leq j \leq i \leq n$, let $\cf_{i,j}$ be the polynomial $f_{i,j}(x_1,\ldots,x_n,t=0)$ in $\Q[x_1,\ldots,x_n]$. 
In other words, $\cf_{i,j}$ is defined as the following recursive formula:
\begin{align*}
\cf_{j,j}&=\sum_{k=1}^j x_k \quad \textup{ for } j \in [n], \\
\cf_{i,j}&=\cf_{i-1,j-1}+\big(x_j-x_i\big)\cf_{i-1,j} \quad \textup{ for } n\ge i>j\ge 1
\end{align*}
with the convention $\cf_{*,0}:=0$ for any $*$.

\begin{Theorem}\label{theorem: main reg nilpotent} (\cite[Theorem~A]{AHHM2019})
  Let $n$ be a positive integer and $h: [n] \to
  [n]$ a Hessenberg function. Let $\mathsf{N}$
  denote a regular nilpotent matrix in $\mathfrak{gl}_n(\C)$ and let
 $\Hess(\mathsf{N},h) \subset  \fln$ be
  the associated regular nilpotent Hessenberg variety. Then the
  restriction map
\[
H^*(\fln) \to H^*(\Hess(\mathsf{N},h))
\]
is surjective, and there is an isomorphism of graded $\Q$-algebras
\begin{equation}\label{eq:intro Theorem A} 
H^*(\Hess(\mathsf{N},h)) \cong \Q[x_1, \ldots, x_n]/\check \Ih
\end{equation}
where $\check \Ih$ is the ideal of $\Q[x_1, \ldots, x_n]$ defined by
\begin{equation}\label{definition ideal check Ih} 
\check \Ih := (\cf_{h(j),j} \mid j \in [n] ). 
\end{equation}
\end{Theorem}

\begin{Remark} 
It can be shown that Borel's famous presentation of the cohomology of the flag variety (with rational coefficients) can be recovered from Theorem~\ref{theorem: main reg nilpotent}, see \cite{AHHM2019}. 
\end{Remark}

\subsection{Sketch of proof of Theorem~\ref{theorem:reg nilp Hess cohomology}}\label{subsec: sketch proof for coh reg nilp}

In this section we briefly sketch the proof of Theorem~\ref{theorem:reg nilp Hess cohomology}. The structure of the argument is the same as for the previous proofs in this manuscript, namely:
\begin{enumerate} 
\item We construct a ``candidate'' list of relations $\{f_{h(j),j} \, \mid \, j \in [n]\}$ for $H^*_{\S1}(\Hess(\mathsf{N},h))$, 
\item we show using $\S1$-equivariant localization to fixed points that each candidate relation does map to $0$ in $H^*_{\S1}(\Hess(\mathsf{N},h))$, 
\item we compute the Hilbert series of the domain and codomain of the ring map $\varphi_h: \Q[x_1,\cdots, x_n, t] \big/ \langle f_{h(j),j} \, \mid \, j \in [n] \rangle \to H^*_{\S1}(\Hess(\mathsf{N},h))$ and show that the two Hilbert series are equal, 
\item we show that $\varphi_h$ is injective; by (3) above, this suffices to prove that $\varphi_h$ is an isomorphism. 
\end{enumerate}

The point of Section~\ref{subsec: derivation fij} was to motivate and explain step (1) of the process above. In fact, the lengthy discussion which prepared us for the definition of the $f_{i,j}$ suggest that it is not an entirely trivial matter to prove (2), and indeed, this part of the argument is the longest and most technical part of \cite{AHHM2019}. It being a case-by-case and painstaking analysis, which does not lend itself to an intuitive big-picture discussion, we only give the briefest of sketches of part (2) below. Details may be found in \cite{AHHM2019}.  Step (3) of the argument is essentially the same sort of Hilbert series argument seen multiple times in this manuscript, so we omit further discussion of (3). In contrast, step (4) in the process above requires a fundamentally new idea not yet seen in this paper, so we give an expository account of it below.   We now proceed with the explanations as advertised. 

We begin with step (2). For this, we first recall that the set of Hessenberg functions $H_n$ has a natural partial order. 
\begin{Definition} \label{def:partial order} 
  Let $h', h \in H_n$. Then we say $h' \subset h$ if $h'(j)
  \leq h(j)$ for all $j \in [n]$. 
\end{Definition} 
The relation $h' \subset h$ is evidently a partial order on $H_n$.  Note that from the definition of $\Hess(\mathsf{N},h)$ it is immediate that 
$
h' \subset h$ implies $\Hess(\mathsf{N},h') \subset \Hess(\mathsf{N},h)
$
which explains our choice of notation. Visually, it also corresponds to containment of the Hessenberg box diagrams. 

With the partial order in hand, we can briefly sketch the idea for the proof of part (2). Firstly, to prove (2), by $\S1$-equivariant localization to the fixed points, it suffices to show that for all $j \in [n]$ and for any $w \in \Hess(\mathsf{N},h)^{\S1}$ that we have $f_{h(j),j}(w)=0$. Secondly, we complete the partial order on $H_n$ to a total order and induct, where the base case is the Peterson Hessenberg function $h=(2,3,4,\cdots, n,n)$, which we know holds due to Theorem~\ref{theo:3.1}. Thirdly, it turns out that for any $w \in \Hess(\mathsf{N},h)^{\S1}$ there exists a unique minimal (with respect to the partial order above) Hessenberg function, denoted $h_w$, for which $w \in \Hess(\mathsf{N},h)^{\S1}$. By the inductive defiinition of the $f_{i,j}$, it follows that it suffices to check the claim for $w$ with $h=h_w$. The remainder of the proof is a careful case-by-case analysis \cite{AHHM2019}.

We now give some further explanation of step (4), as promised. As already mentioned, this step involves a genuinely new idea, so our exposition will be more leisurely.  
Firstly, we observe that our inductive definition of the polynomials $f_{i,j}$ implies that, for any Hessenberg function $h$, it follows that the ideal $I_{(n,n,\ldots,n)}$ corresponding to $h=(n,n,\ldots,n)$ is contained in the ideal $I_h$ corresponding to $h$. This means that we always have a surjective map of rings $\Q[x_1\ldots,x_n,t]/I_{(n,n,\ldots,n)} \to \Q[x_1,\ldots,x_n,t]/I_h$. 
On the other hand, recall that if $h=(n,n,\ldots,n)$, the associated regular nilpotent Hessenberg variety is the full flag variety $\Flags(\mathbb{C}^n)$, and in this special case we already know that the map $\varphi := \varphi_{(n,n,\ldots,n)}$ is surjective since the Chern classes $\tau_1^S, \ldots, \tau_n^S, t$ are known to generate the equivariant cohomology ring of $\Flags(\mathbb{C}^n)$ (Propositions~\ref{proposition:generators equivariant cohomology}-(1) and \ref{proposition:flag generators non-equivariant}). Since the Hilbert series of both sides are identical, we then know that $\varphi$ is an isomorphism. 
Secondly, we
consider localizations of the rings in question with respect to the multiplicative subset $R:=\Q[t]\backslash\{0\}$.

The following commutative diagram of localized rings (i.e., localized at the multiplicative subset $R=\Q[t] \backslash \{0\}$) is crucial. 
\[
\begin{CD}
R^{-1}\big(\Q[x_1,\dots,x_n,t]/I_{(n,n,\ldots,n)}\big) @>R^{-1}\varphi >\cong> R^{-1}H^*_\mathsf{S}(\Flags(\C^n)) @>>\cong> R^{-1}H^*_\mathsf{S}(\Flags(\C^n)^\mathsf{S}) \\
@VV\text{surj}V @VVV @VV\text{surj}V \\
 R^{-1}\big(\Q[x_1,\dots,x_n,t]/\Ih\big) @>R^{-1}\varphih >> R^{-1}H^*_\mathsf{S}(\Hess(\mathsf{N},h)) @>>\cong> R^{-1}H^*_\mathsf{S}(\Hess(\mathsf{N},h)^\mathsf{S}) 
\end{CD}
\]
Now, the key point is that by the localization theorem (Theorem~\ref{theorem:localization_theorem}), the horizontal arrows in the right-hand square are isomorphisms. Moreover, since we already noted that $\varphi$ is an isomorphism, so is $R^{-1}\varphi$. Finally, the rightmost and leftmost vertical arrows are surjective since the original (i.e. non-localized) ring maps are surjective. Putting this together, this implies that $R^{-1}\varphi_h$ is also surjective, and a comparison of Hilbert series shows that $R^{-1}\varphi_h$ is an isomorphism \cite{AHHM2019}. 

The above (sketch of) argument shows that the localized version $$R^{-1}\varphi_h: R^{-1}\big(\Q[x_1,\dots,x_n,t]/\Ih\big) \to R^{-1}H^*_\mathsf{S}(\Hess(\mathsf{N},h))$$ of $\varphih$ is an isomorphism. In order to show that $\varphih$ itself is injective, which would complete the proof, we consider the commutative diagram
\[
 \begin{CD}
 \Q[x_1,\dots,x_n,t]/\Ih @>\varphih>> H^*_\mathsf{S}(\Hess(\mathsf{N},h)) \\
 @VV\text{inj}V @VV\text{inj}V \\
 R^{-1}\Q[x_1,\dots,x_n,t]/\Ih @>R^{-1}\varphih >\cong > R^{-1}H^*_\mathsf{S}(\Hess(\mathsf{N},h)). 
 \end{CD}
 \]
 Here it follows from general principles that the vertical arrows are injections. 
 Indeed, $f_{h(1),1}, f_{h(2),2}, \cdots, f_{h(n),n}, t$ form a regular sequence in $\Q[x_1,\dots,x_n,t]$, so $t$ is a non-zero-divisor in $\Q[x_1,\dots,x_n,t]/\Ih$ \cite{AHHM2019}. This implies that the left vertical arrow is injective.
The injectivity for the right vertical arrow follows from \eqref{eq:equivariant cohomology of Hess(N,h) free} and knowing that $H^*_{\S1}(\Hess(\mathsf{N},h))$ is a free $H^*_{\S1}(\pt)$-module. 
 From this it follows that $\varphi_h$ is an injection, as desired. 
 This concludes our sketch of step (4) and hence of Theorem~\ref{theorem:reg nilp Hess cohomology}.

\section{Recent developments}\label{section.further.developments}

Hessenberg varieties turn out to be intimately linked with many other areas of research, thus
contributing to an ever-widening community of mathematicians interested in the topic.  In what follows, we briefly recount some of the most recent developments in the study of cohomology rings of Hessenberg varieties, giving 
some context, motivation, and commentary on potential future directions, among others. We 
emphasize that this list is not intended to be exhaustive. 

Before proceeding, it should also be noted that one of the major motivations for the study of these cohomology rings in the last decade is through the connection between these rings and the famous Stanley-Stembridge conjecture in combinatorics and the theory of symmetric functions. This is a rich topic in its own right.  
The related references are \cite{BrosnanChow2018}, \cite{Guay-Paquet}, and \cite{SW}.
Here, we therefore restrict our comments to recent developments in this area which are not as directly related to the Stanley-Stembridge conjecture. 
We also note that the study of the cohomology rings of regular nilpotent Hessenberg varieties is deeply related to the theory of hyperplane arrangements. 
We refer the reader to \cite{AHMMS} and \cite{EnoHorNagTsu-cohomologypresentation}.

\subsection{The cohomology rings of regular
Hessenberg varieties}

We begin with the example of the Hessenberg space 
$H_1=\mathfrak{b} \bigoplus \oplus_{i=1}^n \mathfrak{g}_{-\alpha_i}$. (In Lie type A, this 
choice of Hessenberg space corresponds to the Hessenberg function $h=(2,3,4,\cdots, n,n)$.) 
With respect to this choice, the corresponding regular nilpotent Hessenberg variety $\Hess(\mathsf{N},H_1)$
and the regular semisimple Hessenberg variety $\Hess(\mathsf{S},H_1)$ correspond to the 
Peterson variety and the ``toric variety associated to the weight polytope'' (also called the permutohedral variety in type A) respectively 
\cite{DeMariProcesiShayman1992}.  
Since the Weyl group $W$ acts naturally on the weight polytope, there is also a natural corresponding action 
of $W$ on the cohomology ring of the toric variety $\Hess(\mathsf{S},H_1)$. 
Klyachko showed in \cite{Klyachko-Orbits, Klyachko-ToricFlag} 
that the $W$-invariant subring of $H^*(\Hess(\mathsf{S},H_1))$ has a presentation by generators and relations, 
and this presentation turned out to be the same as that given for the cohomology ring of the Peterson variety 
$\Hess(\mathsf{N},H_1)$ of Theorem~\ref{mainthm Lie terminology}.  This unexpected correspondence gave rise to the natural question of whether or not 
this phenomenon works more generally, for other (or all?) Hessenberg spaces, where we equip the cohomology 
rings of general regular semisimple Hessenberg varieties -- for any Hessenberg space/function -- with the
``dot'' $W$-action defined by Tymoczko in \cite{Tymoczko08}.

In the case of Lie type A,
in \cite{AHHM2019} the authors obtain an isomorphism $H^*(\Hess(\mathsf{N},h)) \stackrel{\cong}{\rightarrow} H^*(\Hess(\mathsf{S},h))^{S_n}$ in such a way that the following diagram is commutative: 
\[
  \xymatrix{
  & H^*(\fln) \ar@{->>}[dl] \ar@{->>}[dr]& \\
  H^*(\mbox{Hess}(\mathsf{N},h)) \ar[rr]^{\cong}& & H^*(\mbox{Hess}(\mathsf{S},h))^{S_n}
  }
\]
and the construction of the bottom horizontal isomorphism requires the detailed knowledge of $H^*(\Hess(\mathsf{N},h))$ coming from Theorem~\ref{theorem: main reg nilpotent}.  Shortly thereafter, the result above was generalized using techniques from hyperplane arrangements to cover arbitrary Lie types \cite{AHMMS}. Specifically, the authors of \cite{AHMMS} show that the cohomology of the regular nilpotent Hessenberg variety $H^*(\Hess(\mathsf{N},H))$ is isomorphic as a ring to the $W$-invariant subring $H^*(\Hess(\mathsf{S},H))^W$ where $W$ is the relevant Weyl group.

The above relationship between the cohomology rings of regular nilpotent and regular semisimple Hessenberg varieties, suggests natural generalizations, and these topics are relevant to the Stanley-Stembridge conjecture. For simplicity, we restrict the following discussion to the case of Lie type A. 
Let $\mathsf{X}$ be a \textbf{regular} element of $\gl_n(\C)$ -- i.e., in the Jordan canonical form of $\mathsf{X}$, the eigenvalues corresponding to distinct Jordan blocks are all \emph{distinct}.   
Suppose that the Jordan canonical form of $\mathsf{X}$ has Jordan blocks of size $\lambda_1, \lambda_2, \cdots, \lambda_m$ (where we may assume $\lambda_1 \geq \lambda_2 \geq \cdots \geq \lambda_m$). Then we say $\mathsf{X}$ is \textbf{regular of type $\lambda = (\lambda_1, \lambda_2, \cdots, \lambda_m)$.} Note that $\lambda$ is a partition of $n$, i.e. $\lambda_1+\lambda_2+\cdots+\lambda_m=n$.  In what follows we denote a regular operator of type $\lambda$ by $\mathsf{R}_\lambda$  and we call its associated 
Hessenberg variety $\Hess(\mathsf{R}_\lambda, h)$ a \textbf{regular Hessenberg variety}. It is not hard to see that the regular nilpotent Hessenberg variety is a regular Hessenberg variety with $\lambda=(n)$, whereas a regular semisimple Hessenberg variety is an instance with $\lambda=(1,1,\cdots,1)$. 

Now, given a partition $\lambda = (\lambda_1, \lambda_2, \cdots, \lambda_m)$ of $n$, let $S_\lambda$ denote the subgroup of $S_n$ given by 
$S_{\lambda_1} \times S_{\lambda_2} \times \cdots \times S_{\lambda_m}$.  In \cite{BrosnanChow2018}, Brosnan and Chow showed that the $S_\lambda$-invariant subspace $H^*(\Hess(\mathsf{S},h))^{S_\lambda}$ of the $S_n$-representation $H^*(\Hess(\mathsf{S},h))$
is additively isomorphic (i.e. as vector spaces) to the cohomology ring of $\Hess(\mathsf{R}_\lambda, h)$, i.e., 
\begin{equation} \label{eq:Section9_1}
H^*(\Hess(\mathsf{R}_\lambda,h)) \cong H^*(\Hess(\mathsf{S},h))^{S_\lambda}
\end{equation}
and this was a key step in their proof of the Shareshian-Wachs conjecture, which established the connection between Hessenberg varieties and the Stanley-Stembridge conjecture. 
More recently, Balibanu and Crooks show that the identification~\eqref{eq:Section9_1} is in fact a ring isomorphism for all Lie types \cite{BalCro} (see also \cite{VilXue}). Moreover, they show that the cohomology ring $H^*(\Hess(\mathsf{R}_\lambda, h))$ is a Poincar\'e duality algebra. As far as we are aware, it remains an open question whether there exist bases of geometrically defined classes (e.g. corresponding to subvarieties of $\Hess(\mathsf{R}_\lambda,h)$) with respect to which the Poincar\'e duality can be explicitly and positively realized, and such that their multiplicative structure constants have interesting combinatorial interpretations.

\subsection{Hessenberg varieties and the cohomology rings of toric varieties}

Let $\lambda=(\lambda_1, \cdots, \lambda_m)$ be a partition of $n$. 
As another application of the isomorphism~\eqref{eq:Section9_1} in a different direction, it was recently observed \cite{HorMasShaSon} that in the case of $h=h_1=(2,3,4,\cdots, n,n)$, the cohomology ring $H^*(\Hess(\mathsf{R}_\lambda,h_1))$ is isomorphic to the cohomology ring of a certain toric variety. To see this, first recall that $\Hess(\mathsf{R}_{(1,1,\cdots,1)},h_1)$ is the toric variety associated to the permutohedron, which we denote $P_n$, in $\R^n$. One way to realize $P_n$ is by taking 
\begin{align*}
P_n=P_n(a_1, \ldots, a_n):=\mbox{ConvexHull}\{ (a_{w(1)}, \ldots, a_{w(n)}) \in \R^n \mid w \in S_n \}
\end{align*}
where the $a_1,\ldots,a_n \in \mathbb{R}_{>0}$ are pairwise distinct. Now, given a partition $\lambda$ of $n$, we may define a polytope $P_\lambda$ following \cite{HorMasShaSon}: 
$$
P_\lambda:= \{(x_1,\ldots,x_n) \in P_n \mid x_k \leq x_{k+1} \ {\rm for} \ k \ {\rm with} \ s_k \in S_\lambda \}.
$$
The idea is that the polytope $P_\lambda$ is a realization of the orbit space $P_n/S_\lambda$.  Now let $X(P_\lambda)$ 
denote the toric variety associated to $P_\lambda$.  The arguments of 
\cite{HorMasShaSon} show that the cohomology ring $H^*(X(P_\lambda))$ is isomorphic, as a ring, to the 
$S_\lambda$-invariant subring of $H^*(X_n = \Hess(\mathsf{R}_{(1,1,\cdots,1)}, h_1))$. On the other 
hand, from~\eqref{eq:Section9_1} we know that $H^*(\Hess(\mathsf{R}_\lambda, h_1)) \cong H^*(X_n)^{S_\lambda}$. 
Putting this together we obtain the ring isomorphisms 
\begin{equation*} 
H^*(\Hess(\mathsf{R}_\lambda,h_1)) \cong H^*(X_n)^{S_\lambda} \cong H^*(X(P_\lambda)).
\end{equation*}
Thus we get a (new) connection between the theory of Hessenberg varieties and that of toric orbifolds. 
In fact, the above result holds also in the classical Lie types \cite{HorMasShaSon}.

\subsection{Poincar\'e duals of Hessenberg varieties}

As is suggested by the ring isomorphisms~\eqref{eq:Section9_1}, it seems natural to ask whether there is some underlying 
geometric relationship between the $\Hess(\mathsf{R}_\lambda,h)$ for different partitions $\lambda$.   One way to investigate this question is to consider their Poincar\'e duals in $H^*(\fln)$.  In the special case $\lambda=(n)$ and $\lambda=(1,1,\cdots,1)$ and $h$ a Hessenberg function with $h(i)\geq i+1$ for $i<n$, it was shown in  \cite{ADGH} that these agree, i.e., 
\[
[\Hess(\mathsf{N},h)]=[\Hess(\mathsf{S},h)] \ {\rm in } \ H^*(\fln). 
\]
More generally, under the same hypothesis that $h(i) \geq i+1$, the work of Abe, Fujita, and Zeng shows that the Poincar\'e dual $[\Hess(\mathsf{R}_\lambda, h)] \in H^*(\fln)$ is independent of $\lambda$ \cite{AbeFujitaZeng2020}. 
Although we explain the case of Lie type $A$, this independence of the Poincar\'e dual is in fact stated in \cite[Corollary~3.9]{AbeFujitaZeng2020} for all Lie types. 

Another natural question arises if we now vary the Hessenberg function $h$ instead of the partition $\lambda$.  It was recently shown by Enokizono, Horiguchi, Nagaoka, and Tsuchiya \cite{EnoHorNagTsu-additivebasis} that certain Poincar\'e duals $[\Hess(\mathsf{N},h')] \in H^*(\fln)$, as we vary the $h'$, are linearly independent in $H^*(\fln)$. We give more details on this result in Section~\ref{subsec:additive.basis}.

\subsection{Hessenberg Schubert calculus} 

We saw in Theorem~\ref{theorem:pvA-basis} that there exists a computationally convenient $H^*_{\S1}(\pt)$-module basis of the $\S1$-equivariant cohomology $H^*_{\S1}(\Y)$ of the Peterson variety $\Y$ consisting of the Peterson Schubert classes $p_{v_{\mathcal{A}}}$. Moreover, in Theorem~\ref{theorem:Monk_typeA} and Theorem~\ref{theorem_Giambelli_typeA} we saw a Monk formula and a Giambelli formula for these Peterson Schubert classes, and we saw in Section~\ref{subsec.drellich} that Drellich generalized the original work of \cite{BH} and \cite{HarTym-Monk} to all Lie types. There has been much more recent developments surrounding these issues, as we now briefly explain.

Recall that the Peterson Schubert classes $p_{v_{\mathcal{A}}}$ satisfy multiplicative relations 
\[
p_{v_{\mathcal{A}'}} \cdot p_{v_{\mathcal{A}}} = \sum_{\mathcal{B} \subset [n-1]} c^{\mathcal{B}}_{\mathcal{A}', \mathcal{A}} p_{v_{\mathcal{B}}}
\]
where the structure constants $c^{\mathcal{B}}_{\mathcal{A}', \mathcal{A}}$ lie in $H^*_S(\pt)$. Taking Theorem~\ref{theorem:Monk_typeA} further, Goldin and Gorbutt gave manifestly positive formulas for all $c^{\mathcal{B}}_{\mathcal{A}', \mathcal{A}}$ in \cite{GoldinGorbutt2020}.  From this positivity result, it is natural to ask whether there is some geometric interpretation for these classes $p_{v_{\mathcal{A}}}$. Abe, Horiguchi, Kuwata, and Zeng, give a geometric interpretation in the context of ordinary cohomology \cite{AbeHorKuwZen2021} and they also give a different formula for the structure constants by introducing a combinatorial object called left-right diagrams and defining a combinatorial game using these diagrams. In addition, recent work of Goldin, Mihalcea, and Singh gives a geometric interpretation to the equivariant Peterson Schubert classes and the Graham-positivity of their products, in arbitrary Lie type \cite{GolMihSin2021}.  Finally, follow-up work of Goldin and Singh \cite{GoldinSingh2021} derives, again in arbitrary Lie type, new explicit formulas for Monk and Chevalley rules in $H^*_{\S1}(Pet)$ using modified degree-$2$ classes that are different from those used by Drellich.

\subsection{Mixed Eulerian numbers}

The permutohedral variety (in type A) $$X_n = \Hess(\mathsf{S}, (2,3,4,\cdots,n,n))$$ is irreducible and has complex dimension $n-1$. Viewed as a subvariety of $\fln$, its associated cohomology class $[X_n]$ in $H^*(\fln)$ can be expressed as a linear combination of the Schubert classes $\{\sigma_{w_0w} \, \mid \, w \in S_n, \ell(w)=n-1\}$ as follows: 
\begin{align*}
[X_n] = \sum_{w \in S_n \atop \ell(w)=n-1} a_w \sigma_{w_0w}, \ \ \ a_w \in \Z
\end{align*}
In the above equation, the coefficient $a_w \in \Z$ is the intersection number $\displaystyle\int_{\fln}[X_n] \cdot \sigma_w$, so in particular, it is non-negative \cite[Exercise 12 in Section 10.6]{Fulton-Young}.  The question was posed in \cite{HarHorMasPar} whether the $a_w$ are in fact all \emph{positive}, not just non-negative.  Recent work of Nadeau and Tewari answered this question in the affirmative, and for arbitrary Lie type \cite{NadeauTewari}. Their method of proof deserves particular mention, since it incorporates new ideas into the study of Hessenberg varieties.  Specifically, Nadeau and Tewari's strategy is to express the coefficients $a_w$ in terms of the \textbf{mixed Eulerian numbers} introduced by Postnikov in \cite{Postnikov}.

In related work, the paper \cite{BergetSpinkTseng} of Berget, Spink, and Tseng uses the perspective of \textbf{matroids} to give a simple computation of mixed Eulerian numbers in the case of Lie type A. This suggests the natural problem of giving analogous simple computations of these numbers in arbitrary Lie type. 
A solution to this problem was given in \cite{Horiguchi-mixedEuler} by observing that the mixed Eulerian numbers in arbitrary Lie types are intimately related to the structure constants of Peterson Schubert calculus.

\subsection{Additive bases for $H^*(\Hess(\mathsf{N},h))$}\label{subsec:additive.basis}

Theorem~\ref{theorem: main reg nilpotent} gives a presentation of $H^*(\Hess(\mathsf{N},h))$, but an additional interesting question 
which arises in this context is whether there exist natural, or computationally convenient, \emph{additive bases} for $H^*(\Hess(\mathsf{N},h))$ 
which interact well with, or have natural descriptions in terms of, the known presentation $H^*(\Hess(\mathsf{N},h)) \cong R/I_h$. 
Some results in this direction are obtained in the recent papers \cite{EnoHorNagTsu-additivebasis} and \cite{HarHorMurPreTym}, as follows.

\begin{itemize}

\item We begin by describing the results of \cite{EnoHorNagTsu-additivebasis}. For $1 \leq i < j \leq n$, let $\alpha_{i,j} := \tau_i - \tau_j$ where the $\tau_i$ are the (images in $H^*(\Hess(\mathsf{N},h))$ of the) ordinary Chern classes of~\eqref{def of ch in flag}. For each $i \in [n]$, fix a choice of permutation $w^{(i)}$ of $\{i+1, i+2, \cdots,h(i)\}$. Then for each such choice of permutations $w^{(i)}$ there exists a corresponding additive basis of $H^*(\Hess(\mathsf{N},h))$ as follows: 

\begin{equation}\label{eq: monomial basis 1} 
\left\{ \prod_{i=1}^{n} \alpha_{i,w^{(i)}(h(i))} \cdot \alpha_{i,w^{(i)}(h(i)-1)} \cdots \alpha_{i,w^{(i)}(h(i)-m_i+1)} \ \middle| \ 0 \leq m_i \leq h(i)-i \ {\rm for} \ i \in [n] \right\}
\end{equation}
where we take the convention that if $m_i=0$ then $$\alpha_{i,w^{(i)}(h(i))} \cdot \alpha_{i,w^{(i)}(h(i)-1)} \cdots \alpha_{i,w^{(i)}(h(i)-m_i+1)}=1.$$  Moreover, in the special case when $w^{(i)}$ is chosen to be the identity permutation for each $i$, this additive basis has a geometric interpretation.  Let $h'$ be a Hessenberg function which is less than $h$ with respect to the partial order on $H_n$ of Definition~\ref{def:partial order}. Setting $m_i := h(i) - h'(i)$ for $1 \leq i \leq n$, it turns out that the product expression~\eqref{eq: monomial basis 1} associated to this choice of $m_i$ represents (up to a scalar multiple) the Poincar\'e dual $[\Hess(\mathsf{N},h')] \in H^*(\Hess(\mathsf{N},h))$ of $\Hess(\mathsf{N},h')$ viewed as a subvariety of $\Hess(\mathsf{N},h)$. In particular, it follows from this observation that the set of such Poincar\'e duals $\{[\Hess(\mathsf{N},h')] \, \mid \, h' \subset h\}$ in $H^*(\Hess(\mathsf{N},h))$ is a linearly independent set.  In other words, the additive basis~\eqref{eq: monomial basis 1} can be viewed as an extension of this linearly independent set $\{[\Hess(\mathsf{N},h')] \, \mid \, h' \subset h\}$ to a basis of $H^*(\Hess(\mathsf{N},h))$.   Results for some other Lie types (not just Lie type A) are also obtained in \cite{EnoHorNagTsu-additivebasis}.

\item We now describe the results of \cite{HarHorMurPreTym}.   Here, the starting point is the well-known fact from Proposition~\ref{proposition:flag generators non-equivariant} that the monomials 
\[
\{\tau_1^{m_1}\tau_2^{m_2}\cdots \tau_n^{m_n} \mid 0 \leq m_i \leq n-i \ {\rm for} \ i \in [n] \}
\]
form an additive basis of the cohomology ring of the flag variety $\fln$.
One of the main results of \cite{HarHorMurPreTym} is a natural generalization of this fact to the case of regular nilpotent Hessenberg varieties.  Specifically, the result states that the set of monomials 
\begin{equation*} 
\left\{ \tau_1^{m_1}\tau_2^{m_2}\cdots \tau_n^{m_n} \mid 0 \leq m_i \leq h(i)-i \ {\rm for} \ i \in [n] \right\}
\end{equation*}
forms an additive (monomial) basis for $H^*(\Hess(\mathsf{N},h))$.

\end{itemize} 

It should be noted that both of the manuscripts \cite{EnoHorNagTsu-additivebasis} and \cite{HarHorMurPreTym} use, in a fundamental way, the explicit presentation of $H^*(\Hess(\mathsf{N},h))$ of Theorem~\ref{theorem: main reg nilpotent}.

\subsection{Hessenberg Schubert classes and Hessenberg Schubert polynomials}

We already discussed some topics related to Schubert calculus for the Peterson variety in Sections~\ref{section: peterson in type A}. It is natural to also ask if we can extend to the setting of general regular nilpotent Hessenberg varieties $\Hess(\mathsf{N},h)$. For instance, one approach is to ask:  can we determine a subset of the Schubert classes which, when projected to $H^*(\Hess(\mathsf{N},h))$ under the surjective restriction homomorphism $H^*(\fln) \to H^*(\Hess(\mathsf{N},h))$, form a basis of $H^*(\Hess(\mathsf{N},h))$? In the discussion in \cite{HarTym-Monk}, where this question was first posed, it was suggested that a potential candidate for such a subset is 
\begin{align} \label{eq:Harada-Tymoczko_conjecture}
\{p_w \in H^*(\Hess(\mathsf{N},h)) \mid w(i) \leq h(i) \ {\rm for} \ i \in [n] \}
\end{align}
where we denote by $p_w$ the image of the Schubert class $\sigma_w$ under the restriction map $H^*(\fln) \to H^*(\Hess(\mathsf{N},h))$. 
However, as far as we are aware, this question is still open. Nevertheless, recent developments do give some computational methods for attacking the problem, as we now explain. 

Specifically, we can obtain explicit linear relations that are satisfied by the set of all images $\{p_w \, \mid \, w \in S_n\}$ of the Schubert classes, as follows. Let $h$ be a fixed Hessenberg function with $h \neq (1,2,3,\ldots,n)$. Viewing $h$ as a collection of boxes, we fix a choice of a corner box, i.e., a box corresponding to $(h(j),j)$ where $h(j-1)<h(j)$.  Suppose the corner box is at position $(i,j)$, and let $h'$ denote the Hessenberg function obtained from $h$ by removing the $(i,j)$ corner box. From Theorem~\ref{theorem: main reg nilpotent} it is not hard to see that the kernel of the surjective ring homomorphism $H^*(\Hess(\mathsf{N},h)) \rightarrow H^*(\Hess(\mathsf{N},h'))$ is the principal ideal generated by the (equivalence class of the) element $\check{f}_{i-1,j}(\tau) :=\check{f}_{i-1,j}(\tau_1, \cdots, \tau_n)$.  Now, the results of \cite{Horiguchi-Schubert} show that the polynomial $\check{f}_{i-1,j}$ can be expressed as an alternating sum of Schubert polynomials as: 
\begin{align} \label{eq:f_ijSchubert}
\check{f}_{i-1,j} = \sum_{k=1}^{i-j} (-1)^{k-1} \mathfrak{S}_{w_k^{(i,j)}}
\end{align}
where 
the permutation $w_k^{(i,j)}$ (for $1  \leq k \leq  i-j$)  is defined as 
\[
w_k^{(i,j)} := (s_{i-k}s_{i-k-1} \cdots s_j )(s_{i-k+1}s_{i-k+2} \cdots s_{i-1}). 
\]
It is interesting to note that the set $\{w_k^{(i,j)} \mid 1 \leq k \leq i-j \}$ coincides with the set of 
minimal length permutations $w$ in $S_n$ such that 
\begin{align*}
\Hess(\mathsf{N},h) \cap X_w^{\circ} \neq \emptyset \ {\rm and} \ \Hess(\mathsf{N},h') \cap X_w^{\circ} = \emptyset. 
\end{align*}

Putting these observations together with~\eqref{eq:f_ijSchubert}, we obtain the following linear relation among the $p_w$'s in $H^*(\Hess(\mathsf{N},h'))$:
\begin{align*} 
\sum_{k=1}^{i-j} (-1)^{k-1} p_{w_k^{(i,j)}}=0. 
\end{align*}
We may now naturally wonder what other linear relations hold among the $p_w$'s.   In relation to this question, we note that one of the results of \cite{HarHorMurPreTym} shows that the following is a monomial basis for the kernel of $H^*(\Hess(\mathsf{N},h)) \rightarrow H^*(\Hess(\mathsf{N},h'))$: 
\begin{align} \label{eq:Schubert_classes_in_Hess(N,h)}
\left\{ \tau_1^{m_1}\cdots \tau_{j-1}^{m_{j-1}}\cdot \check{f}_{i-1,j}(\tau) \cdot \tau_{j+1}^{m_{j+1}}\cdots \tau_n^{m_n} \mid 0 \leq m_s \leq h(s)-s \ {\rm for} \ 1 \leq s \leq n, s \neq j \right\}. 
\end{align}
We also know the Monk formula in $H^*(\fln)$ (\cite{Monk}, cf. also \cite[p.180-181]{Fulton-Young}) which states 
\begin{equation} \label{Monk_equivalent} 
\tau_r \cdot \sigma_{w}=\sum_{w'} \sigma_{w'}-\sum_{w''} \sigma_{w''}
\end{equation}
where the first sum is over those $w'$ obtained from $w$ by interchanging the values of $w$ in positions $r$ and $q$ for those $r<q$ with $w(r)<w(q)$, and $w(i)$ is not in the interval $(w(r),w(q))$ for any $i$ in the interval $(r,q)$, and the second sum is over those $w''$ obtained from $w$ by interchanging the values of $w$ in positions $r$ and $p$ for those $p<r$ with $w(p)<w(r)$, and $w(i)$ is not in the interval $(w(p),w(r))$ for any $i$ in the interval $(p,r)$.
Taking this equation down to $H^*(\Hess(\mathsf{N},h))$ via the surjective restriction map $H^*(\fln) \to H^*(\Hess(\mathsf{N},h))$ we immediately obtain the relation 
\begin{equation} \label{Monk_equivalent_Hess(N,h)} 
\tau_r \cdot p_{w}=\sum_{w'} p_{w'}-\sum_{w''} p_{w''}
\end{equation}
which we may think of as a ``Monk relation in $\Hess(\mathsf{N},h)$''. 

Now, using~\eqref{eq:f_ijSchubert}, \eqref{eq:Schubert_classes_in_Hess(N,h)}, \eqref{Monk_equivalent_Hess(N,h)}, we may obtain a complete set of relations satisfied by the $p_w$, i.e. the images of the Schubert classes. 
Stated slightly more precisely, what we mean is the following. Starting first with the cohomology ring $H^*(\fln)$, which corresponds to $h=(n,n,\cdots,n)$, we may obtain an arbitrary Hessenberg function $h$ by successively removing boxes from $h=(n,n,\cdots,n)$.  At each step of removing one (corner) box, we may use the above relations to obtain a new set of relations which must be satisfied by the image Schubert classes $p_w$.   Using this method, it is possible -- for example -- to confirm that the set of classes proposed in~\eqref{eq:Harada-Tymoczko_conjecture} do indeed form a basis of $H^*(\Hess(\mathsf{N},h))$ in the cases for which $n=4$.



\end{document}